\documentclass[reqno, 11pt]{amsart}

\usepackage[text={155mm,235mm},centering]{geometry}
\input diagxy
\input xy

\usepackage[active]{srcltx} 

\newtheorem{thm}{Theorem}[section]
\newtheorem{cor}[thm]{Corollary}
\newtheorem{lem}[thm]{Lemma}
\newtheorem{prop}[thm]{Proposition}
\newtheorem{quest}[thm]{Question}

\theoremstyle{definition}
\newtheorem{defn}[thm]{Definition}
\theoremstyle{property}

\theoremstyle{remark}
\newtheorem{rem}[thm]{Remark}

\newtheorem{main theorem}[thm]{Main Theorem}

\newtheorem{ex}[thm]{Example}
\numberwithin{equation}{section}
\usepackage[mathscr]{eucal}
\usepackage[all]{xy}
\usepackage{mathrsfs}
\usepackage{xypic}
\usepackage{amsfonts}
\usepackage{amsmath}
\usepackage{amsthm,bm}
\usepackage{amssymb,tikz-cd}
\usepackage{latexsym}
\usepackage{tabularx}
\usepackage{graphicx}
\usepackage{pict2e}
\usepackage[pagebackref]{hyperref}
\usepackage{tikz}
\usepackage{textcomp}
\usepackage[scr=rsfs]{mathalfa}
\definecolor{ceruleanblue}{rgb}{0.16, 0.32, 0.75}
\hypersetup{colorlinks=true,allcolors=ceruleanblue}
\allowdisplaybreaks

\begin{document}

\title[Blow-up formulae for twisted cohomologies with  supports]{Blow-up formulae for twisted cohomologies with supports}

\author{Lingxu Meng}
\address{Department of Mathematics, North University of China, Taiyuan, Shanxi 030051, P.R. China}
\email{menglingxu@nuc.edu.cn}%



\begin{abstract}
  We study twisted cohomologies with paracompactifying families of supports.
  The K\"{u}nneth theorems, Leray-Hirsch theorems and self-intersection formulae are established.
  Based on these results, we eventually give explicit expressions of complex blow-up formulae for twisted Dolbeault cohomology on  arbitrary complex manifolds and the ones of generalized blow-ups  formulae for twisted de Rham cohomology  on arbitrary oriented smooth manifolds.
  These expressions are  induced by the morphisms of (simple or double) complexes of spaces of  forms and currents rather than just the maps between cohomologies, which help us to obtain the corresponding results
  for twisted Bott-Chern, Aeppli cohomologies and hypercohomologies of truncated twisted holomorphic de Rham complexes.

\end{abstract}

\maketitle

\section{Introduction}
Unless stated otherwise, all manifolds are assumed to be \emph{connected}, \emph{paracompact},  all submanifolds (resp. complex submanifolds) are  assumed to be \emph{closed} (in the topological sense) \emph{embedded smooth} (resp. \emph{complex}) \emph{submanifolds   without boundary}
and set $\mathbf{k}$ $=$ $\mathbb{R}$ or $\mathbb{C}$.

The concept of blow-ups was invented by algebraic geometers in the study of birational transformations.
O. Zariski \cite{Zar} first defined it in modern language and used it to study singularities.
In complex geometry, the corresponding notion was first introduced by H. Hopf \cite{Ho}, which is said to be the \emph{complex blow-up} in the present paper.
For smooth complex algebraic varieties,   blow-ups in complex setting coincide with the ones in  algebraic setting.
Blow-up transformations play an important role in complex and algebraic geometries.
Using them, K. Kodaira \cite{K} proved the well-known embedding theorem and H. Hironaka \cite{Hi} constructed the first example of non-algebraic Moishezon threefold.
J.-P. Demailly and M. Paun \cite{DP} obtained a characterization of  the Fujiki class $\mathcal{C}$ via the complex blow-up operations.
Besides algebraic and complex settings,  blow-ups can also be defined in other geometric categories.
D. McDuff \cite{Mc} defined the symplectic blow-ups and used it to construct the examples of simply-connected non-K\"{a}hlerian symplectic manifolds.
Inspired by \cite{Mc}, S. Yang, X.-D. Yang and G. Zhao \cite{YYZ} defined
the blow-up of a locally conformally symplectic manifold along a compact induced symplectic submanifold and proved it  admits a locally conformally symplectic structure.
Generalized complex geometry unified complex geometry and symplectic geometry in one framework, which plays a significant role in string theory.
To find more examples of generalized complex manifolds, the ones have been trying to construct  blow-ups in this setting.
G. Cavalcanti and M. Gualtieri \cite{CG} showed that a blow-up exists for a generically symplectic $4$-manifold along a non-degenerate point of
complex type.
This was used to produce
new examples of generalized complex structures on $m\mathbb{C}P^2\#n\mathbb{C}P^2$ for $m$ odd.
In \cite{CG2,v2}, the condition of generalized K\"{a}hlerian blow-ups were studied.
M. Bailey, G. Cavalcanti and J. van der Leer Dur\'{a}n \cite{BCv}  introduced the concept of holomorphic ideals  to define a blow-up in the category of smooth manifolds.
They  proved that a generalized Poisson submanifold carries a canonical holomorphic ideal and gave a necessary and sufficient condition for the  blow-up of a generalized complex manifold  along a generalized Poisson submanifold  to be generalized complex.
They also proved a normal form theorem for a neighborhood of a generalized Poisson transversal and used it to define a generalized
complex blow-up of a generalized complex manifold  along a generalized Poisson transversal.
Forgetting additional structures, the underlying manifolds of blow-ups in all above settings are roughly viewed as a union of two pieces: the complement of blow-up center and the complex projectivization of normal bundle of blow-up center.
Based on this observation, we define the concept of generalized blow-ups (see Sect. 5.2.1) and investigate it.

How invariants vary under  blow-up transformations is a natural and important question.
Recently, there are some progress on the complex blow-up formulae for twisted cohomologies.
On compact complex manifolds, complex blow-up formulae were established for twisted Dolbeault cohomologies \cite{RYY,RYY2,ASTT,St1,St2}
and twisted de Rham cohomologies \cite{YZ, YYZ,  CY, Z}.
In different approaches, we \cite{M1,M2,M2.5,M3} established and explicitly expressed these formulae on arbitrary complex manifolds \emph{without the
hypothesis of compactness}.

Our first goal of the present paper is to  establish  complex blow-up formulae for twisted Dolbeault cohomology with supports in a paracompactifying family.
\begin{thm}\label{1.2}
Let $\pi:\widetilde{X}\rightarrow X$ be the complex blow-up of a complex manifold $X$ along a complex submanifold $Y$ of complex codimension $r$ with the exceptional
divisor $E$.
Suppose that  $\mathcal{E}$ is a locally free sheaf of $\mathcal{O}_X$-modules of finite rank on $X$ and $\Phi$ is a paracompactifying family of supports on
$X$.
Denote by $i_E:E\rightarrow\widetilde{X}$  the inclusion and by $h\in H^{1,1}(E)$ the Dolbeault class of a first Chern form of the universal line bundle
$\mathcal{O}_E(-1)$ over the projective bundle $\pi|_E:E=\mathbb{P}(N_{Y/X})\rightarrow Y$ associated to the normal bundle $N_{Y/X}$ of $Y$ in $X$.
Then
\begin{equation}\label{b-u-m1}
\pi^*+\sum_{i=1}^{r-1}i_{E*}\circ (h^{i-1}\cup)\circ (\pi|_E)^*
\end{equation}
gives an isomorphism
\begin{displaymath}
H_{\Phi}^{\bullet,\bullet}(X,\mathcal{E})\oplus \bigoplus_{i=1}^{r-1}H_{\Phi|_Y}^{\bullet-i,\bullet-i}(Y,i_Y^*\mathcal{E})\tilde{\rightarrow}
H_{\pi^{-1}\Phi}^{\bullet,\bullet}(\widetilde{X},\pi^*\mathcal{E})
\end{displaymath}
and
\begin{equation}\label{b-u-m2}
(\pi_*,\mbox{ }G_{h,\Phi}^{-1}\circ i_E^*,\mbox{ }...,\mbox{ }G_{h,\Phi}^{-r+1}\circ i_E^*)
\end{equation}
gives an isomorphism
\begin{equation}\label{b-u-m4}
H_{\pi^{-1}\Phi}^{\bullet,\bullet}(\widetilde{X},\pi^*\mathcal{E})\tilde{\rightarrow} H_{\Phi}^{\bullet,\bullet}(X,\mathcal{E})\oplus
\bigoplus_{i=1}^{r-1}H_{\Phi|_Y}^{\bullet-i,\bullet-i}(Y,i_Y^*\mathcal{E}),
\end{equation}
where $G_{h,\Phi}^{-i}: H_{(\pi^{-1}\Phi)|_E}^{\bullet,\bullet}(E,(\pi|_E)^*i_Y^*\mathcal{E})\rightarrow H_{\Phi|_Y}^{\bullet-i,\bullet-i}(Y,i_Y^*\mathcal{E})$
is defined in Sect. 4.3.
\end{thm}

On compact complex manifolds, S. Rao, S. Yang and X.-D. Yang \cite{RYY2} gave an expression of (\ref{b-u-m4}) for $\Phi=clt_X$ in the form (\ref{rep0}) (see Sect. 4.4.1), which is exactly  (\ref{b-u-m2}) via $\tau^{\mathcal{E}}_{Dol,clt_X}$ (see Sect. 4.4.1).
An advantage of the expressions given here is to help us understand (\ref{b-u-m4}) on the level of forms and currents rather than just on the level of
cohomologies.
The twisted Bott-Chern cohomology is a useful tool in the study  of locally conformally K\"{a}hlerian geometry \cite{OV,OVV}.
Utilizing (\ref{b-u-m1}) (\ref{b-u-m2}) and a Stelzig's result \cite{St2}, we give the  explicit complex blow-up formulae for twisted Bott-Chern, Aeppli cohomologies in Sect. 4.5.1.
Hypercohomology  of a truncated holomorphic de Rham complex is an important invariant in Hodge theory.
For instance, it is used to compute the Hodge filtration on de Rham cohomology of a compact K\"{a}hler manifold \cite{V};
it  connects with  singular cohomology and Deligne cohomology \cite{EV},  singular cohomology and integral Bott-Chern cohomology \cite{Sch}.
By use of (\ref{b-u-m1}) (\ref{b-u-m2}), we obtain the complex blow-up formula for the hypercohomology of a truncated twisted holomorphic de Rham complex on an arbitrary complex manifold,
which extends an open question of Y. Chen, S. Yang \cite{CY} and answers it in Sect. 4.5.2.

Our second goal is to establish the formulae of generalized blow-ups  for twisted de Rham cohomologies with supports in paracompactifying families  on (not necessarily compact) oriented smooth manifolds, which apply to complex, symplectic, locally conformally symplectic and generalized complex blow-ups.
\begin{thm}\label{1.3}
Let $\pi:\widetilde{X}\rightarrow X$ be a generalized blow-up of an oriented smooth manifold $X$ along an oriented smooth submanifold $Y$ of codimension $2r$ with the exceptional divisor $E$.
Assume that $\mathcal{L}$ is a  local system of $\mathbf{k}$-modules with finite rank  on $X$ and $\Phi$ is a paracompactifying family of supports on $X$.
Denote by $i_E:E\rightarrow \widetilde{X}$ the inclusion and by $h=c_1(\mathcal{O}_E(-1))\in H_{dR}^2(E,\mathbf{k})$ the first Chern form of the universal line bundle
$\mathcal{O}_E(-1)$ on the complex projectivization $\pi|_E:E=\mathbb{P}(N_{Y/X})\rightarrow Y$ associated to the normal bundle $N_{Y/X}$ of $Y$ in $X$.
Then
\begin{equation}\label{b-u-m1b}
\pi^*+\sum_{i=1}^{r-1}i_{E*}\circ (h^{i-1}\cup)\circ (\pi|_E)^*:
\end{equation}
\begin{equation}\label{b-u-m2b}
H_{\Phi}^k(X,\mathcal{L})\oplus \bigoplus_{i=1}^{r-1}H_{\Phi|_Y}^{k-2i}(Y,\mathcal{L}|_Y)\tilde{\rightarrow} H_{\pi^{-1}\Phi}^k(\widetilde{X},\pi^{-1}\mathcal{L})
\end{equation}
and
\begin{equation}\label{b-u-m3b}
(\pi_*,\mbox{ }G_{h,\Phi}^{-1}\circ i_E^*,\mbox{ }...,\mbox{ }G_{h,\Phi}^{-r+1}\circ i_E^*):
\end{equation}
\begin{equation}\label{b-u-m4b}
H_{\pi^{-1}\Phi}^{\bullet}(\widetilde{X},\pi^{-1}\mathcal{L})\rightarrow H_{\Phi}^{\bullet}(X,\mathcal{L})\oplus
\bigoplus_{i=1}^{r-1}H_{\Phi|_Y}^{\bullet-2i}(Y,\mathcal{L}|_Y),
\end{equation}
are inverse  isomorphisms, where $G_{h,\Phi}^{-i}: H_{(\pi^{-1}\Phi)|_E}^{\bullet}(E,(\pi|_E)^{-1}(\mathcal{L}|_Y))\rightarrow H_{\Phi|_Y}^{\bullet-2i}(Y,\mathcal{L}|_Y)$
is defined as in Sect. 5.1.
\end{thm}
It is worthy to notice that, self-intersection formulae (Propositions \ref{key0}, \ref{key})  play a key role for giving the expressions (\ref{b-u-m1}) (\ref{b-u-m1b}) (Proposition \ref{special}, Sect. 5.2.3) and studying the inverse relationships of  (\ref{b-u-m1}) and (\ref{b-u-m2}),  (\ref{b-u-m1b}) and (\ref{b-u-m3b}) (Proposition \ref{relation}, Theorem \ref{1.3}).

All these  formulae are established on the cohomologies with supports in quite general families,
which apply to the usual cohomologies and the cohomologies with compact supports.
They may be useful to study the topological properties of blow-up manifolds.

\subsection*{Acknowledgements}
The author is supported by Scientific and Technologial Innovation Programs of Higher Education Institutions in Shanxi (STIP), the Natural Science Foundation of
Shanxi Province of China (Grant No. 201901D111141) and the Fund for Shanxi ``1331KIRT".


\section{Preliminaries}
\subsection{Families of supports}
For the readers' convenience, we collect some terminology and results on families of supports, refer to \cite[I. \S6, II. \S9, IV. \S5]{Br}.

\emph{A  family  $\Phi$  of supports} on a topological space $X$ means a family $\Phi$ of closed subsets of $X$   satisfying that:

$(1)$ any closed subset of a member of $\Phi$ is a member of $\Phi$,

$(2)$ $\Phi$ is closed under finite unions.\\
If in addition:

$(3)$ each element of $\Phi$ is paracompact,

$(4)$ each element of $\Phi$ has a closed neighborhood which is in $\Phi$,\\
then $\Phi$ is said to be \emph{a paracompactifying  family of supports} on $X$.

Let $\Phi$  be  a family  of supports on $X$. For a subset $A\subseteq X$, set $\Phi|_A:=\{K\subseteq A|\mbox{ }K\in \Phi\}$ and $\Phi\cap A:=\{K\cap A|\mbox{ }K\in \Phi\}$, which are two families of supports on $A$.
If $B\subseteq A$, $(\Phi|_A)|_B=\Phi|_B$.
For the  families $\Phi$ and $\Psi$ of supports on $X$, $\Phi\cap\Psi$ denotes the family of all closed subsets of sets of the form $K\cap L$ for $K\in \Psi$
and $L\in \Phi$.
For the  families $\Phi$ and $\Psi$ of supports on $X$, $Y$ respectively, $\Phi\times\Psi$ means the family of all closed subsets of sets of the form $K\times
L$ for $K\in \Psi$ and $L\in \Phi$.
Let $f:Y\rightarrow X$ be a continuous map of topological spaces and let $\Phi$  be a family of supports on $X$.
Denote by $f^{-1}(\Phi)$  the family of all closed subsets of sets of the form $f^{-1}(K)$  for $K\in \Phi$ .
For a continuous map $g:Z\rightarrow Y$ of topological spaces, $g^{-1}(f^{-1}\Phi)=(f\circ g)^{-1}\Phi$.

Denote by $clt_X$ and $c_X$ the families of  all closed subsets and all compact subsets of $X$  respectively.
If $X$ is paracompact, $clt_X$ is paracompactifying.
If $X$ is locally compact, $c_X$ is paracompactifying.
Let $f:X\rightarrow Y$ be a continuous map of topological spaces. Then $f^{-1}clt_Y=clt_X$.
Moreover, if $f$ is \emph{proper}, $f^{-1}c_Y=c_X$.
Let $pr_1:X\times Y\rightarrow X$ be the first projection and $\Phi$ a family of supports on $X$.
Then $pr_1^{-1}\Phi=\Phi\times clt_Y$, which is denoted by $\Phi\times Y$.

\begin{prop}\label{inverse-paracompact}
$(1)$ Let $\Phi$ be a family of supports on a topological space $X$.
For a closed subset $A\subseteq X$, $\Phi|_A=\Phi\cap A=i^{-1}\Phi$, where $i:A\rightarrow X$ is the inclusion.

$(2)$ Let  $f:X\rightarrow Y$ be a proper map of locally compact Hausdorff spaces and $\Phi$ a family of supports of $Y$.
For any  subset $A\subseteq Y$,
$(f|_{f^{-1}(A)})^{-1}(\Phi|_A)=(f^{-1}\Phi)|_{f^{-1}(A)}$.

$(3)$ Assume that $\Phi$ and $\Psi$ are families of supports on a topological space $X$.
Then $\Phi\cap \Psi=l^{-1}(\Phi\times \Psi)$, where $l:X\rightarrow X\times X$  is the diagonal map.

$(4)$ Let $\Phi$ be a paracompactifying family of supports on $X$.
For  a locally closed subset $A\subseteq X$, $\Phi|_A$ is paracompactifying  on $A$.

$(5)$ Let $f:X\rightarrow Y$ be a continuous  map of locally compact Hausdorff spaces and $\Psi$ a paracompactifying family of supports on $Y$.
Then $f^{-1}\Psi$ is  paracompactifying  on $X$.

$(6)$ Let $\Phi$ and $\Psi$ be paracompactifying families of supports on locally compact Hausdorff spaces $X$ and $Y$ respectively.
Then $\Phi\times \Psi$ is paracompactifying on $X\times Y$.
\end{prop}
\begin{proof}
By the definitions, we immediately get $(1)$ $(3)$ $(6)$.
For $(4)$, see \cite[I. 6.5]{Br}.
By \cite[IV. 5.4 (3), 5.5]{Br}, $(5)$ holds.
For $K\in (f|_{f^{-1}(A)})^{-1}(\Phi|_A)$, $K$ is a closed subset of  $(f|_{f^{-1}(A)})^{-1}(L)=f^{-1}(L)$ for a set $L\subseteq A$  with $L\in\Phi$.
Then $K\in f^{-1}\Phi$ and $K\subseteq f^{-1}(A)$, i.e., $K\in(f^{-1}\Phi)|_{f^{-1}(A)}$.
So $(f|_{f^{-1}(A)})^{-1}(\Phi|_A)\subseteq (f^{-1}\Phi)|_{f^{-1}(A)}$.
For $K\in(f^{-1}\Phi)|_{f^{-1}(A)}$, $K\subseteq f^{-1}(A)$ and there is a set $L$  such that $L\in\Phi$ and $K$ is a closed subset of  $f^{-1}(L)$.
Set $Z=f(K)$. Then $Z\subseteq L\in \Phi$ and $Z\subseteq A$.
Since $f$ is proper, $Z$ is closed in $Y$.
So $Z\in\Phi|_A$.
Since $K$ is a closed subset of $(f|_{f^{-1}(A)})^{-1}(Z)$, $K\in (f|_{f^{-1}(A)})^{-1}(\Phi|_A)$.
We proved $(2)$.
\end{proof}

Suppose that $\mathcal{F}$ is a sheaf on $X$ and $\Phi$ is a family of supports on $X$. Denote by $\Gamma_\Phi(X,\mathcal{F})$ the group of sections of
$\mathcal{F}$ on $X$ with supports in $\Phi$ and by $H^\bullet_\Phi(X,\mathcal{F})$ the cohomology of $\mathcal{F}$ with supports in $\Phi$.
The sheaf $\mathcal{F}$ is said to be \emph{$\Phi$-acyclic}, if  $H^p_\Phi(X,\mathcal{F})=0$ for $p>0$.
The sheaf $\mathcal{F}$ is called a \emph{$\Phi$-soft} sheaf, if the restriction map $\Gamma(X,\mathcal{F})\rightarrow \Gamma(Z,\mathcal{F})$ is surjective for
any $Z\in \Phi$.
Let $\Phi$ be a paracompactifying familiy of supports on a complex manifold $X$.
The sheaf $\mathcal{C}_X^\infty$ of germs of (real or complex valued) smooth functions on $X$ is $\Phi$-soft (\cite[II. 9.4]{Br}), so are
$\mathcal{A}_X^{p,q}$ and $\mathcal{D}_X^{\prime p,q}$  for any $p$, $q$ by \cite[II. 9.16]{Br}.
Hence $\mathcal{A}_X^{p,q}$ and $\mathcal{D}_X^{\prime p,q}$  are $\Phi$-acyclic by \cite[II. 9.11]{Br}.

\subsection{Sheaf Theory}
We recall some notations and results in sheaf theory, refer to  \cite[IV. Sects. 2, 8, 9]{Dem}.
Suppose that $U$ is any open set of $X$.
For a sheaf $\mathcal{F}$ on $X$, $\Gamma(U,\mathcal{F}^{[0]})$ is the set of all maps $f:U\rightarrow \mathcal{F}$ such that $f(x)\in\mathcal{F}_x$ for all $x\in U$.
For any $s\in\Gamma(U,\mathcal{F})$, let $\tilde{s}(x)=s_x$ for any $x\in U$, where $s_x\in\mathcal{F}_x$ is the stalk of $s$ over $x$.
Then $s\mapsto \tilde{s}$ define a natural injection $j:\mathcal{F}\rightarrow \mathcal{F}^{[0]}$ of sheaves.
Set $\mathcal{F}^{[p]}=(\mathcal{F}^{[p-1]})^{[0]}$ for  $p\geq 0$.
For all $p$, $\mathcal{F}^{[p]}$ are flabby sheaves.
The stalk $\mathcal{F}_x^{[p]}$ can be considered as the set of equivalence classes of maps $f:X^{p+1}\rightarrow \mathcal{F}$ such that $f(x_0,\ldots,x_p)\in
\mathcal{F}_{x_p}$ for any $(x_0,\ldots,x_p)\in X^{p+1}$, with such two maps are equivalent if and only if they coincide on a set of the form
\begin{displaymath}
x_0\in V,\mbox{ } x_1\in V(x_0),\mbox{ \ldots, } x_p\in V(x_0,\ldots,x_{p-1}),
\end{displaymath}
where $V$ is an open neighborhood of $x$ and $V(x_0,\ldots,x_j)$ an open neighborhood of $x_j$, depending on  $x_0,\ldots,x_j$.
Similarly, $\Gamma(U,\mathcal{F}^{[p]})$ can be considered as the set of equivalence classes of maps $f:X^{p+1}\rightarrow \mathcal{F}$ such that $f(x_0,\ldots,x_p)\in
\mathcal{F}_{x_p}$ for any $(x_0,\ldots,x_p)\in X^{p+1}$, with such two maps are equivalent if and only if they coincide on a set of the form
\begin{displaymath}
x_0\in U,\mbox{ } x_1\in V(x_0),\mbox{ \ldots, } x_p\in V(x_0,\ldots,x_{p-1}).
\end{displaymath}
For $f\in \Gamma(U,\mathcal{F}^{[p]})$, denote by $\textrm{supp}f$ the support of $f$.
Then $U-\textrm{supp}f$ is just the set of the point $x\in U$ satisfying that there exists an open neighborhood $V\subseteq U$ of $x$ such that $f:X^p\rightarrow \mathcal{F}$ is zero on a set of the form
\begin{displaymath}
x_0\in V, \mbox{ } x_1\in V(x_0),\mbox{ \ldots, } x_p\in V(x_0,\ldots,x_{p-1}).
\end{displaymath}

For $t\in \mathcal{F}_x$, let $S(t)$ be any element in $\mathcal{F}^{[0]}(X)$ whose restriction on some neighborhood $U$ of $x$ is in $\Gamma(U,\mathcal{F})$ and $S(t)(x)=t$.
Define  $d^p:\mathcal{F}^{[p]}\rightarrow \mathcal{F}^{[p+1]}$ as
\begin{displaymath}
\small{
\begin{aligned}
(d^pf)(x_0,\ldots,x_{p+1})=\sum\limits_{0\leq i\leq p}(-1)^if(x_0,\ldots,\hat{x}_i,\ldots x_{p+1})+(-1)^{p+1}S(f(x_0,\ldots,x_{p}))(x_{p+1}).
\end{aligned}}
\end{displaymath}
The definition is independent of the choice of $S$.
Moreover, $d^0\circ j=0$ and $d^{p+1}\circ d^p=0$ for any $p\geq 0$.
The complex $(\mathcal{F}^{[\bullet]},d)$ is a flabby resolution of the sheaf $\mathcal{F}$, which is called the \emph{simplicial flabby resolution} of $\mathcal{F}$ and is briefly denoted  as $\mathcal{F}^{[\bullet]}$.

For any exact sequence $\mathcal{F}\rightarrow \mathcal{G}\rightarrow \mathcal{H}$,
their resolutions $\mathcal{F}^{[\bullet]}\rightarrow
\mathcal{G}^{[\bullet]}\rightarrow \mathcal{H}^{[\bullet]}$ is exact.
So $\mathcal{F}\mapsto \mathcal{F}^{[p]}$ is an exact functor for any $p$.
Any flabby sheaf on $X$ is $\Phi$-acyclic for any family $\Phi$  of supports on $X$,
so the functor $\mathcal{F}\mapsto \Gamma_\Phi(X,\mathcal{F}^{[p]})$ is exact for any $p$.
In particular,
\begin{equation}\label{exact-functor}
H^q\left(\Gamma_\Phi(X,(\mathcal{F}^\bullet)^{[p]})\right)=\Gamma_\Phi\left(X,(\mathcal{H}^q(\mathcal{F}^\bullet))^{[p]})\right)
\end{equation}
for a complex $\mathcal{F}^\bullet$ of sheaves on $X$,
where  $\mathcal{H}^q(\mathcal{F}^\bullet)$ denotes the $q$-th cohomological sheaf of $\mathcal{F}^\bullet$.

Let $\mathcal{F}$, $\mathcal{G}$ be sheaves of $\mathcal{O}_X$-modules (resp. $\mathbf{k}$-modules) on a complex manifold (resp. topological space) $X$ and let $\Phi$, $\Psi$ be two families of supports on $X$.
For $u\in\Gamma(X,\mathcal{F}^{[p]})$ and $v\in\Gamma(X,\mathcal{G}^{[q]})$, the cup product $u\cup v\in \Gamma(X,(\mathcal{F}\otimes_{\mathcal{O}_X} \mathcal{G})^{[p+q]})$ (resp.  $\Gamma(X,(\mathcal{F}\otimes_{\underline{\mathbf{k}}_X} \mathcal{G})^{[p+q]})$) is defined as
\begin{displaymath}
\begin{aligned}
&u\cup v(x_0,\ldots,x_{p+q})=S(u(x_0,\ldots,x_q))(x_{p+q})\otimes v(x_p,\ldots,x_{p+q})\\
&\in \mathcal{F}_{x_{p+q}}\otimes_{\mathcal{O}_{X,x_{p+q}}} \mathcal{G}_{x_{p+q}} \mbox{ (resp. $\mathcal{F}_{x_{p+q}}\otimes_{\mathbf{k}} \mathcal{G}_{x_{p+q}}$)}
\end{aligned}
\end{displaymath}
In such way, we get a $\mathbb{C}$-bilinear (resp. $\mathbf{k}$-bilinear) map
\begin{displaymath}
\Gamma_\Phi(X,\mathcal{F}^{[p]})\times \Gamma_{\Psi}(X,\mathcal{G}^{[q]})\rightarrow \Gamma_{\Phi\cap\Psi}(X,(\mathcal{F}\otimes_{\mathcal{O}_X}
\mathcal{G})^{[p+q]})
\end{displaymath}
\begin{displaymath}
(\mbox{resp. }\Gamma_\Phi(X,\mathcal{F}^{[p]})\times \Gamma_{\Psi}(X,\mathcal{G}^{[q]})\rightarrow \Gamma_{\Phi\cap\Psi}(X,(\mathcal{F}\otimes_{\underline{\mathbf{k}}_X}
\mathcal{G})^{[p+q]})\mbox{ }),
\end{displaymath}
which maps $(u,v)$ to $u\cup v$.
It induces a cup product
\begin{equation}\label{cup1}
\cup:H^p_\Phi(X,\mathcal{F})\times H^q_{\Psi}(X,\mathcal{G})\rightarrow H^{p+q}_{\Phi\cap\Psi}(X,\mathcal{F}\otimes_{\mathcal{O}_X} \mathcal{G})
\end{equation}
\begin{displaymath}
(\mbox{resp. }H^p_\Phi(X,\mathcal{F})\times H^q_{\Psi}(X,\mathcal{G})\rightarrow H^{p+q}_{\Phi\cap\Psi}(X,\mathcal{F}\otimes_{\underline{\mathbf{k}}_X} \mathcal{G})\mbox{ }).
\end{displaymath}
The wedge product of holomorphic forms gives an embedding $\wedge:\Omega_X^r\otimes_{\mathcal{O}_X}\Omega_X^s\hookrightarrow \Omega_X^{r+s}$ for $r+s\leq \textrm{dim}_{\mathbb{C}}X$.
Via this embedding, we furthermore define a \emph{second type of cup product}
\begin{equation}\label{cup2}
\begin{aligned}
\cup:&H^p_\Phi(X,\mathcal{F}\otimes_{\mathcal{O}_X}\Omega_X^r)\times H^q_{\Psi}(X,\mathcal{G}\otimes_{\mathcal{O}_X}\Omega_X^s)\rightarrow H^{p+q}_{\Phi\cap\Psi}(X,\mathcal{F}\otimes_{\mathcal{O}_X} \mathcal{G}\otimes_{\mathcal{O}_X}\Omega_X^r\otimes_{\mathcal{O}_X}\Omega_X^s)\\
&\hookrightarrow H^{p+q}_{\Phi\cap\Psi}(X,\mathcal{F}\otimes_{\mathcal{O}_X} \mathcal{G}\otimes_{\mathcal{O}_X}\Omega_X^{r+s}).
\end{aligned}
\end{equation}

Let $f:Y\rightarrow X$ be a holomorphic (resp. continuous) map of complex manifolds (resp. topological spaces) and $\mathcal{F}$ a sheaf of $\mathcal{O}_Y$-modules (resp. $\mathbf{k}$-modules) on $Y$.
For any $u\in \Gamma(X,\mathcal{F}^{[p]})$, define $f^*u\in \Gamma(Y,(f^{*}\mathcal{F})^{[p]})$ (resp. $\Gamma(Y,(f^{-1}\mathcal{F})^{[p]})$) as
\begin{displaymath}
(f^*u)(y_0,\ldots,y_p)=u(f(y_0),\ldots,f(y_p))\otimes 1 \in (f^*\mathcal{F})_{y_p}=\mathcal{F}_{f(y_p)}\otimes_{\mathcal{O}_{X,f(y_p)}} \mathcal{O}_{Y,y_p}
\end{displaymath}
\begin{displaymath}
(\mbox{resp. }(f^*u)(y_0,\ldots,y_p)=u(f(y_0),\ldots,f(y_p)) \in (f^{-1}\mathcal{F})_{y_p}=\mathcal{F}_{f(y_p)}\mbox{ }).
\end{displaymath}
Clearly, $\textrm{supp}(f^*u)\subseteq f^{-1}(\textrm{supp}u)$.
Suppose that $\Phi$ is a family of supports on $X$.
We get a morphism  $f^*:\Gamma_\Phi(X,\mathcal{F}^{[\bullet]})\rightarrow \Gamma_{f^{-1}\Phi}(Y,(f^*\mathcal{F})^{[\bullet]})$
(resp. $f^*:\Gamma_\Phi(X,\mathcal{F}^{[\bullet]})\rightarrow \Gamma_{f^{-1}\Phi}(Y,(f^{-1}\mathcal{F})^{[\bullet]})$) of complexes,
which induces a pullback
\begin{displaymath}
f^*:H_\Phi^p(X,\mathcal{F})\rightarrow H_{f^{-1}\Phi}^p(Y,f^*\mathcal{F}).
\end{displaymath}
\begin{displaymath}
(\mbox{resp. }f^*:H_\Phi^p(X,\mathcal{F})\rightarrow H_{f^{-1}\Phi}^p(Y,f^{-1}\mathcal{F})\mbox{ }).
\end{displaymath}
The pullback of holomorphic forms gives a morphism $f^*\Omega_X^q\rightarrow \Omega_Y^{q}$, hence we can define a \emph{second type of pullback}
\begin{equation}\label{pullback2}
\begin{aligned}
f^*:&H^p_\Phi(X,\mathcal{F}\otimes_{\mathcal{O}_X}\Omega_X^q)\rightarrow H^{p}_{f^{-1}\Phi}(Y,f^*\mathcal{F}\otimes_{\mathcal{O}_Y} f^*\Omega_X^q)\\
&\rightarrow H^{p}_{f^{-1}\Phi}(Y,f^*\mathcal{F}\otimes_{\mathcal{O}_Y} \Omega_Y^{q}).
\end{aligned}
\end{equation}

\begin{lem}\label{simple}
Suppose that $X$ is a topological space.

$(1)$ Let $j:U\rightarrow X$  be an inclusion of an open subset.
Assume that $\mathcal{F}$ is a sheaf on $X$ and $\mathcal{G}$ is a sheaf on $U$.
Then $(j^{-1}\mathcal{F})^{[p]}=j^{-1}(\mathcal{F}^{[p]})$ and $(j_{!}\mathcal{G})^{[p]}=j_{!}(\mathcal{G}^{[p]})$ for any $p$.

$(2)$ Let $i:Z\rightarrow X$ be an inclusion of a closed subset and $\mathcal{H}$  a sheaf on $Z$.
Then $(i_{*}\mathcal{H})^{[p]}=i_{*}(\mathcal{H}^{[p]})$ for any $p$.
\end{lem}
\begin{proof}
We only prove $(j_{!}\mathcal{G})^{[p]}=j_{!}(\mathcal{G}^{[p]})$ and the other two conclusions can be obtained similarly.
For any open set $W\subseteq X$,
\begin{displaymath}
\Gamma(W,j_!(\mathcal{G}^{[0]}))=\{t\in\Gamma(W\cap U, \mathcal{G}^{[0]})|\mbox{ }\textrm{supp}t \mbox{ } \textrm{is}\mbox{ } \textrm{closed}\mbox{ } \textrm{in} \mbox{ } W\}.
\end{displaymath}
For any $s\in\Gamma(W,(j_!\mathcal{G})^{[0]})$, $s(x)=0$ for any $x\in W-U$, so $\textrm{supp}s\subseteq W\cap U$.
Set $\tilde{s}=s|_{W\cap U}$.
Then $\textrm{supp}\tilde{s}=\textrm{supp}s$ is closed in $W$.
Hence $s\mapsto \tilde{s}$ for all $s\in \Gamma(W,(j_!\mathcal{G})^{[0]})$ give a morphism $\Gamma(W,(j_!\mathcal{G})^{[0]})\rightarrow \Gamma(W,j_!(\mathcal{G}^{[0]}))$.
For $t\in \Gamma(W,j_!(\mathcal{G}^{[0]}))$, set $\bar{t}=t(x)$ for $x\in W\cap U$ and $0$ for $x\in W-U$.
Then $t\mapsto \bar{t}$ for all $t\in \Gamma(W,j_!(\mathcal{G}^{[0]}))$ give a morphism $\Gamma(W,j_!(\mathcal{G}^{[0]}))\rightarrow \Gamma(W,(j_!\mathcal{G})^{[0]})$.
We easily see that the two morphisms are inverse to each other.
So $(j_{!}\mathcal{G})^{[0]}=j_{!}(\mathcal{G}^{[0]})$.
By the induction, we complete the proof.
\end{proof}

Assume that $X$ is a topological space and $\mathcal{F}$ is a sheaf on $X$.
Let $Z$ be a closed subset of $X$ and set $U=X-Z$.
Denote by $i:Z\rightarrow X$ and $j:U\rightarrow X$ the inclusions.
Suppose that $\Phi$ is a family of supports on $X$.
As we know, $0\rightarrow j_!j^{-1}\mathcal{F}\rightarrow\mathcal{F}\rightarrow i_*i^{-1}\mathcal{F}\rightarrow 0$ is exact,
so is $0\rightarrow\Gamma_\Phi(X,(j_!j^{-1}\mathcal{F})^{[p]})\rightarrow\Gamma_\Phi(X,\mathcal{F}^{[p]})\rightarrow \Gamma_\Phi(X,(i_*i^{-1}\mathcal{F})^{[p]})\rightarrow 0$.
By Lemma \ref{simple},  $\Gamma_\Phi(X,(i_*i^{-1}\mathcal{F})^{[p]})=\Gamma_{\Phi|_Z}(Z,(i^{-1}\mathcal{F})^{[p]})$ and
$\Gamma_\Phi(X,(j_!j^{-1}\mathcal{F})^{[p]})=\Gamma_{\Phi|_U}(U,\mathcal{F}^{[p]})$.
So we have the short exact
\begin{displaymath}
\xymatrix{
0 \ar[r]& \Gamma_{\Phi|_U}(U,\mathcal{F}^{[p]})      \ar[r]^{\mbox{ }\mbox{ } j_*}& \Gamma_\Phi(X,\mathcal{F}^{[p]})     \ar[r]^{i^*\quad } & \Gamma_{\Phi|_Z}(Z,(i^{-1}\mathcal{F})^{[p]})     \ar[r]& 0}
\end{displaymath}
for any $p$, where $j_*$ is  the \emph{extension by zero} and $i^*$ is the pullback.
Furthermore, we have

\begin{lem}\label{comm-long}
Let $f:Y\rightarrow X$ be a continuous map of topological spaces.
Suppose that $\mathcal{F}$ a sheaf on $X$ and $\Phi$ is a family of supports on $X$.
Put $A$ a closed subset of $X$ and set $B=f^{-1}(A)$.
Denote by $j:X-A\rightarrow X$, $\tilde{j}:Y-B\rightarrow Y$, $i:A\rightarrow X$, $\tilde{i}:B\rightarrow Y$ the inclusions.
Then there exists a commutative diagram
\begin{displaymath}
\tiny{\xymatrix{
 \cdots H_{\Phi|_{X-A}}^k(X-A,\mathcal{F})\ar[d]^{f_{X-A}^*} \ar[r]^{\qquad\quad j_*}& H_{\Phi}^k(X,\mathcal{F}) \ar[d]^{f^*} \ar[r]^{i^*\quad}& H_{\Phi|_A}^k(A,i_A^{-1}\mathcal{F}) \ar[d]^{f_A^*}\ar[r]& H_{\Phi|_{X-A}}^{k+1}(X-A,\mathcal{F})\ar[d]^{f_{X-A}^*}\cdots\\
 \cdots H_{(f^{-1}\Phi)|_{Y-B}}^k(Y-B,f^{-1}\mathcal{F})       \ar[r]^{\qquad\quad \tilde{j}_*}& H_{f^{-1}\Phi}^k(Y,f^{-1}\mathcal{F})     \ar[r]^{\tilde{i}^*\quad } & H_{(f^{-1}\Phi)|_B}^k(B,i_B^{-1}f^{-1}\mathcal{F})     \ar[r]& H_{(f^{-1}\Phi)|_{Y-B}}^{k+1}(Y-B,f^{-1}\mathcal{F})         \cdots }}
\end{displaymath}
of long exact sequences, where $f_Z:f^{-1}(Z)\rightarrow Z$ denotes the restriction of $f$ for any $Z\subseteq X$.
\end{lem}
\begin{proof}
We easily check the commutative diagram
\begin{displaymath}
\tiny{\xymatrix{
 0\ar[r]& \Gamma_{\Phi|_{X-A}}(X-A,\mathcal{F}^{[\bullet]})\ar[d]^{f_{X-A}^*} \ar[r]^{\quad \qquad j_*}& \Gamma_\Phi(X,\mathcal{F}^{[\bullet]}) \ar[d]^{f^*} \ar[r]^{i^*\qquad }& \Gamma_{\Phi|_A}(A,(i^{-1}\mathcal{F})^{[\bullet]}) \ar[d]^{f_A^*}\ar[r]& 0\\
 0\ar[r]& \Gamma_{(f^{-1}\Phi)|_{Y-B}}(Y-B,(f^{-1}\mathcal{F})^{[\bullet]})       \ar[r]^{\qquad\quad  \tilde{j}_*}& \Gamma_{f^{-1}\Phi}(Y,(f^{-1}\mathcal{F})^{[\bullet]})     \ar[r]^{\tilde{i}^*\qquad } & \Gamma_{(f^{-1}\Phi)|_{B}}(B,(\tilde{i}^{-1}f^{-1}\mathcal{F})^{[\bullet]}     \ar[r]& 0         }}
\end{displaymath}
of short exact sequences of complexes, which implies the conclusion.
\end{proof}

\begin{lem}\label{higher-direct-image}
Assume that $\pi:X\rightarrow Y$ is a continuous map of topological spaces and  $\mathcal{F}$ is a sheaf of  on  $X$.
Let  $\mathfrak{H}$ be the group consisting of the map $h:Y^{p+1}\times X^{q+1}\rightarrow \mathcal{F}$ satisfying
$h(y_0,\ldots,y_p; x_0,\ldots,x_{q})\in\mathcal{F}_{x_q}$
for any $(y_0,\ldots,y_p; x_0,\ldots,x_q)\in Y^{p+1}\times X^{q+1}$
and let  $\mathfrak{H}_0$ be the subgroup of $\mathfrak{H}$ consisting of the map $h$ satisfying that $h=0$ on a set of the form
\begin{equation}\label{set}
\begin{aligned}
&y_0\in Y, \mbox{ }y_i\in V(y_0,\ldots,y_{i-1}), \mbox{ }1\leq i\leq p,\\
&\mbox{ }x_0\in \pi^{-1}(V(y_0,\ldots,y_p)),\\
& \mbox{ }x_j\in V(y_0,\ldots,y_p; x_0,\ldots,x_{j-1}), 1\leq j\leq q,
\end{aligned}
\end{equation}
where $V(y_0,\ldots,y_{i-1})$ is an open neighborhood of $y_i$ in $Y$ depending on $y_0,\ldots,y_{i-1}$ for $1\leq i\leq p$
and  $V(y_0,\ldots,y_p; x_0,\ldots,x_{j-1})$ is an open neighborhood of $x_j$ in $X$ depending on  $y_0,\ldots,y_p, x_0,\ldots,x_{j-1}$ for $1\leq j\leq q$.
Then $\mathfrak{H}/\mathfrak{H}_0$ can be viewed as a subgroup of $\Gamma\left(Y,(\pi_*(\mathcal{F}^{[q]}))^{[p]}\right)$.
\end{lem}
\begin{proof}
Let $\mathfrak{P}$ be the group consisting of the map $f:Y^{p+1}\rightarrow \pi_*(\mathcal{F}^{[q]})$
satisfying that $f(y_0,\ldots,y_p)\in(\pi_*(\mathcal{F}^{[q]}))_{y_p}$ for any $(y_0,\ldots,y_p)\in Y^{p+1}$.
Let $\mathfrak{P}_0$ be the subgroup of $\mathfrak{P}$ consisting of the map $f$ which is zero
on a set of the form $y_0\in Y$, $y_1\in V(y_0)$, \ldots, $y_p\in V(y_0,\ldots,y_{p-1})$.
Then $\Gamma\left(Y,(\pi_*(\mathcal{F}^{[q]}))^{[p]}\right)\cong \mathfrak{P}/\mathfrak{P}_0$.
For any $h\in\mathfrak{H}$, set
\begin{displaymath}
G(h)(y_0,\ldots,y_p)=h(y_0,\ldots,y_p;\bullet): X^{q+1}\rightarrow \mathcal{F}
\end{displaymath}
for any $(y_0,\ldots,y_p)\in Y^{p+1}$.
Then $G(h)(y_0,\ldots,y_p)$ can be viewed as a section of $\mathcal{F}^{[q]}$ on $X$.
Set
\begin{displaymath}
\small{
\begin{aligned}
P(h)(y_0,\ldots,y_p)=[G(h)(y_0,\ldots,y_p)]_{y_p}\in
\lim\limits_{\overrightarrow{V\ni y_p}}\Gamma(\pi^{-1}(V),\mathcal{F}^{[q]})
=\left(\pi_*(\mathcal{F}^{[q]})\right)_{y_p},
\end{aligned}}
\end{displaymath}
where $[\bullet]_{y_p}$ denotes the equivalent class under the direct limit.
Then $P(h)\in \mathfrak{P}$.
Define $\mathfrak{H}\rightarrow\mathfrak{P}/\mathfrak{P}_0$ as $h\mapsto P(h) \mbox{ modulo } \mathfrak{P}_0$.
We only need to prove that the kernel of this morphism is $\mathfrak{H}_0$.

Assume that  $h\in \mathfrak{H}$ satisfies that $P(h)\in \mathfrak{P}_0$,
i.e., $[G(h)(y_0,\ldots,y_p)]_{y_p}=0$ on the set of the form $y_0\in Y$,  $y_1\in V(y_0)$, \ldots, $y_p\in V(y_0,\ldots,y_{p-1})$.
This is equivalent to say that, there exists an open neighborhood $V$ of $y_p$ such that $G(h)(y_0,\ldots,y_p)|_{\pi^{-1}(V)}=0$ on the set of the form $y_0\in Y$, $y_i\in V(y_0,\ldots,y_{i-1})$, $1\leq i\leq p$.
We can write $V=V(y_0,\ldots,y_p)$, since $y_p\in V(y_0,\ldots,y_{p-1})$ and $V$ also depends on $y_p$.
Hence  $h(y_0,\ldots,y_p; x_0,\ldots,x_{q})=0$ on a set of the form (\ref{set}), i.e., $h\in\mathfrak{H}_0$.
Inversely, if $h\in\mathfrak{H}_0$, $P(h)\in \mathfrak{P}_0$ from above arguments.
We complete the proof.
\end{proof}

\subsection{Gluing principle}
Recall the \emph{gluing principle}, which will be used in the following part.
\begin{lem}[{\cite[Lemma 2.1]{M2}}]\label{glued}
Let $X$ be a smooth manifold and denote by $\mathcal{P}(U)$ a statement on any open subset $U$ in $X$. Assume that $\mathcal{P}$ satisfies  conditions:

$(i)$ \emph{(local condition)} There exists a basis $\mathfrak{U}$ of the topology of $X$ such that $\mathcal{P}(U_1\cap\ldots\cap U_l)$ holds for any finite
many $U_1$, $\ldots$, $U_l\in \mathfrak{U}$.

$(ii)$ \emph{(disjoint condition)} Let $\{U_n|n\in\mathbb{N}^+\}$ be any collection of disjoint open subsets of $X$. If $\mathcal{P}(U_n)$ hold for all
$n\in\mathbb{N}^+$, $\mathcal{P}(\bigcup\limits_{n=1}^\infty U_n)$ holds.

$(iii)$ \emph{(Mayer-Vietoris condition)} For open subsets $U$, $V$ of $X$, $\mathcal{P}(U\cup V)$ holds if $\mathcal{P}(U)$, $\mathcal{P}(V)$ and
$\mathcal{P}(U\cap V)$ hold.\\
Then $\mathcal{P}(X)$ holds.
\end{lem}

\section{Twisted forms and currents with supports}
\subsection{Locally free sheaves on complex manifolds}
Let $X$ be a complex manifold and $\mathcal{E}$ a locally free sheaf of $\mathcal{O}_X$-modules of rank $m$ on $X$.
An open subset $U$ of $X$ is said to be \emph{$\mathcal{E}$-free}, if the restriction $\mathcal{E}|_U$ is a free sheaf of $\mathcal{O}_U$-modules.
An open covering $\mathfrak{U}$ of $X$ is said to be \emph{$\mathcal{E}$-free}, if all $U\in\mathfrak{U}$ are $\mathcal{E}$-free.
An open covering of $X$ is called an \emph{$\mathcal{E}$-free} \emph{basis}, if it is both a basis of the topology  and an $\mathcal{E}$-free covering of $X$.
For an open set $U\subseteq X$, the elements of $\Gamma(U, \mathcal{E}\otimes_{\mathcal{O}_X}\mathcal{A}_X^{p,q})$ and $\Gamma(U,
\mathcal{E}\otimes_{\mathcal{O}_X}\mathcal{D}_X^{\prime p,q})$ are called \emph{$\mathcal{E}$-vlaued $(p,q)$-forms and currents on $U$}, respectively.
In  Sects. 3 and 4, the tensor $\otimes_{\mathcal{O}_X}$ of sheaves of $\mathcal{O}_X$-modules will be simply denoted by $\otimes$.

\subsubsection{\emph{\textbf{Local representations}}}
Let $U$ be an $\mathcal{E}$-free open subset of $X$ and  $e_1$, $\ldots$, $e_m$  a basis of $\Gamma(U,\mathcal{E})$ as an $\mathcal{O}_X(U)$-module.
For $\omega\in\Gamma(X, \mathcal{E}\otimes\mathcal{A}_X^{p,q})$, the restriction $\omega|_U$ to $U$ can be written as $\sum\limits_{i=1}^{m}e_i\otimes
\alpha_i$, where $\alpha_1$, $\ldots$, $\alpha_m\in \mathcal{A}^{p,q}(U)$.
Similarly, for $S\in\Gamma(X, \mathcal{E}\otimes\mathcal{D}_X^{\prime p,q})$, $S|_U=\sum\limits_{i=1}^{m}e_i\otimes T_i$ for some $T_1$, $\ldots$, $T_m\in
\mathcal{D}^{\prime p,q}(U)$.
We easily get
\begin{lem}\label{support}
For any $1\leq i\leq m$, $\emph{supp}\alpha_i\subseteq\emph{supp}\omega\cap U$ and $\emph{supp}T_i\subseteq\emph{supp}S\cap U$.
\end{lem}
\subsubsection{\emph{\textbf{Extensions by zero and restrictions}}}
Let $\Phi$ be a paracompactifying family  of supports on $X$.
Assume that $j:V\rightarrow X$ is the inclusion of the open subset $V$ into $X$.
Let $\mathfrak{U}$ be an $\mathcal{E}$-free covering of $X$ and $e^U_1$, $\ldots$, $e^U_m$  a basis of $\Gamma(U,\mathcal{E})$ as an $\mathcal{O}_X(U)$-module
for any $U\in\mathfrak{U}$.

For $\omega\in \Gamma_{\Phi|_V}(V,\mathcal{E}\otimes\mathcal{A}_X^{p,q})$ and $U\in\mathfrak{U}$, the restriction $\omega|_{V\cap U}$ to $U\cap V$ is
$\sum\limits_{i=1}^me_i^U|_{V\cap U}\otimes \alpha_i$, where $\alpha_1$, $\ldots$, $\alpha_m\in \mathcal{A}^{p,q}(V\cap U)$. Clearly, $\alpha_i=0$ on $(V\cap
U)\cap (U-\textrm{supp}\alpha_i)=V\cap U-\textrm{supp}\alpha_i$. So $\alpha_i$ can be extended on $(V\cap U)\cup (U-\textrm{supp}\alpha_i)=U$ by zero, denoted
by $\tilde{\alpha}_i$.  Set
\begin{displaymath}
\widetilde{\omega}_U=\sum_{i=1}^me^U_i\otimes \tilde{\alpha}_i
\end{displaymath}
in $\Gamma(U,\mathcal{E}\otimes\mathcal{A}_X^{p,q})$.
Then $\{\widetilde{\omega}_U|U\in\mathfrak{U}\}$ can be glued as a global section of $\mathcal{E}\otimes\mathcal{A}_X^{p,q}$ on $X$, denoted by $j_*\omega$.
It is noteworthy that $j_*\omega$ doesn't depend on the choice of the $\mathcal{E}$-free open covering $\mathfrak{U}$.
Since $\textrm{supp}j_*\omega=\textrm{supp}\omega\in\Phi$, we get a map
\begin{displaymath}
j_*:\Gamma_{\Phi|_V}(V,\mathcal{E}\otimes\mathcal{A}_X^{p,q})\rightarrow\Gamma_\Phi(X,\mathcal{E}\otimes\mathcal{A}_X^{p,q}).
\end{displaymath}
Similarly, we can define
\begin{equation}\label{extension2}
j_*:\Gamma_{\Phi|_V}(V,\mathcal{E}\otimes\mathcal{D}_X^{\prime p,q})\rightarrow\Gamma_\Phi(X,\mathcal{E}\otimes\mathcal{D}_X^{\prime p,q}).
\end{equation}

Let $j^*:\Gamma(X,\mathcal{E}\otimes\mathcal{D}_X^{\prime p,q})\rightarrow\Gamma(V,\mathcal{E}\otimes\mathcal{D}_X^{\prime p,q})$ be the \emph{restriction} of
the sheaf $\mathcal{E}\otimes\mathcal{D}_X^{\prime p,q}$.
For any $S\in \Gamma(X,\mathcal{E}\otimes\mathcal{D}_X^{\prime p,q})$ and $U\in\mathfrak{U}$, if the restriction $S|_U=\sum\limits_{i=1}^{m}e_i^U\otimes T_i$
for some $T_1$, $\ldots$, $T_m\in \mathcal{D}^{\prime p,q}(U)$, then $(j^*S)|_{V\cap U}=\sum\limits_{i=1}^{m}e^U_i|_{V\cap U}\otimes T_i|_{V\cap U}$.

\subsubsection{\emph{\textbf{Pullbacks and pushforwards}}}
Let $f:Y\rightarrow X$ be a holomorphic map of  complex manifolds and $r=\textrm{dim}_{\mathbb{C}}Y-\textrm{dim}_{\mathbb{C}}X$. Put
$f^*\mathcal{E}=f^{-1}\mathcal{E}\otimes_{f^{-1}\mathcal{O}_X}\mathcal{O}_Y$ the inverse image of $\mathcal{E}$ by $f$. The adjunction morphism
$\mathcal{E}\rightarrow f_*f^*\mathcal{E}$ induces $f_U^*:\Gamma(U,\mathcal{E})\rightarrow \Gamma(f^{-1}(U),f^*\mathcal{E})$ for any open set $U\subseteq X$,
where $f_U:f^{-1}(U)\rightarrow U$ is the restriction of $f$ to $f^{-1}(U)$.

\textbf{Pullbacks.}
The pullback $\mathcal{A}_X^{p,q}\rightarrow f_*\mathcal{A}_Y^{p,q}$ induces a morphism of sheaves
\begin{displaymath}
\mathcal{E}\otimes\mathcal{A}_X^{p,q}\rightarrow \mathcal{E}\otimes f_*\mathcal{A}_Y^{p,q}=f_*(f^*\mathcal{E}\otimes\mathcal{A}_Y^{p,q}),
\end{displaymath}
hence induces a \emph{pullback} of $\mathcal{E}$-valued $(p,q)$-forms
\begin{equation}\label{pullback0}
f^*:\Gamma(X,\mathcal{E}\otimes\mathcal{A}_X^{p,q})\rightarrow\Gamma(Y,f^*\mathcal{E}\otimes\mathcal{A}_Y^{p,q}).
\end{equation}
Suppose that  $U$  is an $\mathcal{E}$-free  open set in $X$ and $e_1$, $\ldots$, $e_m$ is  a basis of $\Gamma(U,\mathcal{E})$ as an $\mathcal{O}_X(U)$-module.
Obviously, $f_U^*e_1$, $\ldots$, $f_U^*e_m$ is  a basis of $\Gamma(f^{-1}(U),f^*\mathcal{E})$ as an $\mathcal{O}_Y(f^{-1}(U))$-module.
For an $\mathcal{E}$-valued $(p,q)$-form $\omega$, set $\omega|_U=\sum\limits_{i=1}^{m}e_i\otimes \alpha_i$ for some $\alpha_1$, $\ldots$, $\alpha_m\in
\mathcal{A}^{p,q}(U)$.
Then
\begin{equation}\label{pullback-rep}
(f^*\omega)|_{f^{-1}(U)}=\sum_{i=1}^{m}f_U^*e_i\otimes f_U^*\alpha_i.
\end{equation}
We have
\begin{lem}\label{pullback-support}
$\emph{supp}f^*\omega\subseteq f^{-1}(\emph{supp}\omega)$.
\end{lem}
\begin{proof}
For any $y\in Y- f^{-1}(\textrm{supp}\omega)$, $f(y)\in X-\textrm{supp}\omega$.
There exists an $\mathcal{E}$-free open neighborhood $V$ of $f(y)$ such that $V\subseteq X- \textrm{supp}\omega$, and then $\omega|_V=0$. By the local
representation (\ref{pullback-rep}) of $f^*\omega$, $(f^*\omega)|_{f^{-1}(V)}=0$, i.e., $\textrm{supp}f^*\omega\cap f^{-1}(V)=\emptyset$. So $y$ is not in
$\textrm{supp}f^*\omega$. We proved the lemma.
\end{proof}

By Lemma \ref{pullback-support}, the pullback (\ref{pullback0})  gives
\begin{equation}\label{pull}
f^*:\Gamma_\Phi(X,\mathcal{E}\otimes\mathcal{A}_X^{p,q})\rightarrow\Gamma_{f^{-1}\Phi}(Y,f^*\mathcal{E}\otimes\mathcal{A}_Y^{p,q})
\end{equation}
for a paracompactifying family  $\Phi$  of supports on $X$.

\textbf{Pushforwards.}
Assume that $S$ is an $f^*\mathcal{E}$-valued $(p,q)$-current on $Y$  satisfying that $f|_{\textrm{supp}S}:\textrm{supp}S\rightarrow X$ is \emph{proper}.
Let $\mathfrak{U}$ be an  $\mathcal{E}$-free  covering of $X$ and  $e^U_1$, $\ldots$, $e^U_m$  a basis of $\Gamma(U,\mathcal{E})$ as an
$\mathcal{O}_X(U)$-module for any $U\in\mathfrak{U}$.
For $U\in\mathfrak{U}$, $S$ can be written as $\sum\limits_{i=1}^{m}f_U^*e^U_i\otimes T_i$ on $f^{-1}(U)$, where $T_1$, $\ldots$, $T_m\in \mathcal{D}^{\prime
p,q}(f^{-1}(U))$.
By Lemma \ref{support}, $\textrm{supp}T_i\subseteq\textrm{supp}S\cap f^{-1}(U)$, and then $f_U|_{\textrm{supp}T_i}:\textrm{supp}T_i\rightarrow U$ is proper.
So $f_{U*}T_i$ is well defined.
Define  an $\mathcal{E}$-valued $(p-r,q-r)$-current
\begin{equation}\label{pushout}
\widetilde{S}_U=\sum_{i=1}^{m}e^U_i\otimes f_{U*}T_i
\end{equation}
on $U$.
For $U$, $U^\prime\in\mathfrak{U}$ satisfying $U\cap U^\prime\neq\emptyset$, $\widetilde{S}_U=\widetilde{S}_{U^\prime}$ on $U\cap U^\prime$.
We obtain an $\mathcal{E}$-valued $(p-r,q-r)$-current on $X$ , denoted by $f_*S$, such that $(f_*S)|_U=\widetilde{S}_U$.
The definition of  $f_*S$ is independent of the choice of $\mathcal{E}$-free coverings.

\begin{lem}\label{pushout-support}
Suppose that $f|_{\emph{supp}S}:\emph{supp}S\rightarrow X$ is proper. Then $\emph{supp}f_*S\subseteq f(\emph{supp}S)$.
\end{lem}
\begin{proof}
For any $x\in X- f(\textrm{supp}S)$, there exists an $\mathcal{E}$-free open neighborhood $U$ of $x$ such that $U\subseteq X- f(\textrm{supp}S)$. Clearly,
$f^{-1}(U)\cap\textrm{supp}S=\emptyset$, and then $S|_{f^{-1}(U)}=0$. By the definition (\ref{pushout}) of $f_*S$, $(f_*S)|_U=0$, i.e., $\textrm{supp}f_*S\cap
U=\emptyset$. Hence, $x$ is not in $\textrm{supp}f_*S$.  The lemma follows.
\end{proof}

Let $\Phi$ be a paracompactifying family  of supports on $X$.
Then
\begin{displaymath}
\Phi(c):=\{K\in f^{-1}\Phi|\mbox{ }f|_K:K\rightarrow X \mbox{ is proper}\}
\end{displaymath}
is paracompactifying on $Y$ by \cite[IV. 5.3 (b), 5.5]{Br}.
By Lemma \ref{pushout-support}, we get a \emph{pushforward}
\begin{equation}\label{pushout0}
f_*:\Gamma_{\Phi(c)}(Y,f^*\mathcal{E}\otimes\mathcal{D}_Y^{\prime p,q})\rightarrow\Gamma_\Phi(X,\mathcal{E}\otimes\mathcal{D}_X^{\prime p-r,q-r}).
\end{equation}
If $f$ is \emph{proper}, $\Phi(c)=f^{-1}\Phi$ and hence (\ref{pushout0}) is
\begin{displaymath}
f_*:\Gamma_{f^{-1}\Phi}(Y,f^*\mathcal{E}\otimes\mathcal{D}_Y^{\prime p,q})\rightarrow\Gamma_{\Phi}(X,\mathcal{E}\otimes\mathcal{D}_X^{\prime p-r,q-r}).
\end{displaymath}

Let $j:V\rightarrow X$ be the inclusion  of an open subset $V$ into $X$.
Clearly, $\Phi|_V\subseteq \Phi(c)$ in such case.
For $K\in \Phi(c)$, $K\cap S$ is compact for any compact set $S\subseteq X$, so $K$ is closed in $X$.
Since $K\in j^{-1}\Phi$, $K=L\cap V\subseteq L$ for some $L\in\Phi$, which implies  $K\in \Phi|_V$.
So  $\Phi(c)\subseteq \Phi|_V$.
Hence $\Phi(c)=\Phi|_V$.
The pushforward $j_*$ is just \emph{the extension by zero} of sections of the sheaf $\mathcal{E}\otimes\mathcal{D}_X^{\prime p,q}$, i.e., (\ref{extension2}).

We easily check that
\begin{displaymath}
\xymatrix{
\Gamma_{\Phi|_V}(V,\mathcal{E}\otimes\mathcal{A}_X^{p,q})\ar[r]^{\quad j_{*}} & \Gamma_{\Phi}(X,\mathcal{E}\otimes\mathcal{A}_X^{p,q})\ar[r]^{j^*}&
\Gamma_{j^{-1}\Phi}(V,\mathcal{E}\otimes\mathcal{A}_X^{p,q})
},
\end{displaymath}
\begin{equation}\label{push-pull-1}
\xymatrix{
\Gamma_{\Phi|_V}(V,\mathcal{E}\otimes\mathcal{D}_X^{\prime p,q})\ar[r]^{\quad j_{*}} & \Gamma_{\Phi}(X,\mathcal{E}\otimes\mathcal{D}_X^{\prime
p,q})\ar[r]^{j^*}& \Gamma_{j^{-1}\Phi}(V,\mathcal{E}\otimes\mathcal{D}_X^{\prime p,q})
}
\end{equation}
are both inclusions and
\begin{prop}\label{com1}
Let $f:Y\rightarrow X$ be a proper holomorphic map of  complex manifolds and let $\mathcal{E}$ be a locally free sheaf of $\mathcal{O}_X$-modules of finite
rank on $X$.
For an open set $V\subseteq X$, denote by $f_V:f^{-1}(V)\rightarrow V$  the restriction of $f$ to $f^{-1}(V)$ and by  $j:V\rightarrow X$, $j':f^{-1}(V)\rightarrow Y$ the inclusions.
Assume that $\Phi$ is a paracompactifying family of supports on $X$.
Then $j'_*f_V^*=f^*j_*$ on $\Gamma_{\Phi|_V}(V,\mathcal{E}\otimes\mathcal{A}_X^{\bullet,\bullet})$ and $f_{V*}j'^*=j^*f_*$ on
$\Gamma_{f^{-1}\Phi}(Y,f^*\mathcal{E}\otimes\mathcal{D}_Y^{\prime\bullet,\bullet})$.
\end{prop}

We have the Mayer-Vietoris sequences as follows.
\begin{prop}\label{short-exact}
Let $X$ be a complex manifold and $X=U\cup V$ for open sets $U$, $V$.
Denote the corresponding inclusions by $j_1:U\rightarrow X$, $j_2:V\rightarrow X$, $j^{\prime}_1:U\cap V\rightarrow V$, $j^{\prime}_2:U\cap V\rightarrow U$
respectively.
Assume that $\Phi$ is a paracompactifying family of supports on $X$.
Then
\begin{displaymath}
\tiny{\xymatrix{
0\ar[r]^{} &\Gamma_{\Phi|_{U\cap V}}(U\cap V,\mathcal{E}\otimes\mathcal{A}_X^{p,q})\ar[r]^{(j^{\prime}_{2*},j^{\prime}_{1*})\quad\qquad} &
\Gamma_{\Phi|_U}(U,\mathcal{E}\otimes\mathcal{A}_X^{p,q})\oplus \Gamma_{\Phi|_V}(V,\mathcal{E}\otimes\mathcal{A}_X^{p,q})\ar[r]^{\qquad\qquad\quad
j_{1*}-j_{2*}}& \Gamma_{\Phi}(X,\mathcal{E}\otimes\mathcal{A}_X^{p,q})\ar[r]^{} & 0
}},
\end{displaymath}
\begin{displaymath}
\tiny{\xymatrix{
0\ar[r]^{} &\Gamma_{\Phi|_{U\cap V}}(U\cap V,\mathcal{E}\otimes\mathcal{D}_X^{\prime p,q})\ar[r]^{(j^{\prime}_{2*},j^{\prime}_{1*})\quad\qquad} &
\Gamma_{\Phi|_U}(U,\mathcal{E}\otimes\mathcal{D}_X^{\prime p,q})\oplus \Gamma_{\Phi|_V}(V,\mathcal{E}\otimes\mathcal{D}_X^{\prime p,q})\ar[r]^{\qquad\qquad\quad
j_{1*}-j_{2*}}& \Gamma_{\Phi}(X,\mathcal{E}\otimes\mathcal{D}_X^{\prime p,q})\ar[r]^{} & 0
}}
\end{displaymath}
are exact sequences for any $p$, $q$.
\end{prop}
\begin{proof}
The proofs of the two conclusions are similarly and we only give the proof of the first one.
Clearly, $(j^{\prime}_{2*},j^{\prime}_{1*})$ is injective and $(j_{1*}-j_{2*})\circ(j^{\prime}_{2*},j^{\prime}_{1*})=0$.
Suppose that  $\alpha_1\in \Gamma_{\Phi|_V}(V,\mathcal{E}\otimes\mathcal{A}_X^{p,q})$ and $\alpha_2\in
\Gamma_{\Phi|_U}(U,\mathcal{E}\otimes\mathcal{A}_X^{p,q})$ satisfy $j_{1*}\alpha_2-j_{2*}\alpha_1=0$.
Then $\textrm{supp}\alpha_1=\textrm{supp}\alpha_2\subseteq U\cap V$ is in $\Phi$.
Set $\alpha=\alpha_1|_{U\cap V}$.
Then $\textrm{supp}\alpha\in \Phi|_{U\cap V}$.
Moreover, $j^{\prime}_{i*}\alpha=\alpha_i$ for $i=1$, $2$.
Hence $\textrm{ker} (j_{1*}-j_{2*})\subseteq \textrm{Im} (j^{\prime}_{2*}-j^{\prime}_{1*})$.
Let $\{\rho_U,\mbox{ }\rho_V\}$ be a partition of unity subordinate to $\{U,\mbox{ }V\}$.
For any $\beta\in \Gamma_{\Phi}(X,\mathcal{E}\otimes\mathcal{A}_X^{p,q})$, $\textrm{supp} (\rho_U\cdot \beta)\subseteq U$,
hence $\textrm{supp} (\rho_U\cdot \beta)|_U=\textrm{supp} (\rho_U\cdot \beta)$.
Notice that $\textrm{supp} (\rho_U\cdot \beta)$ is closed in $X$ and $\textrm{supp} (\rho_U\cdot \beta)\subseteq\textrm{supp} \beta \in\Phi$, so $\textrm{supp}
(\rho_U\cdot \beta)|_U\in \Phi|_U$.
Moreover, $j_{1*}((\rho_U\cdot \beta)|_U)=\rho_U\cdot \beta$.
Similarly, $\textrm{supp} (\rho_V\cdot \beta)|_V\in \Phi|_V$ and $j_{2*}((\rho_V\cdot \beta)|_V)=\rho_V\cdot \beta$.
Then $\beta=(j_{1*},\mbox{ }j_{2*})\left((\rho_U\cdot \beta)|_U,-(\rho_V\cdot \beta)|_V\right)$.
Hence $j_{1*}-j_{2*}$ is surjective.
We finish the proof.
\end{proof}

\subsubsection{\emph{\textbf{Twisted Dolbeault cohomology}}}
We still denote by $\bar{\partial}$ the differentials $1\otimes \bar{\partial}:\mathcal{E}\otimes\mathcal{A}_X^{\bullet,\bullet}\rightarrow
\mathcal{E}\otimes\mathcal{A}_X^{\bullet,\bullet+1}$ and $1\otimes \bar{\partial}:\mathcal{E}\otimes\mathcal{D}_X^{\prime \bullet,\bullet}\rightarrow
\mathcal{E}\otimes\mathcal{D}_X^{\prime \bullet,\bullet+1}$.
Let $\Phi$ be a paracompactifying family of supports on $X$.
Then $\mathcal{E}\otimes\Omega_X^p$ has two $\Phi$-soft resolutions
\begin{displaymath}
\xymatrix{
0\ar[r] &\mathcal{E}\otimes\Omega_X^p\ar[r]^{i}
&\mathcal{E}\otimes\mathcal{A}_{X}^{p,0}\ar[r]^{\quad\bar{\partial}}&\cdots\ar[r]^{\bar{\partial}\quad}&\mathcal{E}\otimes\mathcal{A}_{X}^{p,n}\ar[r]&0
},
\end{displaymath}
\begin{displaymath}
\xymatrix{
0\ar[r] &\mathcal{E}\otimes\Omega_X^p\ar[r]^{i} &\mathcal{E}\otimes\mathcal{D}_{X}^{\prime
p,0}\ar[r]^{\quad\bar{\partial}}&\cdots\ar[r]^{\bar{\partial}\quad}&\mathcal{E}\otimes\mathcal{D}_{X}^{\prime p,n}\ar[r]&0
},
\end{displaymath}
where $\textrm{dim}_{\mathbb{C}}X=n$.
So
\begin{displaymath}
H_\Phi^q(X,\mathcal{E}\otimes\Omega_X^p)\cong H^q(\Gamma_\Phi(X,\mathcal{E}\otimes\mathcal{A}_{X}^{p,\bullet}))\cong
H^q(\Gamma_\Phi(X,\mathcal{E}\otimes\mathcal{D}_{X}^{\prime p,\bullet})).
\end{displaymath}
The later two are uniformly called the \emph{twisted Dolbeault cohomology with supports in $\Phi$} and denoted by  $H_{\Phi}^{p,q}(X,\mathcal{E})$.
Assume that $E$ is the holomorphic vector bundle associated to $\mathcal{E}$.
Then $H^{p,q}(X,\mathcal{E})$ coincides with the bundle-valued Dolbeault cohomology $H^{p,q}(X,E)$ (see \cite[V. Proposition 11.5]{Dem}).
Clearly, all operators defined in Sects 3.1.1-3.1.3 commutate with $\bar{\partial}$, hence induce the corresponding morphisms at the level of cohomology.

For a complex manifold $X$, denote by $\mathcal{M}od(\mathcal{O}_X)$ the  category of sheaves of $\mathcal{O}_X$-modules on $X$.
\begin{prop}\label{compatible-pullback}
Let $f:Y\rightarrow X$ be a flat holomorphic map \emph{(}e.g., holomorphic submersion, open embedding\emph{)}   of complex manifolds, i.e., $\mathcal{O}_Y$ is a flat  $f^{-1}\mathcal{O}_X$-module sheaf.
Assume that $\mathcal{E}$ is a locally free sheaf of $\mathcal{O}_X$-modules of finite rank on $X$
and $\Phi$ is a paracompactifying family of supports on $X$.
Then the pullback $f^*:H^{p,q}_\Phi(X,\mathcal{E})\rightarrow H^{p,q}_{f^{-1}\Phi}(Y,f^*\mathcal{E})$ defined via \emph{(\ref{pull})} is compatible with the second type of pullback defined via \emph{(\ref{pullback2})}.
\end{prop}
\begin{proof}
Let $\mathcal{I}^\bullet$ and $\mathcal{J}^\bullet$ be injective resolutions of $\mathcal{E}\otimes \Omega_X^{p}$ and $f^*\mathcal{E}\otimes \Omega_Y^{p}$ in $\mathcal{M}od(\mathcal{O}_X)$ and $\mathcal{M}od(\mathcal{O}_Y)$  respectively.
By \cite[I. Theorem 6.2]{I}, there exist quasi-isomorphisms $\mathcal{E}\otimes \mathcal{A}_X^{p,\bullet}\rightarrow \mathcal{I}^\bullet$ and $f^*\mathcal{E}\otimes \mathcal{A}_Y^{p,\bullet}\rightarrow \mathcal{J}^\bullet$, which are unique up to chain homotopy.
Any injective sheaf is flabby (\cite[II. Proposition 5.3]{Br}) and hence is $\Phi$-acyclic (\cite[II. Proposition 5.5]{Br}).
We have the isomorphisms
$H_{\Phi}^{p,q}(X,\mathcal{E})\tilde{\rightarrow}H^q(\Gamma_{\Phi}(Y,\mathcal{I}^\bullet))$
and $H_{f^{-1}\Phi}^{p,q}(Y,f^*\mathcal{E})\tilde{\rightarrow}H^q(\Gamma_{f^{-1}\Phi}(Y,\mathcal{J}^{\bullet}))$ by \cite[II. 4.2]{Br}.
Since $f$ is flat, $f^*:\mathcal{M}od(\mathcal{O}_X)\rightarrow \mathcal{M}od(\mathcal{O}_Y)$ is an exact functor,
so $f^*(\mathcal{E}\otimes \mathcal{A}_X^{p,\bullet})\rightarrow f^*\mathcal{I}^\bullet$ is a quasi-isomorphism of complexes of sheaves of $\mathcal{O}_Y$-modules.
By \cite[I. Theorem 6.2]{I},  there is a unique morphism $\varphi:f^*\mathcal{I}^\bullet \rightarrow \mathcal{J}^{\bullet}$ of complexes of sheaves of $\mathcal{O}_Y$-modules up to chain homotopy such that the diagram
\begin{displaymath}
\xymatrix{
  f^*(\mathcal{E}\otimes \mathcal{A}_X^{p,\bullet})\ar[d]_{}   \ar[r]^{}  & f^*\mathcal{E}\otimes \mathcal{A}_Y^{p,\bullet} \ar[d]^{} \\
 f^*\mathcal{I}^\bullet \ar[r]^{\varphi}  & \mathcal{J}^{\bullet}.}
\end{displaymath}
is commutative up to chain homotopy, where the upper map is induced by the pullback  $f^*\mathcal{A}_X^{\bullet,\bullet}\rightarrow \mathcal{A}_Y^{\bullet,\bullet}$.
On cohomologies, we have the commutative diagram
\begin{equation}\label{com-pullback}
\xymatrix{
  H_{\Phi}^{p,q}(X,\mathcal{E})\ar[d]_{\cong}   \ar[r]^{}  &H^q(\Gamma_{f^{-1}\Phi}(Y,f^*(\mathcal{E}\otimes \mathcal{A}_X^{p,\bullet})))\ar[d]_{}   \ar[r]^{}  & H_{f^{-1}\Phi}^{p,q}(Y,f^*\mathcal{E}) \ar[d]^{\cong} \\
 H^q(\Gamma_{\Phi}(Y,\mathcal{I}^\bullet)) \ar[r]^{}  & H^q(\Gamma_{f^{-1}\Phi}(Y,f^*\mathcal{I}^\bullet)) \ar[r]^{\quad H^q(\Gamma_{f^{-1}\Phi}(\varphi))}  & H^q(\Gamma_{f^{-1}\Phi}(Y,\mathcal{J}^{\bullet})),}
\end{equation}
where the maps in the two rows of the left square are induced by adjunctions $\mathcal{E}\otimes \mathcal{A}_X^{p,\bullet}\rightarrow f_*f^*(\mathcal{E}\otimes \mathcal{A}_X^{p,\bullet})$ and $\mathcal{I}^\bullet\rightarrow f_*f^*\mathcal{I}^\bullet$ respectively.
The composition of the two maps in the lower row of (\ref{com-pullback}) is induced by
$\mathcal{I}^\bullet\rightarrow f_*f^*\mathcal{I}^\bullet\rightarrow f_*\mathcal{J}^{\bullet}$, denoted by
\begin{displaymath}
\delta:H^q(\Gamma_{\Phi}(Y,\mathcal{I}^\bullet))\rightarrow H^q(\Gamma_{f^{-1}\Phi}(Y,\mathcal{J}^{\bullet})).
\end{displaymath}
Notice that $\delta$ is independent of the choice of $\varphi$, since $H^q(\Gamma_{f^{-1}\Phi}(\varphi))$ is.
Clearly,  the composition of the two maps in the upper row of (\ref{com-pullback}) is just the pullback $f^*$ defined by (\ref{pull}).
Hence  (\ref{pull}) is compatible with  $\delta$.
By similar arguments,
we can prove that
the pullback $f^*$ defined by (\ref{pullback2}) is compatible with  $\delta$ using $(\mathcal{E}\otimes \Omega_X^{p})^{[\bullet]}$ and $(f^*\mathcal{E}\otimes \Omega_Y^{p})^{[\bullet]}$ instead of $\mathcal{E}\otimes \mathcal{A}_X^{p,\bullet}$ and $f^*\mathcal{E}\otimes \mathcal{A}_Y^{p,\bullet}$ respectively.
We complete the proof.
\end{proof}

Suppose that $\Phi$ and $\Psi$ are paracompactifying families of supports on $X$ with $\Phi\subseteq \Psi$.
The inclusion $\Gamma_\Phi(X,\mathcal{E}\otimes\mathcal{A}_{X}^{p,\bullet})\hookrightarrow \Gamma_\Psi(X,\mathcal{E}\otimes\mathcal{A}_{X}^{p,\bullet})$ naturally induces  a morphism
$l:H^q(\Gamma_\Phi(X,\mathcal{E}\otimes\mathcal{A}_{X}^{p,\bullet}))\rightarrow H^q(\Gamma_\Psi(X,\mathcal{E}\otimes\mathcal{A}_{X}^{p,\bullet}))$.
Let $f:Y\rightarrow X$ be a holomorphic map of  complex manifolds.
If $f^{-1}\Phi=f^{-1}\Psi$, there is a commutative diagram
\begin{equation}\label{sub0}
\xymatrix{
   H_{\Phi}^{p,q}(X,\mathcal{E}) \ar[rr]^{l} \ar[dr]_{f^*}
                &  &    H_{\Psi}^{p,q}(X,\mathcal{E}) \ar[dl]^{f^*}    \\
                & H_{f^{-1}\Phi}^{p,q}(Y,f^*\mathcal{E}). }
\end{equation}

\subsubsection{\emph{\textbf{Cup products}}}
Assume that $\mathcal{E}$, $\mathcal{F}$ are locally free sheaves of $\mathcal{O}_X$-modules of rank $m$, $n$ on $X$ respectively, $\mathfrak{U}$ is an
$\mathcal{E}$- and $\mathcal{F}$-free open covering of $X$ and $\Phi$, $\Psi$ is a paracompactifying family of supports on $X$.
Let $e^U_1$, $\ldots$, $e^U_m$ and $f^U_1$, $\ldots$, $f^U_n$ be bases of $\Gamma(U,\mathcal{E})$ and $\Gamma(U, \mathcal{F})$ as $\mathcal{O}_X(U)$-modules,
respectively.

For $S\in\Gamma_\Phi(X, \mathcal{E}\otimes\mathcal{D}_X^{\prime r,p})$, $\omega\in\Gamma_\Psi(X, \mathcal{F}\otimes\mathcal{A}_X^{s,q})$ and $U\in\mathfrak{U}$, $S$ and
$\omega$  are represented by $\sum\limits_{i=1}^{m}e^U_i\otimes T_i$ and $\sum\limits_{j=1}^{n}f^U_j\otimes \alpha_j$ on $U$ respectively, where
$T_i\in\mathcal{D}^{\prime r,p}(U)$ and $\alpha_j\in\mathcal{A}^{s,q}(U)$ for any $i$.
Then
\begin{displaymath}
\sum\limits_{\substack{1\leq i\leq m\\1\leq j\leq n}} e^U_i\otimes f^U_j\otimes (T_i\wedge\alpha_j)
\end{displaymath}
gives an $\mathcal{E}\otimes\mathcal{F}$-valued $(r+s,p+q)$-current on $U$. They are glued as a global section of
$\mathcal{E}\otimes\mathcal{F}\otimes\mathcal{D}_X^{\prime r+s,p+q}$ on $X$, which is independent of the choice of open coverings.
Denote it by $S\wedge\omega$.
We easily check that $\textrm{supp} (S\wedge \omega)\subseteq \textrm{supp} S \cap \textrm{supp} \omega$,
hence $S\wedge \omega\in \Gamma_{\Phi\cap\Psi}(X, \mathcal{E}\otimes\mathcal{F}\otimes\mathcal{D}_X^{\prime r+s,p+q})$.
Similarly, the $\mathcal{F}\otimes\mathcal{E}$-valued current
$\omega\wedge S$ and the $\mathcal{E}\otimes\mathcal{F}$-valued form $\psi\wedge \omega$  on $X$ can be defined well, where $\psi\in\Gamma_\Phi(X,
\mathcal{E}\otimes\mathcal{A}_X^{r,p})$.
Clearly,
$\bar{\partial}(S\wedge\omega)=\bar{\partial}S\wedge\omega+(-1)^{p+q}S\wedge \bar{\partial}\omega$.

Denote by $[\alpha]_\Phi$  the Dolbeault  class with support in $\Phi$ of  the $\bar{\partial}$-closed form or current $\alpha$.
For paracompactifying families $\Phi$, $\Psi$ of supports on $X$,
define a \emph{cup product}
\begin{equation}\label{cup3}
\cup:H_\Phi^{r,p}(X,\mathcal{E})\times H_\Psi^{s,q}(X,\mathcal{F})\rightarrow H_{\Phi\cap\Psi}^{r+s,p+q}(X,\mathcal{E}\otimes\mathcal{F})
\end{equation}
as $[\psi]_\Phi\cup[\omega]_\Psi=[\psi\wedge\omega]_{\Phi\cap\Psi}$ or $[S]_\Phi\cup[\omega]_\Psi=[S\wedge\omega]_{\Phi\cap\Psi}$.

\begin{rem}\label{compatible-cup product}
There is a  commutative diagram
\begin{displaymath}
\xymatrix{
  (\mathcal{E}\otimes \Omega_X^p)\otimes  (\mathcal{F}\otimes \Omega_X^q)\ar[d]_{(1\otimes i)\otimes (1\otimes i)}   \ar[r]^{\qquad\wedge}  & \mathcal{E}\otimes \mathcal{F}\otimes \Omega_X^{p+q} \ar[d]^{1\otimes 1\otimes i} \\
 (\mathcal{E}\otimes \mathcal{A}_X^{p,\bullet})\otimes  (\mathcal{F}\otimes \mathcal{A}_X^{q,\bullet}) \ar[r]^{\qquad\wedge}  & \mathcal{E}\otimes \mathcal{F}\otimes \mathcal{A}_X^{p+q,\bullet}.}
\end{displaymath}
By \cite[II. Th\'{e}or\`{e}me 6.6.1]{Go}, the cup product (\ref{cup3}) coincides with the second type of cup product (\ref{cup2}).
\end{rem}

Suppose that $f:Y\rightarrow X$ is a holomorphic map of complex manifolds.
Then $f^*(\psi\wedge\omega)=f^*\psi\wedge f^*\omega$.
In addition, let $T$ be an $f^*\mathcal{E}$-valued current on $Y$ satisfying that
$f|_{\textrm{supp}T}:\textrm{supp}T\rightarrow X$ is proper.
Then
\begin{equation}\label{pro-formula1}
f_*(T\wedge f^*\omega)=f_*T\wedge\omega.
\end{equation}
Indeed, it is the classical projection formula locally.
For $\varphi\in H_{\Phi(c)}^{\bullet,\bullet}(Y,f^*\mathcal{E})$ and $\eta\in H_{\Psi}^{\bullet,\bullet}(X,\mathcal{F})$,
$f_*(\varphi\cup f^*\eta)\in H_{\Phi(c)\cap f^{-1}\Psi}^{\bullet,\bullet}(X,\mathcal{E}\otimes \mathcal{F})$ and  $f_*\varphi\cup\eta\in
H_{\Phi\cap\Psi}^{\bullet,\bullet}(X,\mathcal{E}\otimes \mathcal{F})$.
By \cite[IV. 5.4 (7)]{Br},  $\Phi(c)\cap f^{-1}\Psi=(\Phi\cap\Psi)(c)$.
By (\ref{pro-formula1}),
\begin{equation}\label{pro-formula2}
f_*(\varphi\cup f^*\eta)=f_*\varphi\cup\eta.
\end{equation}

\begin{prop}\label{inj-surj}
Let $f:Y\rightarrow X$ be a proper surjective holomorphic map between  complex manifolds and let $\Phi$ be a paracompactifying family of supports on $X$.
Set $r=\emph{dim}_{\mathbb{C}}Y-\emph{dim}_{\mathbb{C}}X$ and assume that there exists a closed current $T\in \mathcal{D}^{\prime r,r}(Y)$ such that $f_*T\neq 0$.
Let  $\mathcal{E}$ be a locally free sheaf of  $\mathcal{O}_X$-modules of finite rank on $X$.
Then $f^*:H_\Phi^{\bullet,\bullet}(X,\mathcal{E})\rightarrow H_{f^{-1}\Phi}^{\bullet,\bullet}(Y,f^*\mathcal{E})$ is injective and  $f_*:H_{f^{-1}\Phi}^{\bullet,\bullet}(Y,f^*\mathcal{E})\rightarrow H_\Phi^{\bullet-r,\bullet-r}(X,\mathcal{E})$ is surjective.
In particular, if $X$ and $Y$ have the same dimensions, then
$f^*:H_\Phi^{\bullet,\bullet}(X,\mathcal{E})\rightarrow H_{f^{-1}\Phi}^{\bullet,\bullet}(Y,f^*\mathcal{E})$ is injective and  $f_*:H_{f^{-1}\Phi}^{\bullet,\bullet}(Y,f^*\mathcal{E})\rightarrow H_\Phi^{\bullet,\bullet}(X,\mathcal{E})$ is surjective.
\end{prop}
\begin{proof}
Since $c=f_*T$ is a closed  current of degree $0$, hence a constant.
By (\ref{pro-formula2}), $f_*([T]\cup f^*\eta)=c\cdot \eta$, where $[T]\in H^{r,r}(Y)$ and $\eta\in H_\Phi^{p,q}(X,\mathcal{E})$.
The proposition follows.
\end{proof}

\subsection{Local systems on smooth manifolds}
For a topology space $X$, \emph{a local system of $\mathbf{k}$- modules} on $X$ refers to a locally constant sheaf of $\mathbf{k}$-modules on $X$, or equivalently, a locally free sheaf of $\underline{\mathbf{k}}_X$-modules on $X$, where $\underline{\mathbf{k}}_X$ is the constant sheaf with stalk $\mathbf{k}$ on $X$.
Assume that $\mathcal{L}$ is a local system of $\mathbf{k}$-modules on $X$. An open subset $U$ of $X$ is said to be \emph{$\mathcal{L}$-constant}, if the restriction $\mathcal{L}|_U$ is a contant sheaf.
An open covering $\mathfrak{U}$ of $X$ is said to be \emph{$\mathcal{L}$-constant}, if all $U\in\mathfrak{U}$ are $\mathcal{L}$-constant.

For a smooth map $f:Y\rightarrow X$ of smooth manifolds and  local systems $\mathcal{L}$, $\mathcal{H}$ of $\mathbf{k}$-modules of finite ranks on $X$, all notions can be similarly defined as those in Section 3.1  and the corresponding results are also true, where we only need to  replace  $\mathcal{O}_X$,  $\mathcal{E}$, $\mathcal{E}$-free, $\mathcal{A}_X^{*,*}$, $\mathcal{D}_X^{\prime *,*}$, $f^*\mathcal{E}$, $H_\Phi^{*,*}(X,\mathcal{E})$ and $\mathcal{E}\otimes_{\mathcal{O}_X}\mathcal{F}$ with $\underline{\mathbf{k}}_X$, $\mathcal{L}$, $\mathcal{L}$-constant, $\mathcal{A}_X^{*}$, $\mathcal{D}_X^{\prime *}$, $f^{-1}\mathcal{L}$, $H_\Phi^*(X,\mathcal{L})$ and $\mathcal{L}\otimes_{\underline{\mathbf{k}}_X}\mathcal{H}$, respectively.
\emph{It is noteworthy that}, if the definitions of notions involve the currents, then the related manifolds must be \emph{oriented}.

Now, we list partial results as follows, which will be frequently used in Sect. 5.

1. Let $\mathcal{L}$ and $\mathcal{H}$ be local systems of $\mathbf{k}$-modules of finite rank on $X$.
Suppose that $f:Y\rightarrow X$ is a smooth map of oriented smooth manifolds. In addition, let $T$ be an $f^{-1}\mathcal{L}$-valued current on $Y$ satisfying that
$f|_{\textrm{supp}T}:\textrm{supp}T\rightarrow X$ is proper.
Then
\begin{equation}\label{pro-formula1-dR}
f_*(T\wedge f^*\omega)=f_*T\wedge\omega.
\end{equation}
For $\varphi\in H_{\Phi(c)}^{\bullet}(Y,f^{-1}\mathcal{L})$ and $\eta\in H_{\Psi}^{\bullet}(X,\mathcal{H})$,
\begin{equation}\label{pro-formula2-dR}
f_*(\varphi\cup f^*\eta)=f_*\varphi\cup\eta.
\end{equation}

2. Suppose that $\Phi$ and $\Psi$ are paracompactifying families of supports on $X$ with $\Phi\subseteq \Psi$.
The inclusion $\Gamma_\Phi(X,\mathcal{L}\otimes\mathcal{A}_{X}^{\bullet})\hookrightarrow \Gamma_\Psi(X,\mathcal{L}\otimes\mathcal{A}_{X}^{\bullet})$ naturally induces  a morphism
$l:H_{\Phi}^{p}(X,\mathcal{L})\rightarrow H_{\Psi}^{p}(X,\mathcal{L})$.
Let $f:Y\rightarrow X$ be a smooth map of  smooth manifolds.
If $f^{-1}\Phi=f^{-1}\Psi$, there is a commutative diagram
\begin{equation}\label{sub00}
\xymatrix{
   H_{\Phi}^{p}(X,\mathcal{L}) \ar[rr]^{l} \ar[dr]_{f^*}
                &  &    H_{\Psi}^{p}(X,\mathcal{L}) \ar[dl]^{f^*}    \\
                & H_{f^{-1}\Phi}^{p}(Y,f^{-1}\mathcal{L}). }
\end{equation}

\section{Twisted Dolbeault cohomology with supports}
\subsection{K\"{u}nneth theorems}
Suppose that $\mathcal{F}$ and $\mathcal{G}$ are coherent analytic sheaves   on complex manifolds $X$ and $Y$ respectively.
The \emph{analytic external tensor product} of $\mathcal{F}$ and $\mathcal{G}$ is defined as
\begin{displaymath}
\mathcal{F}\boxtimes\mathcal{G}=pr_1^*\mathcal{F}\otimes_{\mathcal{O}_{X\times Y}}pr_2^*\mathcal{G},
\end{displaymath}
where $pr_1$ and $pr_2$ are projections from $X\times Y$ onto $X$, $Y$, respectively.
The \emph{cartesian product}
\begin{equation}\label{cartesian}
\times:H^p_\Phi(X,\mathcal{F})\times H^q_{\Psi}(Y,\mathcal{G})\rightarrow H^{p+q}_{\Phi\times\Psi}(X\times Y,\mathcal{F}\boxtimes \mathcal{G})
\end{equation}
is defined as $pr_1^*(\bullet)\cup pr_2^*(\bullet)$.

Denote by $K^\bullet$ the associated simple complex of a double complex $K^{\bullet,\bullet}$.
For a double complex $K^{\bullet,\bullet}$, there are two spectral sequences
$_KE^{\bullet,\bullet}_r\Rightarrow\mbox{ } _KH^\bullet$ and $_K\widetilde{E}^{\bullet,\bullet}_r\Rightarrow\mbox{ }_K\widetilde{H}^\bullet$,
where $_KE^{p,q}_1=H^q(K^{p,\bullet})$, $_KE^{p,q}_2=H^p(E^{\bullet,q}_1)$,
$_K\widetilde{E}^{p,q}_1=H^p(K^{\bullet,q})$, $_K\widetilde{E}^{p,q}_2=H^q(E^{p,\bullet}_1)$,
$_KH^k=\mbox{ }_K\widetilde{H}^k=H^k(K^\bullet)$.

\begin{prop}\label{Kun}
Let $\mathcal{F}$, $\mathcal{G}$ be coherent analytic sheaves on complex manifolds $X$, $Y$ respectively and let $\Phi$ be a family of supports on $X$.
Suppose that $H^\bullet(Y,\mathcal{G})$ is finite dimensional.
Then $(\alpha^p\otimes \beta^q)_{p+q=k}\mapsto \sum\limits_{p+q=k}\alpha^p\times \beta^q$ gives an isomorphism
\begin{displaymath}
\bigoplus\limits_{p+q=k}H_\Phi^p(X,\mathcal{F})\otimes_{\mathbb{C}} H^q(Y,\mathcal{G})\tilde{\rightarrow} H_{\Phi\times Y}^k(X\times Y, \mathcal{F}\boxtimes \mathcal{G}).
\end{displaymath}
\end{prop}
\begin{proof}
Let $C^{\bullet,\bullet}=\Gamma_\Phi(X,\mathcal{F}^{[\bullet]})\otimes_{\mathbb{C}}\Gamma(Y,\mathcal{G}^{[\bullet]})$ be the double complex associated to
the complexes $\Gamma_\Phi(X,\mathcal{F}^{[\bullet]})$ and $\Gamma(Y,\mathcal{G}^{[\bullet]})$.
By \cite[IV. (15.8)]{Dem},
$_CE_1^{p,q}=\Gamma_\Phi(X,\mathcal{F}^{[p]})\otimes_{\mathbb{C}} H^q(Y,\mathcal{G})$
and  $_CE_2^{p,q}=H^p_\Phi(X,\mathcal{F})\otimes_{\mathbb{C}} H^q(Y,\mathcal{G})$.
Let $\pi:X\times Y\rightarrow X$ be the first projection of $X\times Y$ onto $X$.
Set $D^{p,q}=\Gamma_\Phi\left(X,(\pi_*((\mathcal{F}\boxtimes \mathcal{G})^{[q]}))^{[p]}\right)$.
Then
\begin{displaymath}
\begin{aligned}
_DE^{p,q}_1=&H^q\left(\Gamma_\Phi(X,(\pi_*((\mathcal{F}\boxtimes \mathcal{G})^{[\bullet]}))^{[p]})\right)\\
=&\Gamma_\Phi\left(X,\left(\mathcal{H}^q\left(\pi_*((\mathcal{F}\boxtimes \mathcal{G})^{[\bullet]})\right)\right)^{[p]}\right)\quad\mbox{ }( \mbox{by (\ref{exact-functor})})\\
=&\Gamma_\Phi\left(X,\left(R^q\pi_*(\mathcal{F}\boxtimes \mathcal{G})\right)^{[p]}\right)
\end{aligned}
\end{displaymath}
and  $_DE^{p,q}_2=H^p_\Phi(X,R^q\pi_*(\mathcal{F}\boxtimes \mathcal{G}))$.
Since  $H^\bullet(Y,\mathcal{G})$ is finite dimensional,
$R^q\pi_*(\mathcal{F}\boxtimes \mathcal{G})=\mathcal{F}\otimes_{\mathbb{C}} H^q(Y,\mathcal{G})$ by \cite[IX. 5.22 (c)]{Dem}.
So $_DE^{p,q}_2=H^p_\Phi(X,\mathcal{F}\otimes_{\mathbb{C}} H^q(Y,\mathcal{G}))=H^p_\Phi(X,\mathcal{F})\otimes_{\mathbb{C}} H^q(Y,\mathcal{G})$.

Define a morphism $\varphi:C^{\bullet,\bullet}\rightarrow D^{\bullet,\bullet}$ of double complexes as $f\otimes g \mapsto h$, where
\begin{displaymath}
\small{\begin{aligned}
&h(\xi_0,\ldots,\xi_p; (x_0,y_0),\ldots,(x_q,y_q))=S(f(\xi_0,\ldots,\xi_p))(x_q)\otimes 1\otimes g(y_0,\ldots, y_q)\otimes 1\\
\in &(\mathcal{F}\boxtimes \mathcal{G})_{(x_q,y_q)}=(\mathcal{F}_{x_q}\otimes_{\mathcal{O}_{X,x_q}}
\mathcal{O}_{X\times Y,(x_q,y_q)})\otimes_{\mathcal{O}_{X\times Y,(x_q,y_q)}}
(\mathcal{G}_{y_q}\otimes_{\mathcal{O}_{Y,y_q}} \mathcal{O}_{X\times Y,(x_q,y_q)})
\end{aligned}}
\end{displaymath}
for any $(\xi_0,\ldots,\xi_p; (x_0,y_0),\ldots,(x_q,y_q))\in X^{p+1}\times (X\times Y)^{q+1}$.
By Lemma \ref{higher-direct-image}, the map is defined well.
It induces the isomorphism $_CE_2^{p,q}\tilde{\rightarrow} _DE_2^{p,q}$ for any $p$, $q$,
hence induces the isomorphism $_CE_r^{p,q}\tilde{\rightarrow}\mbox{ }_DE_r^{p,q}$ for $r\geq 2$.
Clearly,  $_CE_r^{\bullet,\bullet}$ degenerates at $E_2$-page, so does $_DE_r^{\bullet,\bullet}$.
So $\varphi$ induces the isomorphism  $H^k(\varphi):\mbox{ }_CH^k\tilde{\rightarrow} _DH^k$.
By \cite[IV, (15.5)]{Dem},  $_CH^k=H^p_\Phi(X,\mathcal{F})\otimes_{\mathbb{C}} H^q(Y,\mathcal{G})$.
Moreover,
\begin{displaymath}
_D\widetilde{E}_1^{p,q}=\left\{
 \begin{array}{ll}
\Gamma_{\Phi\times Y}\left(X\times Y,(\mathcal{F}\boxtimes \mathcal{G})^{[q]}\right),&~p=0,\\
 &\\
 0,&~p\geq 1,
 \end{array}
 \right.
\end{displaymath}
\begin{displaymath}
_D\widetilde{E}_2^{p,q}=\left\{
 \begin{array}{ll}
H^{q}_{\Phi\times Y}\left(X\times Y,\mathcal{F}\boxtimes \mathcal{G}\right),&~p=0,\\
 &\\
 0,&~p\geq 1,
 \end{array}
 \right.
\end{displaymath}
where we use \cite[IV. 5.2]{Br}.
Hence $_D\widetilde{E}_r^{\bullet,\bullet}$ degenerates at $E_2$-page, so $_DH^k=H^{k}_{\Phi\times Y}\left(X\times Y,\mathcal{F}\boxtimes \mathcal{G}\right)$.
From the definition of $\varphi$, $H^{\bullet}(\varphi)$  is just the cartesian product.
We complete the proof.
\end{proof}

For bigraded vector spaces $K^{\bullet,\bullet}$ and $L^{\bullet,\bullet}$ over $\mathbb{C}$, the associated bigraded space
$K^{\bullet,\bullet}\otimes_{\mathbb{C}}L^{\bullet,\bullet}$ over $\mathbb{C}$ is defined as
\begin{displaymath}
(K^{\bullet,\bullet}\otimes_{\mathbb{C}}L^{\bullet,\bullet})^{p,q}=\bigoplus\limits_{\substack{k+l=p\\r+s=q}}K^{k,r}\otimes_{\mathbb{C}} L^{l,s}
\end{displaymath}
for any $p$, $q$.

Let  $X$ and $Y$ be two complex manifolds.
Suppose that $H^{\bullet,\bullet}(Y)$ is finite dimensional and $\Phi$ is a paracompacting family of supports on $X$.
Notice that $\Omega_{X\times Y}^p=\bigoplus\limits_{r+s=p}\Omega_{X}^r\boxtimes\Omega_{Y}^s$.
By Proposition  \ref{compatible-pullback} and Remark \ref{compatible-cup product},
\begin{equation}\label{Kun-Dol-cpt}
pr_1^*(\bullet)\cup pr_2^*(\bullet):H_{\Phi}^{\bullet,\bullet}(X)\otimes_{\mathbb{C}} H^{\bullet,\bullet}(Y)\rightarrow H_{\Phi\times Y}^{\bullet,\bullet}(X\times Y)
\end{equation}
coincides with the cartesian product (\ref{cartesian}), which is an isomorphism by Proposition  \ref{Kun}.

\subsection{Leray-Hirsch theorem}
Using Borel's spectral sequence, L. Cordero et al. \cite{CFGU} established  a version of Leray-Hirsch theorem for Dolbeault cohomology and  S. Rao et al.
\cite{RYY2} extend this result to the twisted cases with the similar way.
We obtained a version of these results with a different way \cite{M2,M3}.
Now, we further generalize the Leray-Hirsch theorems on the twisted Dolbeault cohomologies with supports.

\begin{thm}\label{L-H1}
Let $\pi:E\rightarrow X$ be a holomorphic fiber bundle over a complex manifold $X$ and let $\mathcal{E}$ be a locally free sheaf of
$\mathcal{O}_X$-modules of finite rank on $X$. Assume that there exist $e_i\in H^{\bullet,\bullet}(E)$ with degree $(u_i,v_i)$ for $1\leq i\leq r$ such that  their restrictions $e_1|_{E_x},\dots,e_r|_{E_x}$ freely linearly generate $H^{\bullet,\bullet}(E_x)$ for every $x\in X$.
Then
\begin{displaymath}
\sum\limits_{i=1}^r\pi^*(\bullet)\cup e_i:\bigoplus\limits_{i=1}^rH_\Phi^{\bullet-u_i,\bullet-v_i}(X,\mathcal{E}) \rightarrow H_{\pi^{-1}\Phi}^{\bullet,\bullet}(E,\pi^*\mathcal{E})
\end{displaymath}
is an isomorphism for  a paracompactifying family $\Phi$  of supports on $X$.
\end{thm}
\begin{proof}
Let $t_i$ be a $\bar{\partial}$-closed form of degree $(u_i,v_i)$ in $\mathcal{A}^{\bullet,\bullet}(E)$, such that $e_i=[t_i]$ for $1\leq i\leq r$. For any open
set $U\subseteq X$, set $E_U=\pi^{-1}(U)$ and $B^{\bullet,\bullet}(U)=\bigoplus\limits_{i=1}^{r}\mathcal{A}_{\Phi|_U}^{\bullet-u_i,\bullet-v_i}(U)$.
For any $p$, the $q$-th cohomology of the complex $B^{p,\bullet}(U)$  is $D^{p,q}(U)=\bigoplus\limits_{i=1}^{r}H_{\Phi|_U}^{p-u_i,q-v_i}(U)$.
The morphism
\begin{displaymath}
f^{p}_U=\sum\limits_{i=1}^r\pi^*(\bullet)\wedge t_i:B^{p,\bullet}(U)\rightarrow \mathcal{A}_{(\pi|_{E_U})^{-1}(\Phi|_{U})}^{p,\bullet}(E_U)=\mathcal{A}_{(\pi^{-1}\Phi)|_{E_U}}^{p,\bullet}(E_U)
\end{displaymath}
of complexes is defined well by Proposition \ref{inverse-paracompact} $(3)$, which induces a morphism
\begin{displaymath}
F^{p,q}_U=\sum\limits_{i=1}^r\pi^*(\bullet)\cup e_i:D^{p,q}(U)\rightarrow G^{p,q}(U):=H_{(\pi^{-1}\Phi)|_{E_U}}^{p,q}(E_U).
\end{displaymath}
Denoted by $\mathcal{P}(U)$ the statement that $F^{p,q}_U$ are isomorphisms for all $p$, $q$.
The theorem is equivalent to say  that $\mathcal{P}(X)$ holds.
We only need to check the three conditions in Lemma \ref{glued}.
Clearly, $\mathcal{P}$  satisfies the disjoint condition.

We fix some notations. For the inclusion $j:U\subseteq V$ of open sets in $X$, denote by $\tilde{j}:E_U\subseteq E_V$ the corresponding inclusion and denote by
$j_*:\mathcal{A}_{\Phi|_U}^{\bullet,\bullet}(U)\rightarrow \mathcal{A}_{\Phi|_V}^{\bullet,\bullet}(V)$,
$\tilde{j}_*:\mathcal{A}_{(\pi^{-1}\Phi)|_{E_U}}^{\bullet,\bullet}(E_U)\rightarrow \mathcal{A}_{(\pi^{-1}\Phi)|_{E_V}}^{\bullet,\bullet}(E_V)$,
$J=(j_*,\ldots,j_*):B^{\bullet,\bullet}(U)\rightarrow B^{\bullet,\bullet}(V)$ the extensions by zero.
Now, we go back to the proof. Fix an integer $p$.
For open sets $U$, $V\subseteq X$,
let $j_1:U\cap V\rightarrow U$, $j_2:U\cap V\rightarrow V$, $j_3:U\rightarrow U\cup V$, $j_4:V\rightarrow U\cup V$ be inclusions.
By Propositions \ref{short-exact} and \ref{com1},
there is a commutative diagram
\begin{displaymath}
\tiny{\xymatrix{
0\ar[r]&B^{p,\bullet}(U\cap V)\ar[d]^{f^{p}_{U\cap V}}\quad \ar[r]^{(J_1,J_2)\quad\quad} & \quad B^{p,\bullet}(U)\oplus
B^{p,\bullet}(V)\ar[d]^{(f^{p}_U,f^{p}_V)}\quad\ar[r]^{\quad\quad J_3-J_4}& \quad B^{p,\bullet}(U\cup V) \ar[d]^{f^{p}_{U\cup V}}\ar[r]& 0\\
0\ar[r]&\mathcal{A}_{(\pi^{-1}\Phi)|_{E_{U\cap V}}}^{p,\bullet}(E_{U\cap V})     \quad  \ar[r]^{(\tilde{j}_{1*},\tilde{j}_{2*})\qquad\qquad}& \quad
\mathcal{A}_{(\pi^{-1}\Phi)|_{E_U}}^{p,\bullet}(E_U)\oplus \mathcal{A}_{(\pi^{-1}\Phi)|_{E_V}}^{p,\bullet}(E_V)    \quad\ar[r]^{\qquad\qquad
\tilde{j}_{3*}-\tilde{j}_{4*}} & \quad \mathcal{A}_{(\pi^{-1}\Phi)|_{E_{U\cup V}}}^{p,\bullet}(E_{U\cup V})     \ar[r]& 0}}
\end{displaymath}
of  exact sequences of complexes.
Therefore, we have a commutative diagram
\begin{displaymath}
\small{\xymatrix{
    \cdots\ar[r]&D^{p,q}(U\cap V) \ar[d]^{F^{p,q}_{U\cap V}} \ar[r]& D^{p,q}(U)\oplus D^{p,q}(V)
    \ar[d]^{(F^{p,q}_U,F^{p,q}_V)}\ar[r]&  D^{p,q}(U\cup V)\ar[d]^{F^{p,q}_{U\cup V}}\ar[r]&D^{p,q+1}(U\cap V)\ar[d]^{F^{p,q+1}_{U\cap V}}\ar[r]&\cdots\\
 \cdots\ar[r]& G^{p,q}(U\cap V)\ar[r]&G^{p,q}(U)\oplus G^{p,q}(V)       \ar[r]& G^{p,q}(U\cup V)  \ar[r]&G^{p,q+1}(U\cap V)    \ar[r] & \cdots}}
\end{displaymath}
of long exact sequences.
If $F^{p,q}_U$, $F^{p,q}_V$ and $F^{p,q}_{U\cap V}$ are isomorphisms for all $p$, $q$, then so are $F^{p,q}_{U\cup V}$ for
all $p$, $q$ by the five-lemma. Hence $\mathcal{P}$  satisfies the Mayer-Vietoris condition.

To check the local condition, we first verify the following claim:

$(\lozenge)$ Assume that   $U$ is an $\mathcal{E}$-free open set of $X$ satisfying that $E_U$ is holomorphically trivial, then $\mathcal{P}(U)$ holds.

Without generality, assume that $\mathcal{E}|_U=\mathcal{O}_U$.
Suppose that $F$ is the general fiber of $E$  and $\varphi_U: U\times F\rightarrow E_U$ is a holomorphic trivialization. Let $pr_1$ and $pr_2$ be projections
from $U\times F$ to $U$ and $F$ respectively, which satisfy $\pi\circ \varphi_U=pr_1$.
Given a point $o\in U$, set $j_o:F\rightarrow U\times F$ as $f\mapsto (o,f)$.
Clearly, $pr_2\circ j_o=id_F$ and $i_o:=\varphi_U\circ j_o$ is the embedding $F\hookrightarrow E_U$ of the fiber $E_o\cong F$ over $o$ into $E_U$.
Set $e_i^{\prime}=(\varphi_U^{-1})^*pr_2^*i_o^*e_i\in H^{u_i,v_i}(E_U)$, $1\leq i\leq r$.
Then $i_o^*e_i^{\prime}=i_o^*e_i$ for any $i$.
Since $i_o^*e_1$, $\ldots$, $i_o^*e_r$ is linearly independent,  mapping $e_i$ to $e'_i$ for $1\leq i\leq r$ give an isomorphism
$\textrm{span}_{\mathbb{C}}\{e_1, \ldots, e_r\}\tilde{\rightarrow}\textrm{span}_{\mathbb{C}}\{e_1^{\prime}, \ldots, e_r^{\prime}\}$.
For any $p$, $q$, we have a commutative diagram
\begin{equation}\label{com}
\xymatrix{
 (H_{\Phi|_U}^{\bullet,\bullet}(U)\otimes_{\mathbb{C}} \textrm{span}_{\mathbb{C}}\{e_1, \ldots, e_r\})^{p,q}\ar[d]^{\cong}  \ar[dr]_{\cong}^{id\otimes i_o^*}&\\
 (H_{\Phi|_U}^{\bullet,\bullet}(U)\otimes_{\mathbb{C}} \textrm{span}_{\mathbb{C}}\{e_1^{\prime}, \ldots, e_r^{\prime}\})^{p,q}\ar[d]^{\pi^*(\bullet)\cup
 \bullet} \ar[r]^{\qquad id\otimes i_o^*}& (H_{\Phi|_U}^{\bullet,\bullet}(U)\otimes_{\mathbb{C}} H^{\bullet,\bullet}(F))^{p,q} \ar[d]^{pr_1^*(\bullet)\cup
 pr_2^*(\bullet)}\\
 H_{(\pi^{-1}\Phi)|_{E_U}}^{p,q}(E_U) \ar[r]_{\cong}^{\varphi_U^*}&  H_{\Phi|_U\times Y}^{p,q}(U\times F). }
\end{equation}
By the assumption, $i_o^*:\textrm{span}_{\mathbb{C}}\{e_1, \ldots, e_r\}\rightarrow H^{\bullet,\bullet}(F)$ is an isomorphism,
so is $i_o^*:\textrm{span}_{\mathbb{C}}\{e_1^{\prime}, \ldots,e_r^{\prime}\}\rightarrow H^{\bullet,\bullet}(F)$.
By (\ref{Kun-Dol-cpt}), $pr_1^*(\bullet)\cup pr_2^*(\bullet)$ is isomorphic,  so is $\pi^*(\bullet)\cup \bullet$ in (\ref{com}).
Mapping $(\alpha_1,\ldots,\alpha_r)$ to $\sum\limits_{i=1}^r\alpha_i\otimes e_i$ gives a morphism
\begin{equation}\label{id}
\bigoplus\limits_{i=1}^{r}H_{\Phi|_U}^{p-u_i,q-v_i}(U)\rightarrow (H_{\Phi|_U}^{\bullet,\bullet}(U)\otimes_{\mathbb{C}}
\textrm{span}_{\mathbb{C}}\{e_1, \ldots, e_r\})^{p,q},
\end{equation}
which is clearly isomorphic.
Then $F^{p,q}_U$ is the composition of (\ref{id}) and the  two vertical maps in the first column of  (\ref{com}), hence an isomorphism.
We proved $(\lozenge)$.
Let $\mathfrak{U}$ be an $\mathcal{E}$-free basis of the topology of $X$ such that $E_U$ is holomorphically trivial for any $U\in\mathfrak{U}$.
For $U_1$, $\ldots$, $U_l\in\mathfrak{U}$,  $E_{U_1\cap\ldots\cap U_l}$ is holomorphically trivial, then $(\lozenge)$ asserts that
$\mathcal{P}(\bigcap\limits_{i=1}^lU_i)$ is an isomorphism. Hence $\mathcal{P}$ satisfies the local condition.

We complete the proof.
\end{proof}

\subsection{Projective bundle formulae}
We successively define $P_{r-1}$, $P_{r-2}$, $P_{r-3}$, ..., $P_1$, $P_0$  by recursion relations
\begin{equation}\label{Polynomial-ind}
P_i(T_1,....,T_{r-1})= \left\{
 \begin{array}{ll}
(-1)^r\sum\limits_{k=1}^{r-1-i}T_{k}P_{k+i}(T_1,....,T_{r-1}),&~0\leq i<r-1\\
 &\\
 (-1)^{r-1},&~i=r-1,
 \end{array}
 \right.
\end{equation}
for $0\leq i\leq r-1$. For example,
\begin{displaymath}
P_{r-1}=(-1)^{r-1}, \mbox{ }P_{r-2}=-T_1, \mbox{ }P_{r-3}=(-1)^{r-1}T^2_1-T_2,  \mbox{ }...
\end{displaymath}
Clearly, $P_{i}(T_1,....,T_{r-1})\in \mathbb{Z}[T_1,....,T_{r-1}]$.
By (\ref{Polynomial-ind}), it is  easy to prove the following two lemmata by the induction.
We will only give the details of the proof of Lemma \ref{poly3}.
\begin{lem}\label{poly2}
Suppose that
\begin{displaymath}
P_{i}(T_1,....,T_{r-1})=\sum_{d_1,\cdots,d_{r-1}\geq 0} a_{i,d_1,\cdots,d_{r-1}}T^{d_1}_1...T^{d_{r-1}}_{r-1}.
\end{displaymath}
For any nonzero  $a_{i,d_1,\cdots,d_{r-1}}$, $\sum\limits_{k=1}^{r-1}kd_k=r-1-i$.
\end{lem}

\begin{lem}\label{poly3}
Set $T_0=(-1)^{r-1}$. For $k\in\{0, 1, ..., r-1\}$, put
\begin{displaymath}
H_{k}(T_1,...,T_{r-1})=\sum_{i=r-1}^{r-1+k}T_{i-(r-1)}P_{i-k}(T_1,...,T_{r-1}).
\end{displaymath}
Then
\begin{displaymath}\label{Kronecker}
H_{k}(T_1,...,T_{r-1})= \left\{
 \begin{array}{ll}
1,&~k=0\\
 &\\
 0,&~1\leq k\leq r-1.
 \end{array}
 \right.
\end{displaymath}
\end{lem}
\begin{proof}
For $k=0$, the lemma holds clearly. For $l>0$, assume that the lemma holds for any $k<l$ . Then
\begin{displaymath}
\begin{aligned}
H_{l}
=&(-1)^{r-1}T_{l}+(-1)^r\sum_{i=r-1}^{r-2+l}T_{i-(r-1)}\left(\sum_{s=1}^{r-1-i+l}T_{s}P_{s+i-l}\right)\quad\mbox{ }( \mbox{by (\ref{Polynomial-ind})})\\
=&(-1)^{r-1}T_{l}+(-1)^r\sum_{s=1}^{l}T_{s}\left(\sum_{i=r-1}^{r-1+l-s}T_{i-(r-1)}P_{i-(l-s)}\right)\quad\mbox{ }(\mbox{exchange sums})\\
=&(-1)^{r-1}T_{l}+(-1)^r\sum_{s=1}^{l}T_{s}H_{l-s}\qquad\qquad\qquad\qquad\qquad\quad\mbox{   }(\mbox{by the definition of $H_{k}$})\\
=&0\qquad\qquad\qquad\qquad \qquad\qquad \qquad \qquad\qquad\qquad\qquad \quad\mbox{ }\mbox{ }(\mbox{by the inductive assumption}).
\end{aligned}
\end{displaymath}
We complete the proof.
\end{proof}

Suppose that $\pi:\mathbb{P}(E)\rightarrow X$ is the projective vector bundle associated to a holomorphic bundle $E$ of rank $r$ over a complex manifold $X$,
$\mathcal{E}$ is a locally free sheaf  of $\mathcal{O}_X$-modules of finite rank
and $\Phi$ is a paracompactifying family of supports.
Let $t\in \mathcal{A}^{1,1}({\mathbb{P}(E)})$ be a Chern form of the universal line bundle $\mathcal{O}_{\mathbb{P}(E)}(-1)$ over ${\mathbb{P}(E)}$.
Notice that $dt=0$, i.e., $\partial t=\bar{\partial} t=0$.
The morphism $\sum\limits_{i=0}^{r-1}\pi^*(\bullet)\wedge t^i$ defines
\begin{equation}\label{proj-bundle1-cpt-form}
\bigoplus\limits_{i=0}^{r-1}\Gamma_\Phi\left(X,\mathcal{E}\otimes \mathcal{A}_X^{\bullet-i,\bullet-i}\right)\rightarrow
\Gamma_{\pi^{-1}\Phi}\left(\mathbb{P}(E),\pi^*\mathcal{E}\otimes\mathcal{A}_{\mathbb{P}(E)}^{\bullet,\bullet}\right).
\end{equation}
By Lemma \ref{poly2},
\begin{displaymath}
G_{t,\Phi}^{-i}(\bullet)=\sum_{j=0}^{r-1-i}P_{i+j}(\pi_*t^r,\mbox{ }....,\mbox{ }\pi_*t^{2r-2})\wedge\pi_*(t^j\wedge \bullet)
\end{displaymath}
defines
$\Gamma_{\pi^{-1}\Phi}\left(\mathbb{P}(E),\pi^*\mathcal{E}\otimes\mathcal{A}_{\mathbb{P}(E)}^{\bullet,\bullet}\right)\rightarrow
\Gamma_\Phi\left(X,\mathcal{E}\otimes \mathcal{A}_X^{\bullet-i,\bullet-i}\right)$
for $0\leq i\leq r-1$.
The morphism
$(G_{t,\Phi}^{0}(\bullet),\mbox{ }G_{t,\Phi}^{-1}(\bullet),\mbox{ }...,\mbox{ }G_{t,\Phi}^{-r+1}(\bullet))$
defines
\begin{equation}\label{proj-bundle2-cpt-form}
\Gamma_{\pi^{-1}\Phi}\left(\mathbb{P}(E),\pi^*\mathcal{E}\otimes\mathcal{A}_{\mathbb{P}(E)}^{\bullet,\bullet}\right)\rightarrow
\bigoplus\limits_{i=0}^{r-1}\Gamma_\Phi\left(X,\mathcal{E}\otimes \mathcal{A}_X^{\bullet-i,\bullet-i}\right).
\end{equation}
Denote (\ref{proj-bundle1-cpt-form}) and (\ref{proj-bundle2-cpt-form}) by  $\mu_{\Phi}^{\mathcal{E}}$ and  $\tau_{\Phi}^{\mathcal{E}}$ respectively.
Now we prove
\begin{equation}\label{inverse}
\tau_{\Phi}^{\mathcal{E}}\circ\mu_{\Phi}^{\mathcal{E}}=id.
\end{equation}
Clearly, $\pi_*t^i=0$, for $0\leq i\leq r-2$.
Moreover, $\pi_*t^{r-1}=(-1)^{r-1}$. Actually, $\pi_*t^{r-1}$ is a $d$-closed smooth $0$-form on $X$, hence a constant.  For any $x\in X$,
\begin{displaymath}
\pi_*t^{r-1}= \int_{\mathbb{P}(E_x)}t^{r-1}|_{\mathbb{P}(E_x)}
=\int_{\mathbb{P}^{r-1}}c_1(\mathcal{O}_{\mathbb{P}^{r-1}}(-1))^{r-1}=(-1)^{r-1}.
\end{displaymath}
For $\alpha^{p-i,q-i}\in \Gamma_\Phi\left(X,\mathcal{E}\otimes \mathcal{A}_X^{\bullet-i,\bullet-i}\right)$, $0\leq i\leq r-1$, we have
\begin{equation}\label{b}
\pi_*(t^j\wedge \mu_{\Phi}^{\mathcal{E}}(\alpha^{p,q},\alpha^{p-1,q-1},...,\alpha^{p-r+1,q-r+1}))=\sum_{l=r-1-j}^{r-1}\pi_*t^{j+l}\wedge\alpha^{p-l,q-l}
\end{equation}
by (\ref{pro-formula1}).
Then
\begin{displaymath}
\begin{aligned}
&G_{t,\Phi}^{-i}\circ\mu_{\Phi}^{\mathcal{E}}(\alpha^{p,q},\alpha^{p-1,q-1},...,\alpha^{p-r+1,q-r+1})\\
=&\sum_{j=0}^{r-1-i}\sum_{l=r-1-j}^{r-1}P_{i+j}(\pi_*t^r,....,\pi_*t^{2r-2})\wedge\pi_*t^{j+l}\wedge\alpha^{p-l,q-l}\qquad\mbox{ }( \mbox{by (\ref{b})})\\
=&\sum_{l=i}^{r-1}\left[\sum_{j=r-1-l}^{r-1-i}P_{i+j}(\pi_*t^r,....,\pi_*t^{2r-2})\wedge\pi_*h^{j+l}\right]\wedge\alpha^{p-l,q-l}\quad\mbox{ }\mbox{ }(\mbox{exchange
sums})\\
=&\sum_{l=i}^{r-1}H_{l-i}(\pi_*t^r,....,\pi_*t^{2r-2})\wedge\alpha^{p-l,q-l}\qquad\qquad\qquad\qquad\qquad\mbox{   }(\mbox{by the definition of $H_{k}$})\\
=&\alpha^{p-l,q-l},\qquad\qquad\qquad\qquad\qquad\qquad\qquad\qquad\qquad\qquad\qquad\mbox{ }\mbox{ }( \mbox{by Lemma \ref{poly3}})
\end{aligned}
\end{displaymath}
i.e.,
\begin{equation}\label{composit-projection}
G_{t,\Phi}^{-i}\circ\mu_{\Phi}^{\mathcal{E}}=pr_i
\end{equation}
for any $0\leq i\leq r-1$, where $pr_i:\bigoplus\limits_{i=0}^{r-1}\Gamma_\Phi\left(X,\mathcal{E}\otimes \mathcal{A}_X^{\bullet-i,\bullet-i}\right)\rightarrow
\Gamma_\Phi\left(X,\mathcal{E}\otimes \mathcal{A}_X^{\bullet-i,\bullet-i}\right)$ is the $i$-th projection.
So $\tau_{\Phi}^{\mathcal{E}}\circ\mu_{\Phi}^{\mathcal{E}}=id$.

Denote by $h\in H^{1,1}({\mathbb{P}(E)})$ the Dolbeault class of  $t$ and denote by $G_{h,\Phi}^{-i}(\bullet)$
the morphism
$H_{\pi^{-1}\Phi}^{\bullet,\bullet}(\mathbb{P}(E),\pi^*\mathcal{E})\rightarrow H_{\Phi}^{\bullet-i,\bullet-i}(X,\mathcal{E})$)
induced by $G_{t,\Phi}^{-i}(\bullet)$ for $0\leq i\leq r-1$.
Denote by   $\mu_{Dol,\Phi}^{\mathcal{E}}$ and $\tau_{Dol,\Phi}^{\mathcal{E}}$ the morphisms on twisted Dolbeault cohomologies induced by
$\mu_{\Phi}^{\mathcal{E}}$  and  $\tau_{\Phi}^{\mathcal{E}}$ respectively.
We have
\begin{prop}\label{proj-bun}
$\mu_{Dol,\Phi}^{\mathcal{E}}$ and  $\tau_{Dol,\Phi}^{\mathcal{E}}$ are inverse isomorphisms.
\end{prop}
\begin{proof}
For every $x\in X$, $1$, $h$,\ldots, $h^{r-1}$ restricted to the fibre $\pi^{-1}(x)=\mathbb{P}(E_x)$ freely linearly generate $H^{\bullet,\bullet}(\mathbb{P}(E_x))$.  By
Theorem \ref{L-H1},  $\mu_{Dol,\Phi}^{\mathcal{E}}$ is an isomorphism.
By (\ref{inverse}), we easily conclude it.
\end{proof}

\subsection{Blow-up formulae}

Let $X$ be a complex manifold and $i:Y\rightarrow X$ the inclusion of a  complex submanifold $Y$ into $X$.
Suppose that $\Phi$ is a paracompactifying family of supports on $X$.
For any $p$, $q$, set
$\mathcal{F}_{X,Y}^{p,q}=\textrm{ker}(\mathcal{A}_X^{p,q}\rightarrow i_*\mathcal{A}_Y^{p,q})$.
There is an exact sequence of sheaves
\begin{equation}\label{relative}
0\rightarrow\mathcal{F}_{X,Y}^{p,q}\rightarrow\mathcal{A}_X^{p,q}\rightarrow i_*\mathcal{A}_Y^{p,q}\rightarrow 0
\end{equation}
for any $p$ (\cite[Sect. 4.2]{RYY2} or \cite[Sect. 4]{M2}).
Define
$\mathcal{F}_{X,Y}^p=\textrm{ker}(\bar{\partial}:\mathcal{F}_{X,Y}^{p,0}\rightarrow\mathcal{F}_{X,Y}^{p,1})$.
There is an exact sequence
\begin{equation}\label{resolution}
\xymatrix{
0\ar[r] &\mathcal{F}_{X,Y}^p\ar[r]^{i} &\mathcal{F}_{X,Y}^{p,0}\ar[r]^{\bar{\partial}}
&\mathcal{F}_{X,Y}^{p,1}\ar[r]^{\mbox{ }\bar{\partial}}&\cdots\ar[r]^{\bar{\partial}\mbox{ }}&\mathcal{F}_{X,Y}^{p,n}\ar[r]&0,
}
\end{equation}
see \cite[Sect. 4.2]{RYY2} or \cite[Sect. 4]{M2}.
Since $\mathcal{F}_{X,Y}^p$ is a sheaf of $\mathcal{C}^\infty_X$-modules, it is $\Phi$-soft by \cite[II. 9.16]{Br},
so (\ref{resolution}) is a resolution of $\Phi$-soft sheaves of $\mathcal{F}_{X,Y}^p$.

Suppose that $\mathcal{E}$ is a locally free sheaf  of $\mathcal{O}_X$-modules of finite rank on $X$. We get an exact sequence of sheaves
\begin{displaymath}
\xymatrix{
 0\ar[r]&\mathcal{E}\otimes\mathcal{F}_{X,Y}^{p,q} \ar[r]& \mathcal{E}\otimes\mathcal{A}_{X}^{p,q} \ar[r]^{}& i_*(i^*\mathcal{E}\otimes\mathcal{A}_{Y}^{p,q})
 \ar[r]& 0}
\end{displaymath}
by  (\ref{relative}) and the projection formula of sheaves.
Since $\mathcal{E}\otimes\mathcal{F}_{X,Y}^{p,q}$ is $\Phi$-soft,
\begin{equation}\label{relative-exact}
\small{\xymatrix{
 0\ar[r]&\Gamma_\Phi(X,\mathcal{E}\otimes\mathcal{F}_{X,Y}^{p,q}) \ar[r]& \Gamma_\Phi(X,\mathcal{E}\otimes\mathcal{A}_{X}^{p,q}) \ar[r]^{i^*}&
 \Gamma_{\Phi|_Y}(Y,i^*\mathcal{E}\otimes\mathcal{A}_{Y}^{p,q}) \ar[r]& 0}}
\end{equation}
is exact.

Go back to the cases of complex blow-ups.
Suppose that  $\pi:\widetilde{X}\rightarrow X$ is the complex blow-up of a complex manifold $X$ along a complex submanifold $Y$ of complex codimension $r$ with the exceptional divisor $E$.
By \cite[Lemma 4.2]{M2},
\begin{displaymath}
R^q\pi_*(\pi^*\mathcal{E}\otimes\mathcal{F}_{\widetilde{X},E}^{p})=\mathcal{E}\otimes R^q\pi_*\mathcal{F}^p_{\widetilde{X},E}=\left\{
 \begin{array}{ll}
\mathcal{E}\otimes\mathcal{F}^p_{X,Y},&~q=0,\\
 &\\
 0,&~q\geq1.
 \end{array}
 \right.
\end{displaymath}
By \cite[IV. 6.1]{Br}, $\pi^*$ induces an isomorphism
\begin{displaymath}
H_{\Phi}^q(X,\mathcal{E}\otimes\mathcal{F}_{X,Y}^p)\cong H_{\pi^{-1}\Phi}^q(\widetilde{X},\pi^*\mathcal{E}\otimes\mathcal{F}_{\widetilde{X},E}^p).
\end{displaymath}
For a given $p$, we have a commutative diagram of exact sequences of complexes
\begin{displaymath}
\small{\xymatrix{
 0\ar[r]&\Gamma_\Phi(X,\mathcal{E}\otimes\mathcal{F}_{X,Y}^{p,\bullet})\ar[d]^{\pi^*} \ar[r]&
 \Gamma_\Phi(X,\mathcal{E}\otimes\mathcal{A}_{X}^{p,\bullet})\ar[d]^{\pi^*} \ar[r]^{i_Y^*}&
 \Gamma_{\Phi|_Y}(Y,i_Y^*\mathcal{E}\otimes\mathcal{A}_{Y}^{p,\bullet}) \ar[d]^{(\pi|_E)^*}\ar[r]& 0\\
 0\ar[r]&\Gamma_{\pi^{-1}\Phi}(\widetilde{X},\pi^*\mathcal{E}\otimes\mathcal{F}_{\widetilde{X},E}^{p,\bullet})    \ar[r]^{}&
 \Gamma_{\pi^{-1}\Phi}(\widetilde{X},\pi^*\mathcal{E}\otimes\mathcal{A}_{\widetilde{X}}^{p,\bullet})  \ar[r]^{i_E^*\quad} &
 \Gamma_{(\pi^{-1}\Phi)|_E}(E,i_E^*\pi^*\mathcal{E}\otimes\mathcal{A}_{E}^{p,\bullet})    \ar[r]& 0. }}
\end{displaymath}
It induces a commutative diagram of long exact sequences
\begin{displaymath}
\tiny{\xymatrix{
    \cdots H_\Phi^q(X,\mathcal{E}\otimes\mathcal{F}_{X,Y}^{p}) \ar[d]^{\cong} \ar[r]& H_{\Phi}^{p,q}(X,\mathcal{E}) \ar[d]^{\pi^*}\ar[r]^{i_Y^*}&
    H_{\Phi|_Y}^{p,q}(Y,i_Y^*\mathcal{E})\ar[d]^{(\pi|_E)^*}\ar[r]&H_\Phi^{q+1}(X,\mathcal{E}\otimes\mathcal{F}_{X,Y}^{p})\ar[d]^{\cong}\cdots\\
 \cdots  H_{\pi^{-1}\Phi}^q(\widetilde{X},\pi^*\mathcal{E}\otimes\mathcal{F}_{\widetilde{X},E}^{p})\ar[r]&
 H_{\pi^{-1}\Phi}^{p,q}(\widetilde{X},\pi^*\mathcal{E})       \ar[r]^{i_E^*\quad}& H_{({\pi^{-1}\Phi})|_E}^{p,q}(E,i_E^*\pi^*\mathcal{E})     \ar[r] &
 H_{\pi^{-1}\Phi}^{q+1}(\widetilde{X},\pi^*\mathcal{E}\otimes\mathcal{F}_{\widetilde{X},E}^{p})\cdots.}}
\end{displaymath}
By  Proposition \ref{inj-surj},   $\pi^*$ is injective.  By Proposition \ref{proj-bun}, $(\pi|_E)^*$ is injective. By the snake-lemma, $i_E^*$
induces an isomorphism
$\textrm{coker}\pi^*\tilde{\rightarrow}\textrm{coker}(\pi|_E)^*$. We  get a commutative diagram of exact sequences
\begin{equation}\label{commutative1}
\xymatrix{
 0\ar[r]&H_\Phi^{p,q}(X,\mathcal{E})\ar[d]^{i_Y^*} \ar[r]^{\pi^*}& H_{\pi^{-1}\Phi}^{p,q}(\widetilde{X},\pi^*\mathcal{E})\ar[d]^{i_E^*} \ar[r]&
 \textrm{coker}\pi^* \ar[d]^{\cong}\ar[r]& 0\\
 0\ar[r]&H_{\Phi|_Y}^{p,q}(Y,i_Y^*\mathcal{E})       \ar[r]^{(\pi|_E)^*\quad}& H_{({\pi^{-1}\Phi})|_E}^{p,q}(E,i_E^*\pi^*\mathcal{E})   \ar[r]^{} &  \textrm{coker}
 (\pi|_E)^*    \ar[r]& 0, }
\end{equation}
for any $p$, $q$.

Denote (\ref{b-u-m1}) and (\ref{b-u-m2}) by $\psi_{Dol,\Phi}^{\mathcal{E}}$ and $\phi_{Dol,\Phi}^{\mathcal{E}}$ respectively.

\subsubsection{$\phi_{Dol,\Phi}^{\mathcal{E}}$ is an isomorphism}

By (\ref{pro-formula2}), $\pi_*\pi^*=id$, namely, the upper row of (\ref{commutative1}) is  a splitting sequence.
So $(\pi_*,i_E^*)$ gives an isomorphism
\begin{displaymath}
H_{\pi^{-1}\Phi}^{p,q}(\widetilde{X},\pi^*\mathcal{E})\tilde{\rightarrow} H_\Phi^{p,q}(X,\mathcal{E})\oplus \textrm{coker}(\pi|_E)^*.
\end{displaymath}
By Proposition \ref{proj-bun}, it is
\begin{displaymath}
H_{\pi^{-1}\Phi}^{p,q}(\widetilde{X},\pi^*\mathcal{E})\rightarrow H_\Phi^{p,q}(X,\mathcal{E})\oplus
\bigoplus_{i=1}^{r-1}H_{\Phi|_Y}^{p-i,q-i}(Y,i_Y^*\mathcal{E})
\end{displaymath}
\begin{equation}\label{rep0}
\quad\quad\quad\quad\quad\alpha\mapsto(\pi_*\alpha,\alpha^{p-1,q-1},\ldots,\alpha^{p-r+1,q-r+1}),
\end{equation}
where   $\alpha^{p-i,q-i}\in H_{\Phi|_Y}^{p-i,q-i}(Y,i_Y^*\mathcal{E})$ for  $0\leq i\leq r-1$ satisfy
\begin{displaymath}
i_E^*\alpha =\sum_{i=0}^{r-1}h^i\cup (\pi|_E)^*\alpha^{p-i,q-i}.
\end{displaymath}
By Proposition \ref{proj-bun}, $\alpha^{p-i,q-i}=G_{h,\Phi}^{-i}\circ i_E^*(\alpha)$ for  $0\leq i\leq r-1$.
So (\ref{rep0}) is just $\phi_{Dol,\Phi}^{\mathcal{E}}$.

\begin{rem}
On compact complex manifolds, the expression (\ref{rep0}) was  first obtained  in \cite[Sect. 5.2]{RYY2} for $\Phi=clt_X$.
The compactness of $X$ is necessary there, since  the finiteness of dimensions of cohomologies was used.
\end{rem}

\subsubsection{$\psi_{Dol,\Phi}^{\mathcal{E}}$ is an isomorphism}
\begin{lem}\label{key0}
Let $\pi:\widetilde{X}\rightarrow X$ be the complex blow-up of $X$ along a complex submanifold $Y$ with the exceptional divisor $E$.
Assume that $Y$ has a  holomorphically contractible neighborhood in $X$.
Then the composite map
\begin{displaymath}
\xymatrix{
H_{(\pi^{-1}\Phi)|_E}^{\bullet,\bullet}(E)\ar[r]^{i_{E*}} & H_{\pi^{-1}\Phi}^{\bullet+1,\bullet+1}(\widetilde{X})\ar[r]^{i_E^*}&
H_{(\pi^{-1}\Phi)|_E}^{\bullet+1,\bullet+1}(E)
}
\end{displaymath}
is $h\cup\bullet$, where  $h$ is the Dolbeault class of a Chern form of the universal line bundle $\mathcal{O}_E(-1)$ over $E$ and $i_E:E\rightarrow
\widetilde{X}$ is the inclusion.
\end{lem}
\begin{proof}
Let  $U$ be a neighborhood of  $Y$ with a holomorphic map $\tau: U\rightarrow Y$ such that $\tau\circ l_Y=id_Y$, where $l_Y:Y\rightarrow U$  is the inclusion.
Denote by $l_E:E\rightarrow \widetilde{U}$  and by $j:\widetilde{U}\rightarrow \widetilde{X}$ the inclusions.
Set $\widetilde{U}=\pi^{-1}(U)$.
By $($\ref{commutative1}$)$, $l_E^*$ induces a surjective map
\begin{displaymath}
H^{p,q}(\widetilde{U})\rightarrow H^{p,q}(E)/(\pi|_E)^*H^{p,q}(Y).
\end{displaymath}
Since $(\pi|_E)^*H^{p,q}(Y)=l_E^*(\pi|_{\widetilde{U}})^*\tau^*H^{p,q}(Y)\subseteq \textrm{Im} l_E^*$,
$l_E^*:H^{p,q}(\widetilde{U})\rightarrow H^{p,q}(E)$ is surjective.
For any $\alpha\in H^{p,q}(E)$, $\alpha=l_E^*\beta$ for some $\beta\in H^{p,q}(\widetilde{U})$.
Denote by $T_E$ the current on $\widetilde{U}$ defined by the integral along the divisor $E$. Clearly, $l_{E*}(1)=T_E$.
By the Lelong-Poincar\'{e} equation  (see \cite[p. 271 (13.2), (13.5)]{Dem}),
\begin{displaymath}
T_E=\frac{i}{2\pi}\Theta(\mathcal{O}_{\widetilde{U}}(E))+\bar{\partial}\partial T
\end{displaymath}
for some $T\in\mathcal{D}^{\prime 0,0}(\widetilde{U})$, where $\Theta(\mathcal{O}_{\widetilde{U}}(E))$ denotes a Chern form of $\mathcal{O}_{\widetilde{U}}(E)$.
By (\ref{pro-formula2}), $l_{E*}\alpha=l_{E*}l^*_E\beta=l_{E*}(1)\cup\beta$.
So
\begin{equation}\label{sel-int}
l_E^*l_{E*}\alpha=l_E^*[\frac{i}{2\pi}\Theta(\mathcal{O}_{\widetilde{U}}(E))]\cup\alpha
=h\cup\alpha,
\end{equation}
where  we used the fact that $\mathcal{O}_E(-1)=\mathcal{O}_{\widetilde{U}}(E)|_E$.

For any $\sigma\in H_{(\pi^{-1}\Phi)|_E}^{\bullet,\bullet}(E)=H_{(\pi|_E)^{-1}(\Phi|_Y)}^{\bullet,\bullet}(E)$, there exists $\alpha_i\in
H_{\Phi|_Y}^{\bullet-i,\bullet-i}(Y)$ for $0\leq i\leq r-1$ such that $\sigma=\sum\limits_{i=0}^{r-1}h^i\cup (\pi|_E)^*\alpha_i$ by Proposition \ref{proj-bun}.
Let $u_i\in \mathcal{A}_{\Phi|_Y}^{\bullet-i,\bullet-i}(Y)$ be a representative of $\alpha_i$.
Set $\Omega=j^{-1}\pi^{-1}\Phi\cap (\pi|_{\widetilde{U}})^{-1}\tau^{-1}(\Phi|_Y)$.
By Proposition \ref{inverse-paracompact} $(3)-(6)$, $\Omega$ is a paracompactifying family of supports on $\widetilde{U}$.
By Proposition \ref{inverse-paracompact} $(1)$ $(2)$,
\begin{equation}\label{pullsupports}
(\pi|_E)^{-1}(\Phi|_Y)=l_E^{-1}\left(j^{-1}\pi^{-1}\Phi\right)=l_E^{-1}\left((\pi|_{\widetilde{U}})^{-1}\tau^{-1}(\Phi|_Y)\right)=l_E^{-1}\Omega.
\end{equation}
Since $\textrm{supp} \left(t^i\wedge (\pi|_E)^*u_i\right)\in (\pi|_E)^{-1}(\Phi|_Y)$,
\begin{displaymath}
\textrm{supp} \left(l_{E*}(t^i\wedge (\pi|_E)^*u_i)\right)\in  j^{-1}\pi^{-1}\Phi\cap (\pi|_{\widetilde{U}})^{-1}\tau^{-1}(\Phi|_Y)=\Omega.
\end{displaymath}
Then
\begin{displaymath}
\begin{aligned}
l_E^*[l_{E*}(t^i\wedge (\pi|_E)^*u_i)]_{j^{-1}\pi^{-1}\Phi}
=&l_E^*[l_{E*}(t^i\wedge (\pi|_E)^*u_i)]_{\Omega}\qquad\qquad\qquad\quad\mbox{ }\mbox{ }\mbox{ }( \mbox{by (\ref{pullsupports}) and (\ref{sub0})})\\
=&l_E^*[l_{E*}(t^i\wedge (\pi|_E)^*u_i)]_{(\pi|_{\widetilde{U}})^{-1}\tau^{-1}(\Phi|_Y)}\quad\mbox{ }\mbox{ }\mbox{ }\mbox{ }( \mbox{by (\ref{pullsupports}) and (\ref{sub0})})\\
=&l_E^*[l_{E*}t^i\wedge (\pi|_{\widetilde{U}})^{\ast}\tau^{\ast}u_i]_{(\pi|_{\widetilde{U}})^{-1}\tau^{-1}(\Phi|_Y)}\qquad\mbox{ }( \mbox{by (\ref{pro-formula1})})\\
=&l_E^*\left(l_{E*}h^i\cup [(\pi|_{\widetilde{U}})^{\ast}\tau^{\ast}u_i]_{(\pi|_{\widetilde{U}})^{-1}\tau^{-1}(\Phi|_Y)}\right)\\
=&h^{i+1}\cup (\pi|_E)^*\alpha_i.\qquad\qquad\qquad\qquad\qquad\mbox{ }\mbox{ }\mbox{ }( \mbox{by (\ref{sel-int})})
\end{aligned}
\end{displaymath}
Since the support of $l_{E*}(t^i\wedge (\pi|_E)^*u_i)$ is closed in $\widetilde{X}$, $j_*l_{E*}(t^i\wedge (\pi|_E)^*u_i)$ is defined well and is just
$i_{E*}(t^i\wedge (\pi|_E)^*u_i)$.
By (\ref{push-pull-1}),
\begin{equation}\label{mid}
j^*i_{E*}(t^i\wedge (\pi|_E)^*u_i)=l_{E*}(t^i\wedge (\pi|_E)^*u_i).
\end{equation}
Then
\begin{displaymath}
\begin{aligned}
i_E^*i_{E*}\left(h^i\cup (\pi|_E)^*\alpha_i\right)=&i_E^*i_{E*}[t^i\wedge (\pi|_E)^*u_i]_{(\pi^{-1}\Phi)|_{E}}\\
=&l_E^*[j^*i_{E*}(t^i\wedge (\pi|_E)^*u_i)]_{j^{-1}\pi^{-1}\Phi}\\
=&l_E^*[l_{E*}(t^i\wedge (\pi|_E)^*u_i)]_{j^{-1}\pi^{-1}\Phi}\qquad\qquad\mbox{ }( \mbox{by (\ref{mid})})\\
=&h^{i+1}\cup (\pi|_E)^*\alpha_i.
\end{aligned}
\end{displaymath}
So $i_E^*i_{E*}\sigma=\sum\limits_{i=0}^{r-1}h^{i+1}\cup (\pi|_E)^*\alpha_i=h\cup \sigma$.
\end{proof}

With the similar proof of \cite[Proposition 4.5]{M2}, we have
\begin{lem}\label{special}
Suppose that $Y$ has a  holomorphically contractible neighborhood in $X$. Then  $\psi_{Dol,\Phi}^{\mathcal{O}_{X}}$ is an isomorphism.
\end{lem}
\begin{proof}
Suppose that $\pi^*\alpha+\sum\limits_{i=1}^{r-1}i_{E*}\left(h^{i-1}\cup(\pi|_E)^*\beta_i\right)=0$, where $\alpha\in H_{\Phi}^{\bullet,\bullet}(X)$ and
$\beta_i\in H_{\Phi|_Y}^{\bullet-i,\bullet-i}(Y)$ for $1\leq i\leq r-1$.
Pull it back  by $i_E^*$, we get
\begin{displaymath}
(\pi|_E)^*i_Y^*\alpha+\sum_{i=1}^{r-1}h^{i}\cup(\pi|_E)^*\beta_i=0
\end{displaymath}
by Lemma \ref{key0}.
By Proposition \ref{proj-bun}, $\beta_i=0$ for all $i$.
So $\pi^*\alpha=0$.
By Proposition \ref{inj-surj}, $\pi^*$ is injective, which implies that $\alpha=0$.
Then $\psi_{Dol,\Phi}^{\mathcal{O}_{X}}$ is injective.

For any $\gamma\in H_{\pi^{-1}\Phi}^{\bullet,\bullet}(\widetilde{X})$, by Proposition \ref{proj-bun}, there exist $\beta_i\in
H_{\Phi|_Y}^{\bullet-i,\bullet-i}(Y)$ for $0\leq i\leq r-1$ such that $i_E^*\gamma=\sum\limits_{i=0}^{r-1}h^i\cup(\pi|_E)^*\beta_i$.
By Lemma \ref{key0},
\begin{displaymath}
i_E^*\left[\gamma-\sum_{i=1}^{r-1}i_{E*}\left(h^{i-1}\cup(\pi|_E)^*\beta_i\right)\right]=(\pi|_E)^*\beta,
\end{displaymath}
which is zero in $\textrm{coker}(\pi|_E)^*$.
By  (\ref{commutative1}),
\begin{displaymath}
\gamma-\sum_{i=1}^{r-1}i_{E*}\left(h^{i-1}\cup(\pi|_E)^*\beta_i\right)=\pi^*\alpha,
\end{displaymath}
for some $\alpha\in H_{\Phi}^{\bullet,\bullet}(X)$.
So $\psi_{Dol,\Phi}^{\mathcal{O}_{X}}$  is surjective.

We complete the proof.
\end{proof}

Set
\begin{displaymath}
\mathcal{F}^{\bullet,\bullet}=(\mathcal{E}\otimes\mathcal{A}_X^{\bullet,\bullet})\oplus
\bigoplus_{i=1}^{r-1}i_{Y*}(i_Y^*\mathcal{E}\otimes\mathcal{A}_Y^{\bullet,\bullet})[-i,-i].
\end{displaymath}
and $\mathcal{G}^{\bullet,\bullet}=\pi^*\mathcal{E}\otimes\mathcal{D}_{\widetilde{X}}^{\prime \bullet,\bullet}$.
Let $t\in\mathcal{A}^{1,1}(E)$ be a Chern form of the universal line bundle $\mathcal{O}_E(-1)$ over $E$.
Set $\widetilde{U}=\pi^{-1}(U)$ for any open subset $U\subseteq X$.
Notice that $(\pi^{-1}\Phi)|_{\widetilde{U}}=(\pi|_{\widetilde{U}})^{-1}(\Phi|_U)$.
Define a morphism
\begin{displaymath}
\small{\begin{aligned}
&\Gamma_{\Phi|_U}(U,\mathcal{F}^{\bullet,\bullet})=\Gamma_{\Phi|_U}(U,\mathcal{E}\otimes\mathcal{A}_X^{\bullet,\bullet})\oplus \bigoplus_{i=1}^{r-1}\Gamma_{\Phi|_{Y\cap
U}}(Y\cap U,i_Y^*\mathcal{E}\otimes\mathcal{A}_Y^{\bullet-i,\bullet-i})\\
\rightarrow&
\Gamma_{(\pi|_{\widetilde{U}})^{-1}(\Phi|_U)}(\widetilde{U},\pi^*\mathcal{E}\otimes\mathcal{D}_{\widetilde{X}}^{\prime \bullet,\bullet})
=\Gamma_{(\pi^{-1}\Phi)|_{\widetilde{U}}}(\widetilde{U},\mathcal{G}^{\bullet,\bullet})
\end{aligned}}
\end{displaymath}
of double complexes as
\begin{displaymath}
\psi_{\Phi|_U}^{\mathcal{E}|_{U}}=(\pi|_{\widetilde{U}})^*+\sum\limits_{i=1}^{r-1}i_{E\cap \widetilde{U}*}\circ (t^{i-1}|_{E\cap \widetilde{U}}\wedge)\circ
(\pi|_{E\cap \widetilde{U}})^*,
\end{displaymath}
where $i_{E\cap \widetilde{U}}:E\cap \widetilde{U}\rightarrow \widetilde{U}$ is the inclusion.
On cohomologies, it induces a morphism
\begin{displaymath}
\small{\begin{aligned}
L^{\bullet,\bullet}(U):=H_{\Phi|_U}^{\bullet,\bullet}(U,\mathcal{E})\oplus \bigoplus_{i=1}^{r-1}H_{\Phi|_{Y\cap U}}^{\bullet-i,\bullet-i}(Y\cap
U,i_Y^*\mathcal{E})\rightarrow
K^{\bullet,\bullet}(\widetilde{U}):=H_{(\pi^{-1}\Phi)|_{\widetilde{U}}}^{\bullet,\bullet}(\widetilde{U},\pi^*\mathcal{E}),
\end{aligned}}
\end{displaymath}
which is just $\psi_{Dol,\Phi|_U}^{\mathcal{E}|_{U}}$.
Briefly set $F_U^{\bullet,\bullet}= \psi_{Dol,\Phi|_U}^{\mathcal{E}|_{U}}$.
Denote by $\mathcal{P}(U)$ the statement that $F_U^{\bullet,\bullet}$ is an  isomorphism.
The theorem is equivalent to say that $\mathcal{P}(X)$ holds.

Fisrt, we fix some notations.
For the inclusion $j:U\rightarrow V$ of open sets in $X$,
denote by $\tilde{j}:\widetilde{U}\hookrightarrow\widetilde{V}$ and  $j^{\prime}:U\cap Y\rightarrow V\cap Y$
the corresponding inclusions.
Denote by $J=(j_*,j^{\prime}_*,\ldots,j^{\prime}_*):\Gamma_{\Phi|_U}(U,\mathcal{F}^{\bullet,\bullet})\rightarrow \Gamma_{\Phi|_V}(V,\mathcal{F}^{\bullet,\bullet})$ the
extension by zero.
Now, we go back to the proof. Fix an integer $p$.
For any open sets $U$, $V\subseteq X$,
let $j_1:U\cap V\rightarrow U$, $j_2:U\cap V\rightarrow V$, $j_3:U\rightarrow U\cup V$, $j_4:V\rightarrow U\cup V$ be inclusions.
By Propositions \ref{short-exact} and \ref{com1},
there is a commutative diagram of exact sequences of complexes
\begin{displaymath}
\tiny{\xymatrix{
0\ar[r]&\Gamma_{\Phi|_{U\cap V}}(U\cap V,\mathcal{F}^{p,\bullet})\ar[d]^{\psi_{\Phi|_{U\cap V}}^{\mathcal{E}|_{U\cap V}}}\quad \ar[r]^{(J_1,J_2)\quad\quad} &
\quad\Gamma_{\Phi|_{U}}(U,\mathcal{F}^{p,\bullet})\oplus
\Gamma_{\Phi|_V}(V,\mathcal{F}^{p,\bullet})\ar[d]^{(\psi_{\Phi|_U}^{\mathcal{E}|_{U}},\psi_{\Phi|_V}^{\mathcal{E}|_{V}})}\quad\ar[r]^{\quad\quad J_3-J_4}&
\quad\Gamma_{\Phi|_{U\cup V}}(U\cup V,\mathcal{F}^{p,\bullet}) \ar[d]^{\psi_{\Phi|_{U\cup V}}^{\mathcal{E}|_{U\cup V}}}\ar[r]& 0\\
0\ar[r]&\Gamma_{(\pi^{-1}\Phi)|_{\widetilde{U}\cap \widetilde{V}}}(\widetilde{U}\cap \widetilde{V},\mathcal{G}^{p,\bullet})     \quad
\ar[r]^{(\tilde{j}_{1*},\tilde{j}_{2*})\qquad\qquad}& \quad \Gamma_{(\pi^{-1}\Phi)|_{\widetilde{U}}}(\widetilde{U},\mathcal{G}^{p,\bullet})\oplus \Gamma_{(\pi^{-1}\Phi)|_{\widetilde{V}}}(
\widetilde{V},\mathcal{G}^{p,\bullet})    \quad\ar[r]^{\qquad\qquad \tilde{j}_{3*}-\tilde{j}_{4*}} & \quad\Gamma_{(\pi^{-1}\Phi)|_{\widetilde{U}\cup
\widetilde{V}}}(\widetilde{U}\cup
\widetilde{V},\mathcal{G}^{p,\bullet})     \ar[r]& 0.}}
\end{displaymath}
Therefore, we have a commutative diagram
\begin{displaymath}
\small{\xymatrix{
    \cdots L^{p,q}(U\cap V) \ar[d]^{F^{p,q}_{U\cap V}} \ar[r]& L^{p,q}(U)\oplus L^{p,q}(V)
    \ar[d]^{(F^{p,q}_U,F^{p,q}_V)}\ar[r]&  L^{p,q}(U\cup V)\ar[d]^{F^{p,q}_{U\cup V}}\ar[r]&L^{p,q+1}(U\cap V) \ar[d]^{F^{p,q+1}_{U\cap V}}\ar[r]&\cdots\\
 \cdots K^{p,q}(\widetilde{U}\cap \widetilde{V})\ar[r]&K^{p,q}(\widetilde{U})\oplus
 K^{p,q}(\widetilde{V})       \ar[r]&
 K^{p,q}(\widetilde{U}\cup \widetilde{V})     \ar[r] &K^{p,q+1}(\widetilde{U}\cap \widetilde{V})\ar[r]&\cdots}}
\end{displaymath}
of long exact sequences. If $F^{p,q}_U$, $F^{p,q}_V$ and $F^{p,q}_{U\cap V}$ are isomorphisms for all $p$, $q$, then  $F^{p,q}_{U\cup V}$ are also isomorphisms for all $p$, $q$ by the five-lemma.
 Thus $\mathcal{P}$  satisfies the Mayer-Vietoris condition in Lemma \ref{glued}.
Let $\mathfrak{U}$ be an $\mathcal{E}$-free basis of the topology of $X$ such that every $U\in \mathfrak{U}$ is Stein.
For $U_1$, $\ldots$, $U_l\in\mathfrak{U}$,  $\bigcap\limits_{i=1}^lU_i$ is $\mathcal{E}$-free and $Y\cap \bigcap\limits_{i=1}^lU_i$ is Stein.
By \cite[Theorem 3.3.3]{Fo},  any Stein complex submanifold  has a holomorphically contractible neighborhood.
By Lemma \ref{special}, $F_{U_1\cap\ldots\cap U_l}^{\bullet,\bullet}=\psi_{Dol,\Phi|_{U_1\cap\ldots\cap U_l}}^{\mathcal{E}|_{U_1\cap\ldots\cap U_l}}$ is an isomorphism, so $\mathcal{P}$ satisfies the local condition in Lemma \ref{glued}.
Obviously, $\mathcal{P}$ satisfies the disjoint condition in Lemma \ref{glued}.
By Lemma \ref{glued}, $\mathcal{P}(X)$ holds.

We complete the proof.

\begin{rem}
In general, $clt_X|_U\neq clt_U$  for an open set $U\subseteq X$.
We used a type of the Mayer-Vietoris sequences (Proposition \ref{short-exact}) here
and used another type of the ones in  the proof of \cite[Theorem 1.2]{M2} for $\Phi=clt_X$,
where the later one seems difficult to be used to prove the case with supports in a paracompactifying family $\Phi$.
\end{rem}

\subsubsection{Relationship of \emph{(\ref{b-u-m1})} and \emph{(\ref{b-u-m2})}}
A natural question is:
\begin{quest}\label{prob.1}
\emph{Are  $\psi_{Dol,\Phi}^{\mathcal{E}}$ and $\phi_{Dol,\Phi}^{\mathcal{E}}$  inverse to each other?}
\end{quest}

This question has an affirmative answer in the following case.
\begin{prop}\label{relation}
Suppose that
\begin{equation}\label{key2}
i_E^*i_{E*}\sigma=h\cup\sigma
\end{equation}
holds for any $\sigma\in H_{(\pi^{-1}\Phi)|_E}^{\bullet,\bullet}(E,i_E^*\pi^*\mathcal{E})$.
Then  $\psi_{Dol,\Phi}^{\mathcal{E}}$ and  $\phi_{Dol,\Phi}^{\mathcal{E}}$  are inverse isomorphisms.
\end{prop}
\begin{proof}
For $\alpha^{p,q}\in H_{\Phi}^{p,q}(X,\mathcal{E})$ and $\beta^{p-i,q-i}\in H_{\Phi|_Y}^{p-i,q-i}(Y,i_Y^*\mathcal{E})$, $1\leq i\leq r-1$,
\begin{displaymath}
\begin{aligned}
&i_E^*(\psi_{Dol,\Phi}^{\mathcal{E}}(\alpha^{p,q},\beta^{p-1,q-1},\ldots,\beta^{p-r+1,q-r+1})\\
=&\mu_{Dol,\Phi}^{\mathcal{E}}(i_Y^*\alpha^{p,q},\beta^{p-1,q-1},\ldots,\beta^{p-r+1,q-r+1})
\end{aligned}
\end{displaymath}
by (\ref{key2}).
Hence
\begin{displaymath}
G_{h,\Phi}^{-i}\circ i_E^*\left(\psi_{Dol,\Phi}^{\mathcal{E}}(\alpha^{p,q},\beta^{p-1,q-1},\ldots,\beta^{p-r+1,q-r+1}\right)=\beta^{p-i,q-i}
\end{displaymath}
by (\ref{composit-projection}), which implies that $\phi_{Dol,\Phi}^{\mathcal{E}}\circ\psi_{Dol,\Phi}^{\mathcal{E}}=id$,
i.e.,  $\phi_{Dol,\Phi}^{\mathcal{E}}$  is the inverse isomorphism of $\psi_{Dol,\Phi}^{\mathcal{E}}$.
\end{proof}

\begin{prop}\label{key}
Let $\Phi$ be a paracompactifying family of supports on $X$.
Assume that  one of the following conditions is satisfied\emph{:}

$(1)$ $X$ and $Y$ are compact complex manifolds satisfying the $\partial\bar{\partial}$-lemma and $\Phi=clt_X$,

$(2)$ $Y$ has a  holomorphically contractible neighborhood in $X$,

$(3)$ $E$ has a  holomorphically contractible neighborhood in $\widetilde{X}$.\\
Then $i_E^*i_{E*}\sigma=h\cup\sigma$ for any $\sigma\in  H_{(\pi^{-1}\Phi)|_E}^{\bullet,\bullet}(E)$.
\end{prop}
\begin{proof}
If $X$ and $Y$ satisfy the $\partial\bar{\partial}$-lemma, so do $E=\mathbb{P}(N_{Y/X})$ and $\widetilde{X}$  (see \cite[Corollary 3]{ASTT}, \cite[Corollary
26]{St2} or \cite[Theorems 1.1, 1.2]{M4}).
Let $t\in \mathcal{A}^{1,1}(E)$ be a Chern form of $\mathcal{O}_E(-1)$.
Then $h=[t]_{Dol}$ and $c_1(\mathcal{O}_E(-1))=[t]_{dR}$.
Since $\mathcal{O}_E(-1)=\mathcal{O}_X(E)|_E$, $[t]_{dR}=[E]|_E$.
By Proposition \ref{key} (see Sect. 5.1), $i_E^*i_{E*}\sigma=[t]_{dR}\cup\sigma$ for any $\sigma\in H^\bullet(E,\mathbb{C})$.
The $\partial\bar{\partial}$-manifolds satisfy  Hodge decompositions,  which concludes the first case.
The second case is just Lemma \ref{key0}.
With some modification of the proof of Lemma \ref{key0}, we prove the third case as follows:

Let  $W$ be a neighborhood of  $E$ with a holomorphic map $\tau: W\rightarrow E$ such that $\tau\circ l_E=id_E$,  where $l_E:E\rightarrow W$  is the inclusion.
Denote by   $j:W\rightarrow \widetilde{X}$ the inclusion.
Let $u\in \Gamma_{(\pi^{-1}\Phi)|_E}(E,\mathcal{A}_E^{\bullet,\bullet})$ be any $\bar{\partial}$-closed form.
Set $\Omega=j^{-1}\pi^{-1}\Phi\cap \tau^{-1}((\pi^{-1}\Phi)|_E)$.
By Proposition \ref{inverse-paracompact}, $\Omega$ is a paracompactifying family of supports on $W$ and
\begin{equation}\label{pullsupports2}
(\pi^{-1}\Phi)|_E=l_E^{-1}\left(j^{-1}\pi^{-1}\Phi\right)=l_E^{-1}\left(\tau^{-1}((\pi^{-1}\Phi)|_E)\right)=l_E^{-1}\Omega.
\end{equation}
Since $\textrm{supp}u\in (\pi^{-1}\Phi)|_E$,
$\textrm{supp} \left(l_{E*}u\right)\in  j^{-1}\pi^{-1}\Phi\cap \tau^{-1}((\pi^{-1}\Phi)|_E)=\Omega$.
By (\ref{pullsupports2}) and (\ref{sub0}),
\begin{equation}\label{mid0}
l_E^*[l_{E*}u]_{j^{-1}\pi^{-1}\Phi}=l_E^*[l_{E*}u]_{\Omega}=l_E^*[l_{E*}u]_{\tau^{-1}((\pi^{-1}\Phi)|_E)}.
\end{equation}
By (\ref{pro-formula1}), $l_{E*}u=l_{E*}(1)\wedge \tau^{\ast}u$.
So
\begin{equation}\label{mid1}
l_E^*[l_{E*}u]_{\tau^{-1}((\pi^{-1}\Phi)|_E)}
=l_E^*\left([l_{E*}(1)]\cup [\tau^{\ast}u]_{\tau^{-1}((\pi^{-1}\Phi)|_E)}\right)
=h\cup [u]_{(\pi^{-1}\Phi)|_{E}}.
\end{equation}
Then
\begin{displaymath}
\begin{aligned}
i_E^*i_{E*}[u]_{(\pi^{-1}\Phi)|_{E}}=&l_E^*[j^*i_{E*}u]_{j^{-1}\pi^{-1}\Phi}\\
=&l_E^*[l_{E*}u]_{j^{-1}\pi^{-1}\Phi}\qquad\qquad\mbox{ }( \mbox{by (\ref{push-pull-1})})\\
=&h\cup [u]_{(\pi^{-1}\Phi)|_{E}}.\qquad\quad\mbox{ }\mbox{ }\mbox{ }\mbox{ }\mbox{ }( \mbox{by (\ref{mid0}) and (\ref{mid1})})
\end{aligned}
\end{displaymath}
\end{proof}

\subsection{Applications}
\subsubsection{Conjugate local systems}
Let $(V,\cdot)$ be a complex vector space with  the scalar multiplication $\cdot$.
The action $c\ast v=\bar{c}\cdot v$ for any $c\in \mathbb{C}$ and $v\in V$ defines the \emph{conjugate vector space}  $(V,\ast)$  of $(V,\cdot)$.
We shortly write $(V,\cdot)$ and $(V,\ast)$ as $V$ and $\overline{V}$ respectively.
Notice that $\overline{V}$ has the same underlying space with $V$ and the identity map  $V\rightarrow \overline{V}$ is a \emph{real} isomorphism.

Let $\mathcal{L}$ be a local system of $\mathbb{C}$-modules of finite rank on $X$.
Set $\Gamma(U,\overline{\mathcal{L}})=\overline{\Gamma(U,\mathcal{L})}$ for all open set $U\subseteq X$, which define a local system of $\mathbb{C}$-modules of
finite rank on $X$.
Suppose that a trivialization $\mathcal{L}|_U\tilde{\rightarrow} U\times \mathbb{C}^r$ is defined as $\sum\limits_{i=1}^rc_i\cdot v_i(x)\mapsto (x, c_1,\ldots,
c_r)$, where $v_1$, \ldots, $v_r$ is a basis of $\Gamma(U,\mathcal{L})$.
Then $\sum\limits_{i=1}^rc_i\ast v_i(x)\mapsto (x, c_1,\ldots, c_r)$ give a trivialization $\overline{\mathcal{L}}|_U\tilde{\rightarrow} U\times \mathbb{C}^r$.
For any open set $U\subseteq X$, $s\mapsto s$ for all $s\in \Gamma(U,\mathcal{L})$ define a real isomorphism
$\Gamma(U,\mathcal{L})\rightarrow\Gamma(U,\overline{\mathcal{L}})$, which gives an isomorphism $\mathcal{L}\tilde{\rightarrow} \overline{\mathcal{L}}$ of
\emph{real} local systems.

Suppose that $X$ is a  complex manifold.
Then $\mathcal{L}\otimes_{\underline{\mathbb{C}}_X}\mathcal{A}_X^{p,q}=(\mathcal{L}\otimes_{\underline{\mathbb{C}}_X}\mathcal{O}_X)\otimes_{\mathcal{O}_X}\mathcal{A}_X^{p,q}$ and
$\mathcal{L}\otimes_{\underline{\mathbb{C}}_X}\mathcal{D}_X^{\prime p,q}=(\mathcal{L}\otimes_{\underline{\mathbb{C}}_X}\mathcal{O}_X)\otimes_{\mathcal{O}_X}\mathcal{D}_X^{\prime p,q}$,
where $\mathcal{L}\otimes_{\underline{\mathbb{C}}_X}\mathcal{O}_X$ is a locally free sheaf of $\mathcal{O}_X$-modules.
For a holomorphic map $f:Y\rightarrow X$ of complex manifolds,
$f^{-1}\mathcal{L}\otimes_{\underline{\mathbb{C}}_X}\mathcal{A}_Y^{p,q}=f^*(\mathcal{L}\otimes_{\underline{\mathbb{C}}_X}\mathcal{O}_X)\otimes_{\mathcal{O}_Y}\mathcal{A}_Y^{p,q}$
and $f^{-1}\mathcal{L}\otimes_{\underline{\mathbb{C}}_X}\mathcal{D}_Y^{\prime
p,q}=f^*(\mathcal{L}\otimes_{\underline{\mathbb{C}}_X}\mathcal{O}_X)\otimes_{\mathcal{O}_Y}\mathcal{D}_Y^{\prime p,q}$.
Mapping $\sum\limits_{i=1}^rv_i\otimes \alpha_i$ to $\sum\limits_{i=1}^rv_i\otimes \bar{\alpha}_i$ define
$\mathcal{L}\otimes_{\underline{\mathbb{C}}_X}\mathcal{A}_X^{p,q}\rightarrow \overline{\mathcal{L}}\otimes_{\underline{\mathbb{C}}_X}\mathcal{A}_X^{q,p}$
and $\mathcal{L}\otimes_{\underline{\mathbb{C}}_X}\mathcal{D}_X^{\prime p,q}\rightarrow \overline{\mathcal{L}}\otimes_{\underline{\mathbb{C}}_X}\mathcal{D}_X^{\prime q,p}$,
which are said to be the \emph{complex conjugation maps}.

\subsubsection{Projective bundle and blow-up formulae on the $E_1$-level}
We recall some notions of the double complex and  related structures, see \cite[Sect. 2]{St2} for more details. All the double complexes here are assumed to be
\emph{of vector spaces over $\mathbb{C}$} and \emph{bounded}.

Let $(K^{\bullet,\bullet},\partial_1,\partial_2)$ be a double complex  with two endomorphisms $\partial_1$ , $\partial_2$ of bidegree $(1,0)$ and $(0,1)$, which
satisfy that $\partial_i\circ\partial_i=0$ for $i=1$, $2$ and $\partial_1\circ\partial_2+\partial_2\circ\partial_1=0$.
For convenience, we briefly write it as $K^{\bullet,\bullet}$ sometimes.
Let $\partial^{p,q}_1:K^{p,q}\rightarrow K^{p+1,q}$ and $\partial^{p,q}_2:K^{p,q}\rightarrow K^{p,q+1}$ be the restrictions of $\partial_1$ and $\partial_2$
respectively.
Recall the following constructions.

$\bullet$ The \emph{row} and \emph{column cohomologies}
\begin{displaymath}
H^{p,q}_{\partial_1}(K^{\bullet,\bullet})=H^{p}(K^{\bullet,q},\partial_1)\mbox{  } \mbox{  and  } \mbox{  }
H^{p,q}_{\partial_2}(K^{\bullet,\bullet})=H^{q}(K^{p,\bullet},\partial_2).
\end{displaymath}

$\bullet$ The \emph{Bott-Chern} and \emph{Aeppli cohomologies}
\begin{displaymath}
H^{p,q}_{BC}(K^{\bullet,\bullet})=\frac{\textrm{ker}\partial^{p,q}_1\cap\textrm{ker}\partial^{p,q}_2}{\textrm{im}\partial^{p-1,q}_1\circ\partial^{p-1,q-1}_2}
\mbox{  }\mbox{ and }\mbox{  }
H^{p,q}_{A}(K^{\bullet,\bullet})=\frac{\textrm{ker}\partial^{p,q+1}_1\circ\partial^{p,q}_2}{\textrm{im}\partial^{p-1,q}_1+\textrm{im}\partial^{p,q-1}_2}.
\end{displaymath}
A morphism of double complexes is called an \emph{$E_1$-isomorphism}, if it induces an isomorphism on both row and column cohomologies.
J. Stelzig obtained the following result.
\begin{thm}[{\cite[Corollary 13]{St2}\cite[Lemma 1.3]{St1}}]\label{Stelzig}
Any $E_1$-isomorphism of double complexes induces  isomorphisms on  Bott-Chern and Aeppli cohomologies respectively.
\end{thm}

Let $X$ be a complex manifold and let $\mathcal{L}$ be a local system of $\mathbb{C}$-modules of finite rank on $X$.
Set $A_{\Phi}^{p,q}(X,\mathcal{L})=\Gamma_{\Phi}(X,\mathcal{L}\otimes_{\underline{\mathbb{C}}_X}\mathcal{A}_X^{p,q})$ and $\mathcal{D}_{\Phi}^{\prime
p,q}(X,\mathcal{L})=\Gamma_{\Phi}(X,\mathcal{L}\otimes_{\underline{\mathbb{C}}_X}\mathcal{D}_X^{\prime p,q})$.
Denote by  $A_{\Phi}^{\bullet,\bullet}(X,\mathcal{L})$  and  $\mathcal{D}_{\Phi}^{\prime\bullet,\bullet}(X,\mathcal{L})$ the double complexes
$(A_{\Phi}^{\bullet,\bullet}(X,\mathcal{L}),\partial,\bar{\partial})$  and
$(\mathcal{D}_{\Phi}^{\prime\bullet,\bullet}(X,\mathcal{L}),\partial,\bar{\partial})$.
On column cohomologies, the inclusion $i:A_{\Phi}^{\bullet,\bullet}(X,\mathcal{L})\rightarrow \mathcal{D}_{\Phi}^{\prime\bullet,\bullet}(X,\mathcal{L})$ induces
an isomorphism $H_{\bar{\partial}}^q(A_{\Phi}^{p,\bullet}(X,\mathcal{L}))\tilde{\rightarrow} H_{\bar{\partial}}^q(\mathcal{D}_{\Phi}^{\prime
p,\bullet}(X,\mathcal{L}))$, which are just $H^{p,q}_\Phi(X,\mathcal{L}\otimes_{\underline{\mathbb{C}}_X}\mathcal{O}_X)$.
Consider the commutative diagram
\begin{displaymath}
\xymatrix{
 H_{\bar{\partial}}^q(A_\Phi^{p,\bullet}(X,\overline{\mathcal{L}})) \ar[d]_{} \ar[r]^{}& H_{\bar{\partial}}^q(\mathcal{D}_\Phi^{\prime
 p,\bullet}(X,\overline{\mathcal{L}}))\ar[d]^{}\\
 H_{\partial}^p(A_\Phi^{\bullet,q}(X,\mathcal{L}))      \ar[r]^{}& H_{\partial}^p(\mathcal{D}_\Phi^{\prime\bullet,q}(X,\mathcal{L})),  }
\end{displaymath}
where the horizontal maps are induced by the inclusion $i$ and the vertical maps are defined by complex conjugation maps.
The vertical maps are real  isomorphisms and the upper  map is an (complex) isomorphism, so the lower  map is a real  isomorphism.
Since it is a complex linear map, the lower  map is an (complex) isomorphism, i.e., an isomorphism on row cohomologies.
Hence the inclusion $i:A_{\Phi}^{\bullet,\bullet}(X,\mathcal{L})\rightarrow \mathcal{D}_{\Phi}^{\prime\bullet,\bullet}(X,\mathcal{L})$ is an $E_1$-isomorphism.
By Theorem \ref{Stelzig}, the  Bott-Chern cohomologies of $A_{\Phi}^{\bullet,\bullet}(X,\mathcal{L})$  and
$\mathcal{D}_{\Phi}^{\prime\bullet,\bullet}(X,\mathcal{L})$ coincide with each other  via the inclusion, which are both denoted by
$H_{BC,\Phi}^{p,q}(X,\mathcal{L})$.
Similarly, we write their Aeppli cohomologies as $H_{A,\Phi}^{p,q}(X,\mathcal{L})$.

Let $\pi:\mathbb{P}(E)\rightarrow X$ be the projective bundle associated to a holomorphic bundle $E$ of rank $r$ over a complex manifold $X$ and let
$\mathcal{L}$ be a local system  of $\mathbb{C}$-modules of finite rank on $X$.
Denote by $t\in \mathcal{A}^{1,1}({\mathbb{P}(E)})$ a Chern form of the universal line bundle $\mathcal{O}_{\mathbb{P}(E)}(-1)$ over ${\mathbb{P}(E)}$.
By Proposition \ref{proj-bun},
$\mu_{Dol,\Phi}^{\mathcal{L}\otimes_{\underline{\mathbb{C}}_X}\mathcal{O}_X}$ and $\mu_{Dol,\Phi}^{\overline{\mathcal{L}}\otimes_{\underline{\mathbb{C}}_X}\mathcal{O}_X}$ are  isomorphisms.
By similar arguments with those of $i$,
$\mu_\Phi^{\mathcal{L}\otimes_{\underline{\mathbb{C}}_X}\mathcal{O}_X}:\bigoplus\limits_{i=1}^{r-1}A_{\Phi}^{\bullet,\bullet}(X,\mathcal{L})[-i,-i]\rightarrow \mathcal{A}_{\pi^{-1}\Phi}^{\bullet,\bullet}(\mathbb{P}(E),\pi^{-1}\mathcal{L})$
is an $E_1$-isomorphism.
Similarly, $\tau_\Phi^{\mathcal{L}\otimes_{\underline{\mathbb{C}}_X}\mathcal{O}_X}$ is an $E_1$-isomorphism.
Denote by $\mu^{\mathcal{L}}_{BC,\Phi}$,  $\tau^{\mathcal{L}}_{BC,\Phi}$ the morphisms on  Bott-Chern cohomologies  and
by  $\mu^{\mathcal{L}}_{A,\Phi}$, $\tau^{\mathcal{L}}_{A,\Phi}$ the morphisms on Aeppli cohomologies
induced by $\mu_{\Phi}^{\mathcal{L}\otimes_{\underline{\mathbb{C}}_X}\mathcal{O}_X}$ and $\tau_{\Phi}^{\mathcal{L}\otimes_{\underline{\mathbb{C}}_X}\mathcal{O}_X}$.
By Theorem \ref{Stelzig}, $\mu^{\mathcal{L}}_{BC,\Phi}$,  $\tau^{\mathcal{L}}_{BC,\Phi}$,  $\mu^{\mathcal{L}}_{A,\Phi}$, $\tau^{\mathcal{L}}_{A,\Phi}$ are
isomorphisms.
As Proposition \ref{proj-bun}, we easily check that $\mu^{\mathcal{L}}_{BC,\Phi}$ and  $\tau^{\mathcal{L}}_{BC,\Phi}$,
$\mu^{\mathcal{L}}_{A,\Phi}$ and $\tau^{\mathcal{L}}_{A}$ are inverse isomorphisms.

We summarize these results as follows.
\begin{prop}\label{projbun-E_1}
$\mu_{\Phi}^{\mathcal{L}\otimes_{\underline{\mathbb{C}}_X}\mathcal{O}_X}$ and
$\tau_{\Phi}^{\mathcal{L}\otimes_{\underline{\mathbb{C}}_X}\mathcal{O}_X}$  are $E_1$-isomorphisms.
Moreover,  $\mu^{\mathcal{L}}_{BC,\Phi}$ and  $\tau^{\mathcal{L}}_{BC,\Phi}$,  $\mu^{\mathcal{L}}_{A,\Phi}$ and $\tau^{\mathcal{L}}_{A,\Phi}$ are inverse
isomorphisms.
\end{prop}

Let $\pi:\widetilde{X}\rightarrow X$ be the complex blow-up of a  complex manifold $X$ along a complex submanifold $Y$ with the exceptional divisor $E$.
Suppose that $\mathcal{E}$ is a locally free sheaf of $\mathcal{O}_X$-module of finite rank on $X$.
Denote by  $i_E:E\rightarrow \widetilde{X}$ the inclusion.
Let $t\in \mathcal{A}^{1,1}(E)$ be a Chern form of the universal line bundle $\mathcal{O}_{E}(-1)$ on $E={\mathbb{P}(N_{Y/X})}$.
The morphism $\pi^*+\sum\limits_{i=1}^{r-1}i_{E*}\left(t^{i-1}\wedge(\pi|_E)^*(\bullet)\right)$ defines
\begin{equation}\label{blow-up1-cpt-form}
\small{
\begin{aligned}
\Gamma_{\Phi}(X,,\mathcal{E}\otimes \mathcal{A}_X^{\bullet,\bullet})\oplus \bigoplus\limits_{i=1}^{r-1}\Gamma_{\Phi|_Y}(Y,i_Y^{\ast}\mathcal{E}\otimes
\mathcal{A}_Y^{\bullet-i,\bullet-i})\rightarrow \Gamma_{\pi^{-1}\Phi}(\widetilde{X},\pi^{\ast}\mathcal{E}\otimes \mathcal{D}_{\widetilde{X}}^{\prime
\bullet,\bullet}).
\end{aligned}}
\end{equation}
The morphism
$\left(\pi_*(\bullet),\mbox{ }G_{t,\Phi}^{-1}\circ i_E^*(\bullet),\mbox{ }...,\mbox{ }G_{t,\Phi}^{-r+1}\circ i_E^*(\bullet)\right)$ defines
\begin{equation}\label{blow-up2-cpt-form}
\small{
\begin{aligned}
\Gamma_{\pi^{-1}\Phi}(\widetilde{X},\pi^{\ast}\mathcal{E}\otimes \mathcal{A}_{\widetilde{X}}^{\bullet,\bullet}) \rightarrow \Gamma_{\Phi}(X,,\mathcal{E}\otimes
\mathcal{D}_X^{\prime\bullet,\bullet})\oplus \bigoplus\limits_{i=1}^{r-1}\Gamma_{\Phi|_Y}(Y,i_Y^{\ast}\mathcal{E}\otimes \mathcal{A}_Y^{\bullet-i,\bullet-i}).
\end{aligned}}
\end{equation}
Denote  (\ref{blow-up1-cpt-form}) and (\ref{blow-up2-cpt-form}) by $\psi_{\Phi}^{\mathcal{E}}$ and  $\phi_{\Phi}^{\mathcal{E}}$  respectively.
Denote by $\psi^{\mathcal{L}}_{BC,\Phi}$,  $\phi^{\mathcal{L}}_{BC,\Phi}$ the morphisms on  Bott-Chern cohomologies  and
denote by  $\psi^{\mathcal{L}}_{A,\Phi}$, $\phi^{\mathcal{L}}_{A,\Phi}$ the morphisms on Aeppli cohomologies
induced by $\psi_{\Phi}^{\mathcal{L}\otimes_{\underline{\mathbb{C}}_X}\mathcal{O}_X}$ and $\phi_{\Phi}^{\mathcal{L}\otimes_{\underline{\mathbb{C}}_X}\mathcal{O}_X}$.
As Proposition \ref{projbun-E_1}, we have the following result by Theorems \ref{1.2} and \ref{Stelzig}.
\begin{prop}\label{blowup-E_1}
$\psi_\Phi^{\mathcal{L}\otimes_{\underline{\mathbb{C}}_X}\mathcal{O}_X}$ and
$\phi_\Phi^{\mathcal{L}\otimes_{\underline{\mathbb{C}}_X}\mathcal{O}_X}$ are $E_1$-isomorphisms.
Moreover, $\psi^{\mathcal{L}}_{BC,\Phi}$,  $\phi^{\mathcal{L}}_{BC,\Phi}$, $\psi^{\mathcal{L}}_{A,\Phi}$, $\phi^{\mathcal{L}}_{A,\Phi}$ are isomorphisms.
\end{prop}

\begin{rem}
On compact complex manifolds, S. Yang, X.-D. Yang \cite[Theorem 1.2]{YY} and J. Stelzig \cite[Corollary 12, Theorem 23]{St1} \cite[Proposition 4, Theorem
8]{St2}  showed the existence  of the isomorphism $H_{BC}^{p,q}(X)\oplus \bigoplus\limits_{i=1}^{r-1}H_{BC}^{p-i,q-i}(Y)\cong H_{BC}^{p,q}(\widetilde{X})$,
which was not expressed explicitly. 
\end{rem}

\begin{quest}\label{bc-inverse-quest}
\emph{Are $\psi^{\mathcal{L}}_{BC,\Phi}$ and $\phi^{\mathcal{L}}_{BC,\Phi}$ (resp. $\psi^{\mathcal{L}}_{A,\Phi}$ and $\phi^{\mathcal{L}}_{A,\Phi}$) inverse to
each other? }
\end{quest}
Analogue to Proposition \ref{relation},
if $i_E^*i_{E*}(\bullet)=h_{BC}\cup\bullet$ (resp. $i_E^*i_{E*}(\bullet)=h_{A}\cup\bullet$)  holds on  $H_{BC,(\pi^{-1}\Phi)|_E}^{\bullet,\bullet}(E,(\pi_E)^{-1}(\mathcal{L}|_Y))$
(resp. $H_{A,(\pi^{-1}\Phi)|_E}^{\bullet,\bullet}(E,(\pi_E)^{-1}(\mathcal{L}|_Y))$),
$\psi^{\mathcal{L}}_{BC,\Phi}$ and $\phi^{\mathcal{L}}_{BC,\Phi}$ (resp. $\psi^{\mathcal{L}}_{A,\Phi}$ and $\phi^{\mathcal{L}}_{A,\Phi}$) are inverse isomorphisms.

\subsubsection{Hypercohomologies of truncated twisted holomorphic de Rham complexes}
For a complex $\mathcal{F}^\bullet$  of sheaves on a topological space $X$,
set $\mathbb{H}_\Phi^k(X,\mathcal{F}^\bullet)=R^k\Gamma_\Phi(\mathcal{F}^\bullet)$, which is called the \emph{$k$-th hypercohomology with  supports in $\Phi$}.
Using the results in the earlier version \cite{M3} of the present paper, we generalized several classical results for the hypercohomologies of truncated twisted
holomorphic de Rham complexes in \cite{M5}.
Based on the present results, we can extend them to the ones with supports in paracompactifying families.
We  only write out the blow-up formula and give a detailed proof here.

Let $X$ be an $n$-dimensional complex manifold  and
let $\mathcal{L}$ be a local system of $\mathbb{C}$-modules of finite rank  on $X$.
Given any integers $s$ and $t$, the \emph{truncated twisted holomorphic de Rham complex} $\Omega_X^{[s,t]}(\mathcal{L})$ is defined as the zero complex if $s>
t$ and as the complex
\begin{equation}\label{truncated twisted}
\xymatrix{
0\ar[r] &\mathcal{L}\otimes_{\underline{\mathbb{C}}_X}\Omega_X^s\ar[r]^{\partial} &\mathcal{L}\otimes_{\underline{\mathbb{C}}_X}\Omega_X^{s+1}\ar[r]^{
\qquad\partial}&\cdots\ar[r]^{\partial\qquad}&\mathcal{L}\otimes_{\underline{\mathbb{C}}_X}\Omega_X^t\ar[r]&0
}
\end{equation}
if $s\leq t$, where $\mathcal{L}\otimes\Omega_X^k$ is placed in degree $k$ for $s\leq k\leq t$ and zeros are placed in other degrees.
In particular, $\Omega_X^{[0,n]}(\mathcal{L})=\mathcal{L}\otimes_{\underline{\mathbb{C}}_X}\Omega_X^{\bullet}$ is the twisted holomorphic de Rham complex on $X$ and
$\Omega_X^{[p,p]}(\mathcal{L})=(\mathcal{L}\otimes_{\underline{\mathbb{C}}_X}\Omega_X^{p\bullet})[-p]$, where $\Omega_X^{p\bullet}$ denote the complex with $\Omega_X^{p}$ in degree $0$
and zeros in other degrees.

Set
\begin{displaymath}
\mathcal{S}^{p,q}_{X}(\mathcal{L},s,t)=\left\{
 \begin{array}{ll}
\mathcal{L}\otimes_{\underline{\mathbb{C}}_X}\mathcal{A}^{p,q}_{X},&~s\leq p\leq t\\
 &\\
 0,&~\textrm{others}.
 \end{array}
 \right.
\end{displaymath}
\begin{displaymath}
d^{p,q}_1=\left\{
 \begin{array}{ll}
\partial,&~s\leq p< t\\
 &\\
 0,&~\textrm{others}
 \end{array}
 \right.
\textrm{, }\quad
d^{p,q}_2=\left\{
 \begin{array}{ll}
\bar{\partial},&~s\leq p\leq t\\
 &\\
 0,&~\textrm{others}.
 \end{array}
 \right.
\end{displaymath}
Then $(\mathcal{S}^{\bullet,\bullet}_{X}(\mathcal{L},s,t),d_1,d_2)$ is a double complex of sheaves, which is  shortly denoted by
$\mathcal{S}^{\bullet,\bullet}_{X}(\mathcal{L},s,t)$.
Let $\mathcal{S}^{\bullet}_{X}(\mathcal{L},s,t)$  be the simple complex associated to $\mathcal{S}^{\bullet,\bullet}_{X}(\mathcal{L},s,t)$.
For any $p\in\mathbb{Z}$, $\Omega_X^{[s,t]}(\mathcal{L})^p\rightarrow (\mathcal{S}^{p,\bullet}_{X}(\mathcal{L},s,t),d^{p,\bullet}_2)$ given by the inclusion is
a resolution of $\Omega_X^{[s,t]}(\mathcal{L})^p$.
By \cite[Lemma 8.5]{V}, the inclusion gives a quasi-isomorphism $\Omega_X^{[s,t]}(\mathcal{L})\rightarrow \mathcal{S}^{\bullet}_{X}(\mathcal{L},s,t)$ of
complexes of sheaves.
Suppose that $\Phi$ is a paracompactifying family of supports on $X$.
Set $S_\Phi^{p,q}(X,\mathcal{L},s,t)=\Gamma_\Phi(X,\mathcal{S}^{p,q}_{X}(\mathcal{L},s,t))$ and
$S_\Phi^{p}(X,\mathcal{L},s,t)=\Gamma_\Phi(X,\mathcal{S}^{p}_{X}(\mathcal{L},s,t))$.
The sheaf $\mathcal{C}_X^\infty$  is $\Phi$-soft (\cite[II. 9.4]{Br}), so are
$\mathcal{S}^{p}_{X}(\mathcal{L},s,t)$ by \cite[II. 9.16]{Br}.
Hence $\mathcal{S}^{p}_{X}(\mathcal{L},s,t)$  are $\Phi$-acyclic by \cite[II. 9.11]{Br}.
By \cite[Proposition 8.12]{V}, the hypercohomology
\begin{equation}\label{computation1}
\mathbb{H}_\Phi^k(X,\Omega_X^{[s,t]}(\mathcal{L}))\cong H^k(S_\Phi^{\bullet}(X,\mathcal{L},s,t))
\end{equation}
for any $k\in\mathbb{Z}$.
For example, $\mathbb{H}_\Phi^k(X,\Omega_X^{[0,n]}(\mathcal{L}))\cong H_{\Phi}^k(X,\mathcal{L})$ and $\mathbb{H}_\Phi^k(X,\Omega_X^{[p,p]}(\mathcal{L}))\cong
H_\Phi^{p,k-p}(X,\mathcal{L}\otimes_{\underline{\mathbb{C}}_X}\mathcal{O}_X)$.

Similarly, we can define $\mathcal{T}^{\bullet,\bullet}_{X}(\mathcal{L},s,t)$, $\mathcal{T}^{\bullet}_{X}(\mathcal{L},s,t)$,
$T_\Phi^{\bullet,\bullet}(X,\mathcal{L},s,t)$ and $T_\Phi^{\bullet}(X,\mathcal{L},s,t)$, where
\begin{displaymath}
\mathcal{T}^{p,q}_{X}(\mathcal{L},s,t)=\left\{
 \begin{array}{ll}
\mathcal{L}\otimes_{\underline{\mathbb{C}}_X}\mathcal{D}^{\prime p,q}_{X},&~s\leq p\leq t\\
 &\\
 0,&~\textrm{others}.
 \end{array}
 \right.
\end{displaymath}
The inclusion gives a quasi-isomorphism  $\Omega_X^{[s,t]}(\mathcal{L})\rightarrow \mathcal{T}^{\bullet}_{X}(\mathcal{L},s,t)$ of complexes of sheaves.
There is an isomorphism
\begin{equation}\label{computation2}
\mathbb{H}_\Phi^k(X,\Omega_X^{[s,t]}(\mathcal{L}))\cong H_\Phi^k(T^{\bullet}(X,\mathcal{L},s,t))
\end{equation}
for any $k\in\mathbb{Z}$.

The following result extends \cite[Question 10]{CY} and gives it a positive answer.
\begin{thm}
Let $\pi:\widetilde{X}\rightarrow X$ be the complex blow-up of a  complex manifold $X$ along a  complex submanifold $Y$ and $\mathcal{L}$ a local system of
$\mathbb{C}$-modules of finite rank on $X$.
Denote by $i_Y:Y\rightarrow X$  the inclusion and set $r=\emph{codim}_{\mathbb{C}}Y\geq 2$.
Suppose that $\Phi$ is a paracompactifying family of supports on $X$.
Then there exists an isomorphism
\begin{displaymath}
\small{
\begin{aligned}
\mathbb{H}_{\pi^{-1}\Phi}^{k}(\widetilde{X},\Omega_{\widetilde{X}}^{[s,t]}(\pi^{-1}\mathcal{L}))\cong\mathbb{H}_\Phi^k(X,\Omega_X^{[s,t]}(\mathcal{L}))\oplus
\bigoplus_{i=1}^{r-1}\mathbb{H}_{\Phi|_Y}^{k-2i}(Y,\Omega_Y^{[s-i,t-i]}(\mathcal{L}|_Y)),
\end{aligned}}
\end{displaymath}
for any $k$, $s$, $t$.
\end{thm}
\begin{proof}
Fix two integers $s$ and $t$.
Consider the complexes
$K^{\bullet,\bullet}(s,t)=S_{\pi^{-1}\Phi}^{\bullet,\bullet}(\widetilde{X},\pi^{-1}\mathcal{L},s,t)$ and
\begin{displaymath}
\small{
\begin{aligned}
L^{\bullet,\bullet}(s,t)=T_\Phi^{\bullet,\bullet}(X,\mathcal{L},s,t)\oplus\bigoplus_{i=1}^{r-1}T_{\Phi|_Y}^{\bullet,\bullet}(Y,\mathcal{L}|_Y,s-i,t-i)[-i,-i].
\end{aligned}}
\end{displaymath}
We have the first pages
\begin{displaymath}
\small{
\begin{aligned}
_{K}E_1^{p,q}=&H^{q}(S_{\pi^{-1}\Phi}^{p,\bullet}(\widetilde{X},\pi^{-1}\mathcal{L},s,t))\\
=&\left\{
 \begin{array}{ll}
H_{\pi^{-1}\Phi}^{p,q}\left(\widetilde{X},\pi^{*}(\mathcal{L}\otimes_{\underline{\mathbb{C}}_X}\mathcal{O}_{X})\right),&~s\leq p\leq t\\
 &\\
 0,&~\textrm{others},
 \end{array}
 \right.
\end{aligned}}
\end{displaymath}
\begin{displaymath}
\small{
\begin{aligned}
_{L}E_1^{p,q}=& H^q(T_\Phi^{p,\bullet}(X,\mathcal{L},s,t))\oplus\bigoplus_{i=1}^{r-1}H^{q-i}(T_{\Phi|_Y}^{p-i,\bullet}(Y,\mathcal{L}|_Y,s-i,t-i)) \\
=&\left\{
 \begin{array}{ll}
H_\Phi^{p,q}(X,\mathcal{L}\otimes_{\underline{\mathbb{C}}_X}\mathcal{O}_X)\oplus\bigoplus\limits_{i=1}^{r-1}H_{\Phi|_Y}^{p-i,q-i}(Y,i_Y^*(\mathcal{L}\otimes_{\underline{\mathbb{C}}_X}\mathcal{O}_X)),&~s\leq p\leq t\\
 &\\
 0,&~\textrm{others}
 \end{array}
 \right.
\end{aligned}}
\end{displaymath}
of the spectral sequences associated to $K^{\bullet,\bullet}$, $L^{\bullet,\bullet}$ respectively.
Let $t\in \mathcal{A}^{1,1}(E)$ be a first Chern form of the universal line bundle $\mathcal{O}_{E}(-1)$ on $E\cong{\mathbb{P}(N_{Y/X})}$ and let
$i_E:E\rightarrow \widetilde{X}$ be the inclusion.
By Theorem \ref{1.2}, the morphism
\begin{displaymath}
(\pi_*,\mbox{ }G_{t,\Phi}^{-1}\circ i_E^*,\mbox{ }...,\mbox{ }G_{t,\Phi}^{-r+1}\circ i_E^*):K^{\bullet,\bullet}\rightarrow L^{\bullet,\bullet}
\end{displaymath}
induces an isomorphism  $_KE_1^{\bullet,\bullet}\rightarrow _LE_1^{\bullet,\bullet}$  at $E_1$-pages, hence induces an isomorphism
$H^{k}(K^{\bullet})\rightarrow H^{k}(L^{\bullet})$ for any $k$, where $K^\bullet$ and $L^\bullet$ are the associated complex to $K^{\bullet,\bullet}$ and
$L^{\bullet,\bullet}$ respectively.
By (\ref{computation1}) and  (\ref{computation2}),
$H^{k}(K^{\bullet})\cong \mathbb{H}_{\pi^{-1}\Phi}^{k}(\widetilde{X},\Omega_{\widetilde{X}}^{[s,t]}(\pi^{-1}\mathcal{L}))$ and
\begin{displaymath}
\begin{aligned}
H^{k}(L^{\bullet})
\cong&\mathbb{H}_\Phi^k(X,\Omega_X^{[s,t]}(\mathcal{L}))\oplus \bigoplus_{i=1}^{r-1}\mathbb{H}_{\Phi|_Y}^{k-2i}(Y,\Omega_Y^{[s-i,t-i]}(\mathcal{L}|_Y)),
\end{aligned}
\end{displaymath}
from which the theorem  follows.
\end{proof}

\section{Twisted de Rham cohomology with supports}
\subsection{General results}
Suppose that $\mathcal{F}$ and $\mathcal{G}$ are  sheaves of $\mathbf{k}$-modules  on topological spaces $X$ and $Y$ respectively.
The \emph{external tensor product} of $\mathcal{F}$ and $\mathcal{G}$ is defined as
\begin{displaymath}
\mathcal{F}\boxtimes\mathcal{G}=pr_1^{-1}\mathcal{F}\otimes_{\underline{\mathbf{k}}_{X\times Y}}pr_2^{-1}\mathcal{G},
\end{displaymath}
where $pr_1$ and $pr_2$ are projections from $X\times Y$ onto $X$, $Y$, respectively.
\begin{prop}\label{Kun2}
Let $\mathcal{L}$, $\mathcal{H}$ be  local systems of $\mathbf{k}$-modules of finite ranks on smooth manifolds $X$, $Y$ respectively and let $\Phi$ be a family of supports on $X$.
Suppose that $H^\bullet(Y,\mathcal{H})$ has finite dimension.
Then $(\alpha^p\otimes \beta^q)_{p+q=k}\mapsto \sum\limits_{p+q=k}\alpha^p\times \beta^q$ gives an isomorphism
\begin{displaymath}
\bigoplus\limits_{p+q=k}H_\Phi^p(X,\mathcal{L})\otimes_{\mathbf{k}} H^q(Y,\mathcal{H})\rightarrow H_{\Phi\times Y}^k(X\times Y, \mathcal{L}\boxtimes \mathcal{H}).
\end{displaymath}
\end{prop}
\begin{proof}
Let $\mathfrak{U}$ be a good covering of $Y$ such that $U$ is $\mathcal{H}$-constant for any $U\in \mathfrak{U}$.
Assume that $V\subseteq X$ is an $\mathcal{L}$-constant open subset of $X$.
The cartesian product gives an isomorphism $\Gamma(V,\mathcal{L})\otimes_{\mathbb{C}} \Gamma(U,\mathcal{H})\cong \Gamma(V\times U,\mathcal{L}\boxtimes\mathcal{H})$ for  any $\mathcal{H}$-constant open set $U\subseteq Y$.
Then the $\check{C}$ech complex
$\check{C}^\bullet(V\times \mathfrak{U}, \mathcal{L}\boxtimes\mathcal{H})=\Gamma(V,\mathcal{L})\otimes_{\mathbb{C}} \check{C}^\bullet(\mathfrak{U}, \mathcal{H})$.
Since $V\times \mathfrak{U}$ and $\mathfrak{U}$ are acyclic with respect to $\mathcal{L}\boxtimes\mathcal{H}$ and $\mathcal{H}$,
\begin{displaymath}
\small{
\begin{aligned}
H^p(V\times Y,\mathcal{L}\boxtimes\mathcal{H})=H^p(\check{C}^\bullet(V\times \mathfrak{U}, \mathcal{L}\boxtimes\mathcal{H}))
=\Gamma(V,\mathcal{L})\otimes_{\mathbb{C}} H^p(\check{C}^p(\mathfrak{U}, \mathcal{H}))=\Gamma(V,\mathcal{L})\otimes_{\mathbb{C}} H^p(Y,\mathcal{H}).
\end{aligned}}
\end{displaymath}
So $\mathcal{L}\otimes_{\mathbb{C}} H^p(Y,\mathcal{H})$ is the sheaf associated to the presheaf $W\mapsto H^p(W\times Y,\mathcal{L}\boxtimes\mathcal{H})$ for any open set  $W\subseteq X$.
Namely,
\begin{equation}\label{*****}
R^p\pi_*(\mathcal{L}\boxtimes\mathcal{H})=\mathcal{L}\otimes_{\mathbb{C}} H^p(Y,\mathcal{H}),
\end{equation}
where $\pi:X\times Y\rightarrow X$ is the first projection of $X\times Y$ onto $X$.

Define $\Gamma_\Phi(X,\mathcal{L}^{[p]})\otimes_{\mathbb{C}}\Gamma(Y,\mathcal{H}^{[q]})\rightarrow \Gamma_\Phi\left(X,(\pi_*((\mathcal{L}\boxtimes \mathcal{H})^{[q]}))^{[p]}\right)$ as $f\otimes g\mapsto h$, where
\begin{displaymath}
\small{\begin{aligned}
&h(\xi_0,\ldots,\xi_p; (x_0,y_0),\ldots,(x_q,y_q))=S(f(\xi_0,\ldots,\xi_p))(x_q)\otimes g(y_0,\ldots, y_q)\\
\in &(\mathcal{F}\boxtimes \mathcal{G})_{(x_q,y_q)}=\mathcal{F}_{x_q}\otimes_{\mathbf{k}}\mathcal{G}_{y_q}
\end{aligned}}
\end{displaymath}
for any $(\xi_0,\ldots,\xi_p; (x_0,y_0),\ldots,(x_q,y_q))\in X^{p+1}\times (X\times Y)^{q+1}$.
Using $(\ref{*****})$ instead of \cite[IX. 5.22 (c)]{Dem}, we easily prove this theorem as Theorem \ref{Kun}.
\end{proof}

With the similar proof of Theorem \ref{L-H1}, we have
\begin{thm}\label{L-H2}
Let $\pi:E\rightarrow X$ be a smooth fiber bundle  over a smooth manifold $X$ and let $\mathcal{L}$ be a local system  of
$\mathbf{k}$-modules of finite rank on $X$.
Assume that there exist $e_i \in H_{dR}^{\bullet}(E,\mathbf{k})$ with degree $u_i$ for $1\leq i\leq r$ such that  their
restrictions $e_1|_{E_x},\dots,e_r|_{E_x}$ freely linearly generate $H_{dR}^{\bullet}(E_x,\mathbf{k})$ for every $x\in X$.
Then
\begin{displaymath}
\sum\limits_{i=1}^r\pi^*(\bullet)\cup e_i:\bigoplus\limits_{i=1}^rH_\Phi^{\bullet-u_i}(X,\mathcal{L}) \tilde{\rightarrow} H_{\pi^{-1}\Phi}^{\bullet}(E,\pi^{-1}\mathcal{L})
\end{displaymath}
is an isomorphism for  any paracompactifying family $\Phi$  of supports on $X$.
\end{thm}

Let $\pi:\mathbb{P}(E)\rightarrow X$ be the complex projectivization of a complex  vector bundle $E$ of complex rank $r$ over a smooth manifold $X$ and
let $t\in \mathcal{A}^{2}({\mathbb{P}(E)})$ be a Chern form of the universal line bundle $\mathcal{O}_{\mathbb{P}(E)}(-1)$ over ${\mathbb{P}(E)}$.
Then $\mathbb{P}(E)$ is an orientable fiber bundle and $t^{r-1}$ represents a orientation of $\mathbb{P}(E)$  (see \cite[VII., 7.4]{GHV} for definitions).
Define the pushforward $\pi_*$ as \emph{the integral over the fiber}, refer to \cite[VII., 7.12]{GHV}.
Notice that $X$ is \emph{not necessarily orientable} here.
As those in Sect. 4.3, $\pi_*t^i=0$  for $0\leq i\leq r-2$ and $\pi_*t^{r-1}=(-1)^{r-1}$.
By \cite[VII., Proposition X (3)]{GHV},  $\pi_*:\mathcal{A}^{\bullet}(\mathbb{P}(E))\rightarrow\mathcal{A}^{\bullet-2r}(X)$  is a morphism of complexes.

Suppose that $\mathcal{L}$ is a local system of $\mathbf{k}$-modules of finite rank  on $X$
and $\Phi$ is a paracompactifying family of supports of $X$.
Analogue to Sect. 3.1.3, we can further define
$\pi_*:\Gamma_{\pi^{-1}\Phi}(\mathbb{P}(E),\pi^{-1}\mathcal{L}\otimes\mathcal{A}_{\mathbb{P}(E)}^{\bullet})
\rightarrow\Gamma_{\Phi}(X,\mathcal{L}\otimes\mathcal{A}_X^{\bullet-2r})$.
By \cite[VII., Proposition IX]{GHV},
\begin{displaymath}
\pi_*(\alpha\wedge \pi^*\beta)=\pi_*\alpha\wedge\beta
\end{displaymath}
for $\alpha\in \Gamma_{\pi^{-1}\Phi}(\mathbb{P}(E),\pi^{-1}\mathcal{L}\otimes\mathcal{A}_{\mathbb{P}(E)}^{\bullet})$
and $\beta\in\Gamma_{\Phi}(X,\mathcal{L}\otimes\mathcal{A}_X^{\bullet})$.

Put $h=c_1(\mathcal{O}_{\mathbb{P}(E)}(-1))\in H_{dR}^2({\mathbb{P}(E)},\mathbf{k})$ the first Chern class of $\mathcal{O}_{\mathbb{P}(E)}(-1)$.
Denote  by $\mu_{dR,\Phi}^{\mathcal{L}}$ the morphism
\begin{displaymath}
\sum\limits_{i=0}^{r-1}\pi^*(\bullet)\cup h^i:\bigoplus\limits_{i=0}^{r-1}H_{\Phi}^{\bullet-2i}(X,\mathcal{L})\rightarrow H_{\pi^{-1}\Phi}^\bullet(\mathbb{P}(E),\pi^{-1}\mathcal{L})
\end{displaymath}
and by $\tau_{dR,\Phi}^{\mathcal{L}}$  the morphism
\begin{displaymath}
\small{
\begin{aligned}
(G_{h,\Phi}^{0}(\bullet),\mbox{ }G_{h,\Phi}^{-1}(\bullet),\mbox{ }...,\mbox{ }G_{h,\Phi}^{-r+1}(\bullet)): H_{\pi^{-1}\Phi}^\bullet(\mathbb{P}(E),\pi^{-1}\mathcal{L})\rightarrow
\bigoplus\limits_{i=0}^{r-1}H_{\Phi}^{\bullet-2i}(X,\mathcal{L}),
\end{aligned}}
\end{displaymath}
where
\begin{displaymath}
\small{
\begin{aligned}
G_{h,\Phi}^{-i}(\bullet)=\sum_{j=0}^{r-1-i}P_{i+j}(\pi_*h^r,\mbox{ }....,\mbox{ }\pi_*h^{2r-2})\cup\pi_*(h^j\cup \bullet):H_{\pi^{-1}\Phi}^{\bullet}(\mathbb{P}(E),\pi^{-1}\mathcal{L})\rightarrow H_{\Phi}^{\bullet-2i}(X,\mathcal{L})
\end{aligned}}
\end{displaymath}
for $0\leq i\leq r-1$.
As Proposition \ref{proj-bun}, we have
\begin{prop}\label{3}
$\mu_{dR,\Phi}^{\mathcal{L}}$ and  $\tau_{dR,\Phi}^{\mathcal{L}}$ are inverse isomorphisms.
\end{prop}

Now, we prove the \emph{self-intersection formula}.
\begin{prop}\label{key}
Let $Y$ be an oriented  submanifold in an oriented  smooth manifold $X$ with codimension $r$ and $\mathcal{L}$ a local system of $\mathbf{k}$-modules of finite rank  on $X$.
Denote by $i:Y\rightarrow X$  the inclusion and by $[Y]\in H_{dR}^r(X,\mathbf{k})$  the fundamental class of $Y$ in $X$.
Suppose that $\Phi$ is a paracompactifying family of supports on $X$.
Then the composite map
\begin{displaymath}
\xymatrix{
H_{\Phi|_Y}^{\bullet}(Y,\mathcal{L}|_Y)\ar[r]^{i_{*}} & H_{\Phi}^{\bullet+r}(X,\mathcal{L})\ar[r]^{i^*}&
H_{\Phi|_Y}^{\bullet+r}(Y,\mathcal{L}|_Y)
}
\end{displaymath}
is just $[Y]|_Y\cup\bullet$.
Moreover, if  the normal bundle $N_{Y/X}$ of $Y$ in $X$ has a complex vector bundle structure,
then $[Y]|_Y=c_r(N_{Y/X})$ is the $r$-th Chern class of $N_{Y/X}$.
\end{prop}

\begin{proof}
First, we have the following claim.
\vspace{2mm}

\noindent \textbf{Claim 1.} There exist an open neighborhood $U$ of $Y$ in $X$, a smooth map $\tau:U\rightarrow Y$ and  an open covering  $\mathfrak{U}=\{U_\alpha\}$  of $U$ satisfying  that $\mathcal{L}|_{U_\alpha}$ is constant, $U_\alpha \subseteq W_\alpha:=g^{-1}(U_\alpha)$ and $Y\cap W_\alpha=Y\cap U_\alpha$, where $l:Y\rightarrow U$ is the inclusion and $g=l\circ\tau:U\rightarrow U$.
\vspace{2mm}

Let $N$  be a tubular neighborhood of $Y$ in $X$.
Denote by $\tau$ the projection of the vector bundle $N$ onto $Y$ and by $l:Y\rightarrow N$ the inclusion.
Then $\tau\circ l=id_Y$.
For any $y\in Y$, choose an open neighborhood $V_y\subseteq N$ of $y$ such that $\mathcal{L}|_{V_y}$ is constant.
Set $U_y=\tau^{-1}(V_y\cap Y)\cap V_y$, which is an open neighborhood of $y$.
Then
\begin{equation}\label{set2}
l\circ\tau(U_y)\subseteq l(V_y\cap Y)\subseteq V_y\cap \tau^{-1}(V_y\cap Y)=U_y.
\end{equation}
Set $U=\bigcup\limits_{y\in Y}U_y$.
Clearly, $Y\subseteq U$.
Still denote by $l:Y\rightarrow U$ the inclusion and by $\tau:U\rightarrow Y$ the projection.
Set $g=l\circ\tau:U\rightarrow U$.
By (\ref{set2}), $g(U_y)\subseteq U_y$, i.e., $U_y\subseteq g^{-1}(U_y)$.
Evidently, $g\circ l=l$.
For any $x\in Y\cap g^{-1}(U_y)$, $x=l(x)=g\circ l(x)=g(x)\in U_y$, so $x\in Y\cap U_y$.
Hence $Y\cap g^{-1}(U_y)=Y\cap U_y$.
Then $U$, $\tau$ and $\{U_y|y\in Y\}$ satisfy the conditions in this claim.
\vspace{2mm}

Now, we choose $U$, $\tau$ and $\mathfrak{U}$ as the ones in Claim 1.
Denote by  $j:U\rightarrow X$ the inclusion.
Set $\Omega=j^{-1}\Phi\cap \tau^{-1}(\Phi|_Y)$.
By Proposition \ref{inverse-paracompact} $(1)$ $(2)$,  $l^{-1}(j^{-1}\Phi)=\Phi|_Y$, $l^{-1}(\tau^{-1}(\Phi|_Y))=\Phi|_Y$ and $l^{-1}\Omega=l^{-1}(j^{-1}\Phi)\cap l^{-1}(\tau^{-1}(\Phi|_Y))=\Phi|_Y$.
Since $l$ is proper and $l^{-1}(\mathcal{L}|_U)=l^{-1}(g^{-1}(\mathcal{L}|_U))=\mathcal{L}|_Y$,
$l$ induces four pushforwards $l_*$ satisfying the commutative diagrams
\begin{displaymath}
\xymatrix@R=0.5cm{
                &        \mbox{ } \Gamma_{\Omega}(U,\mathcal{L}|_U\otimes\mathcal{D}_U^{\prime\bullet+r}) \ar@{^{(}->}[dd]^{}     \\
  \Gamma_{\Phi|_Y}(Y,\mathcal{L}|_Y\otimes\mathcal{A}_Y^{\bullet}) \mbox{ }\ar[ur]^{l_*} \mbox{ }\ar[dr]_{l_*}                 \\
                &       \mbox{ }\quad  \Gamma_{j^{-1}\Phi}(U,\mathcal{L}|_U\otimes\mathcal{D}_U^{\prime\bullet+r}),                 }
\end{displaymath}
\begin{displaymath}
\xymatrix@R=0.5cm{
                &        \mbox{ } \Gamma_{\Omega}(U,g^{-1}(\mathcal{L}|_U)\otimes\mathcal{D}_U^{\prime\bullet+r}) \ar@{^{(}->}[dd]^{}     \\
  \Gamma_{\Phi|_Y}(Y,\mathcal{L}|_Y\otimes\mathcal{A}_Y^{\bullet}) \mbox{ }\ar[ur]^{l_*} \mbox{ }\ar[dr]_{l_*}                 \\
                &       \mbox{ }\quad  \Gamma_{\tau^{-1}(\Phi|_Y)}(U,g^{-1}(\mathcal{L}|_U)\otimes\mathcal{D}_U^{\prime\bullet+r}),                 }
\end{displaymath}
where   the vertical maps are inclusions.
For clearness, the pushforwards  in the first, second diagrams are written as $^\prime l_*$,  $^{\prime\prime} l_*$ respectively.

For any $U_\alpha\in \mathfrak{U}$, denote by $g_{\alpha}:W_\alpha=g^{-1}(U_\alpha)\rightarrow U_\alpha$  the restriction of $g$.
Since $\mathcal{L}|_{U_\alpha}$ is a constant local system,
$(g^{-1}(\mathcal{L}|_U))|_{W_\alpha}=g_{\alpha}^{-1}(\mathcal{L}|_{U_\alpha})$ is constant with the same rank with $\mathcal{L}|_{U_\alpha}$.
Suppose that $e^\alpha_1$, $\ldots$, $e^\alpha_m$ is a basis of $\Gamma(U_\alpha,\mathcal{L})$.
Then $g_{\alpha}^*e^\alpha_1$, $\ldots$, $g_{\alpha}^*e^\alpha_m$ is a basis of $\Gamma(W_\alpha,g^{-1}(\mathcal{L}|_U))$.
Let $u\in \Gamma_{\Phi|_Y}(Y,\mathcal{L}|_Y\otimes\mathcal{A}_Y^{\bullet})$ be any closed form.
Denote by  $l_{V}:Y\cap V\rightarrow V$ the restriction of $l$ for any open set $V\subseteq U$.
On $Y\cap U_\alpha$, $u=\sum\limits_{i=1}^m l_{U_\alpha}^*e_i^\alpha\otimes u_i^\alpha$
with  $u^\alpha_i\in \Gamma(Y\cap U_\alpha,\mathcal{A}_Y^\bullet)$ for $i=1$, \ldots $m$.
Meanwhile, $u=\sum\limits_{i=1}^m l_{W_\alpha}^*(g^*_{\alpha}e_i^\alpha)\otimes u_i^\alpha$ on $Y\cap W_\alpha=Y\cap U_\alpha$,
since $l_{U_\alpha}=g_{\alpha}\circ l_{W_\alpha}$.
Then
\begin{displaymath}
^\prime l_*u=\sum\limits_{i=1}^m e_i^\alpha\otimes l_{U_\alpha*}u_i^\alpha \mbox{ on } U_\alpha,
\end{displaymath}
\begin{displaymath}
^{\prime\prime} l_*u=\sum\limits_{i=1}^m g_{\alpha}^*e_i^\alpha\otimes l_{W_\alpha*}u_i^\alpha \mbox{ on } W_\alpha.
\end{displaymath}
By Proposition \ref{inverse-paracompact}, $\Omega$ is paracompactifying,
so we may assume that $^{\prime\prime} l_*u=v^{\prime\prime}+dT^{\prime\prime}$ for a closed $v^{\prime\prime}\in\Gamma_{\Omega}(U,g^{-1}(\mathcal{L}|_U)\otimes\mathcal{A}_U^{\bullet+r})$
and $T^{\prime\prime}\in\Gamma_{\Omega}(U,g^{-1}(\mathcal{L}|_U)\otimes\mathcal{D}_U^{\prime\bullet+r-1})$.
On $W_\alpha$, set $v^{\prime\prime}=\sum\limits_{i=1}^m g_{\alpha}^*e_i^\alpha\otimes v_i^{\prime\prime\alpha}$ and
$T^{\prime\prime}=\sum\limits_{i=1}^m g_{\alpha}^*e_i^\alpha\otimes T_i^{\prime\prime\alpha}$,
where $v_i^{\prime\prime\alpha}\in\Gamma(W_\alpha,\mathcal{A}^{\bullet+r}_U)$ and $T_i^{\prime\prime\alpha}\in\Gamma(W_\alpha,\mathcal{D}^{\prime\bullet+r-1}_U)$.
Set $v_i^{\prime\alpha}=v_i^{\prime\prime\alpha}|_{U_\alpha}$ and $T_i^{\prime\alpha}=T_i^{\prime\prime\alpha}|_{U_\alpha}$.
Then
\vspace{2mm}

\noindent \textbf{Claim 2.} $\{\sum\limits_{i=1}^m e_i^\alpha\otimes v_i^{\prime\alpha} \mbox{ on } U_\alpha|\mbox{ }U_\alpha\in\mathfrak{U}\}$ and
$\{\sum\limits_{i=1}^m e_i^\alpha\otimes T_i^{\prime\alpha} \mbox{ on }U_\alpha|\mbox{ }U_\alpha\in\mathfrak{U}\}$ respectively  piece together to give
$v^{\prime}\in\Gamma_{\Omega}(U,\mathcal{L}|_U\otimes\mathcal{A}_U^{\bullet+r})$ and $T^{\prime}\in\Gamma_{\Omega}(U,\mathcal{L}|_U\otimes\mathcal{D}_U^{\prime\bullet+r-1})$
satisfying that $^\prime l_*u=v^{\prime}+dT^{\prime}$.
\vspace{2mm}

Assume that $e_i^\alpha=\sum\limits_{j=1}^mc^{\alpha\beta}_{ij}e_j^\beta$ on $U_\alpha\cap U_\beta$ for the matrix  $\left(c^{\alpha\beta}_{ij}\right)_{1\leq i,j\leq m}\in GL_{m}(\mathbb{R})$.
Clearly, $g_\alpha^*e_i^\alpha=\sum\limits_{j=1}^mc^{\alpha\beta}_{ij}g_\beta^*e_j^\beta$ on $W_\alpha\cap W_\beta$.
Let $\left(d^{\alpha\beta}_{ij}\right)_{1\leq i,j\leq m}$ be the inverse matrix of $\left(c^{\alpha\beta}_{ij}\right)_{1\leq i,j\leq m}$.
Since $v^{\prime\prime}$ is a global form on $U$,
$v_i^{\prime\prime\alpha}=\sum\limits_{j=1}^md^{\alpha\beta}_{ji}v_j^{\prime\prime\beta}$ on $W_\alpha\cap W_\beta$,
which implies that $v_i^{\prime\alpha}=\sum\limits_{j=1}^md^{\alpha\beta}_{ji}v_j^{\prime\beta}$ on $U_\alpha\cap U_\beta$.
Hence $\{\sum\limits_{i=1}^m e_i^\alpha\otimes v_i^{\prime\alpha} \mbox{ on } U_\alpha|\mbox{ }U_\alpha\in\mathfrak{U}\}$ define
$v^{\prime}\in\Gamma(U,\mathcal{L}|_U\otimes\mathcal{A}_U^{\bullet+r})$ well.
Moreover,
\begin{displaymath}
\textrm{supp} v^{\prime}\cap U_\alpha= \bigcup\limits_{i=1}^m \textrm{supp} v_i^{\prime\alpha}=\bigcup\limits_{i=1}^m (\textrm{supp} v_i^{\prime\prime\alpha}\cap U_\alpha)=\textrm{supp} v^{\prime\prime}\cap U_\alpha,
\end{displaymath}
so $\textrm{supp} v^{\prime}=\textrm{supp} v^{\prime\prime}\in\Omega$,
i.e., $v^{\prime}\in\Gamma_{\Omega}(U,\mathcal{L}|_U\otimes\mathcal{A}_U^{\bullet+r})$.
Similarly, $T^{\prime}\in\Gamma_{\Omega}(U,\mathcal{L}|_U\otimes\mathcal{D}_U^{\prime\bullet+r-1})$ is defined well.
Since $U_\alpha\subseteq W_\alpha$, $l_{U_\alpha*}u^\alpha_i=(l_{W_\alpha*}u^\alpha_i)|_{U_\alpha}$ for any $i$.
Notice that $l_{W_\alpha*}u_i^\alpha=v_i^{\prime\prime\alpha}+dT_i^{\prime\prime\alpha}$.
So $l_{U_\alpha*}u_i^\alpha=v_i^{\prime\alpha}+dT_i^{\prime\alpha}$ for all $i$ and $\alpha$,
which implies that $^\prime l_*u=v^{\prime}+dT^{\prime}$.
The claim follows.
\vspace{2mm}

By (\ref{sub00}),
\begin{equation}\label{step-1}
l^*[^{\prime}l_{*}u]_{j^{-1}\Phi}
=l^*[^{\prime}l_{*}u]_{\Omega},
\end{equation}
\begin{equation}\label{step-2}
l^*[^{\prime\prime}l_{*}u]_{\tau^{-1}(\Phi|_Y)}
=l^*[^{\prime\prime}l_{*}u]_{\Omega}.
\end{equation}
Since $l^*_{U_\alpha}v_i^{\prime\alpha}=l^*_{W_\alpha}v_i^{\prime\prime\alpha}$ and $l_{U_\alpha}=g_\alpha\circ l_{W_\alpha}$,
\begin{displaymath}
l^*v^{\prime}=\sum\limits_{i=1}^m l^*_{U_\alpha}e_i^\alpha\otimes l^*_{U_\alpha}v_i^{\prime\alpha}
=\sum\limits_{i=1}^ml^*_{W_\alpha}g_\alpha^*e_i^\alpha\otimes l^*_{W_\alpha}v_i^{\prime\prime\alpha} =l^*v^{\prime\prime}
\end{displaymath}
on $Y\cap U_\alpha=Y\cap W_\alpha$ for all $\alpha$, hence  $l^*v^{\prime}=l^*v^{\prime\prime}$.
So
\begin{equation}\label{step-3}
l^*[^{\prime}l_{*}u]_{\Omega}
=[l^*v^{\prime}]_{\Phi|_Y}
=[l^*v^{\prime\prime}]_{\Phi|_Y}
=l^*[^{\prime\prime}l_{*}u]_{\Omega}.
\end{equation}
By (\ref{step-1})-(\ref{step-3}),
\begin{equation}\label{step-4}
l^*[^{\prime}l_{*}u]_{j^{-1}\Phi}=l^*[^{\prime\prime}l_*u]_{\tau^{-1}(\Phi|_Y)}.
\end{equation}
By (\ref{pro-formula1-dR}),
\begin{equation}\label{sub}
^{\prime\prime}l_*u=l_*(l^*\tau^*u)=l_{*}(1)\wedge \tau^{\ast}u\in\Gamma_{\tau^{-1}(\Phi|_Y)}(U,g^{-1}(\mathcal{L}|_U)\otimes\mathcal{D}_U^{\prime\bullet+r}).
\end{equation}
Notice that $l_*(1)$ is just the current on $U$ defined by the integral along $Y$.
We have $[l_*(1)]=[Y]|_U$.
Then
\begin{equation}\label{step-5}
\begin{aligned}
l^*[^{\prime\prime}l_{*}u]_{\tau^{-1}(\Phi|_Y)}
=&l^*[l_{*}(1)\wedge \tau^{\ast}u]_{\tau^{-1}(\Phi|_Y)}\qquad\qquad\mbox{ }( \mbox{by (\ref{sub})})\\
=&l^*\left([l_{*}(1)]\cup [\tau^{\ast}u]_{\tau^{-1}(\Phi|_Y)}\right)\\
=&[Y]|_Y\cup [u]_{\Phi|_Y}.
\end{aligned}
\end{equation}
Since the support of $l_{*}u$ is closed in $X$, $j_*l_{*}u$ is defined well and is just $i_{*}u$.
By (\ref{pro-formula1-dR}),
\begin{equation}\label{mid}
j^*i_{*}u=^{\prime}l_{*}u\in \Gamma_{j^{-1}\Phi}(U,\mathcal{L}|_U\otimes\mathcal{D}_U^{\prime\bullet+r}).
\end{equation}
Then
\begin{displaymath}
\begin{aligned}
i^*i_{*}[u]_{\Phi|_{Y}}
=&l^*[j^*i_{*}u]_{j^{-1}\Phi}\\
=&l^*[^{\prime}l_{*}u]_{j^{-1}\Phi}\qquad\qquad\mbox{ }( \mbox{by (\ref{mid})})\\
=&[Y]|_Y\cup [u]_{\Phi|_{Y}}.\qquad\quad( \mbox{by (\ref{step-4}) and (\ref{step-5})})
\end{aligned}
\end{displaymath}

Suppose that $N_{Y/X}$ has a complex vector bundle structure and denote by $e(N_{Y/X})$ the Euler class of $N_{Y/X}$.
By \cite[(20.10.6)]{BT}, $e(N_{Y/X})=c_r(N_{Y/X})$ and by \cite[Proposition 12.4]{BT},  $e(N_{Y/X})=[Y]|_Y$.
So $c_r(N_{Y/X})=[Y]|_Y$.
We complete the proof.
\end{proof}

\subsection{Generalized blow-ups}
For convenience,  still denote by $X$ the zero section of  the vector bundle $F$ over $X$.
\subsubsection{Generalized blow-ups}
Recall a McDuff's construction \cite[Definition 2.2]{Mc} with a slight modification as follows:
For a submanifold  $Y$ of a smooth manifold $X$, assume that \emph{the normal bundle $N=N_{Y/X}$ is equipped with a complex vector bundle structure}.
Let $U\subseteq N$ be an open or a closed neighbourhood of the zero section $Y$ of $N$ and let $l:U\rightarrow X$ be a smooth embedding  of $U$ onto an open or a closed neighbourhood $W$ of $Y$ with $l|_Y=\textrm{id}_Y$.
Such $U$ and $l$ always exist (but not unique) by the tubular neighborhood theorem.
Set $N_0=N-Y$ and $\mathcal{O}_{\mathbb{P}(N)}(-1)_0=\mathcal{O}_{\mathbb{P}(N)}(-1)-\mathbb{P}(N)$.
There is a commutative diagram
\begin{equation}\label{blow-up-vb}
\xymatrix{
   \mathcal{O}_{\mathbb{P}(N)}(-1)_0\ar[d]^{\cong} \ar@{^{(}->}[r] & \mathcal{O}_{\mathbb{P}(N)}(-1)\ar[d]^{\varphi}\ar[r]^{\quad\psi}& \mathbb{P}(N)\ar[d]^{p}\\
 N_0   \ar@{^{(}->}[r] & N\ar[r]_{\tau}& Y},
\end{equation}
where $\psi$ and $\varphi$ are induced by the projections from $\mathcal{O}_{\mathbb{P}(N)}(-1)\subseteq\mathbb{P}(N)\times N$ onto $\mathbb{P}(N)$ and $N$ respectively.
Set $\widetilde{U}=\varphi^{-1}(U)\subseteq \mathcal{O}_{\mathbb{P}(N)}(-1)$.
The composition of $\varphi|_{\widetilde{U}-\mathbb{P}(N)}$ and $l|_{U-Y}$ gives a diffeomorphism $\phi:\widetilde{U}-\mathbb{P}(N)\tilde{\rightarrow} U-Y \tilde{\rightarrow} W-Y$.
Define
\begin{displaymath}
\widetilde{X}=(X-Y)\cup_{\phi}\widetilde{U},
\end{displaymath}
where  $W-Y\subseteq X-Y$ is identified with $\widetilde{U}-\mathbb{P}(N)\subseteq \widetilde{U}$  via  $\phi$.
Gluing the inclusion $X-Y\hookrightarrow X$ and $l\circ \varphi|_{\widetilde{U}}:\widetilde{U}\rightarrow U\rightarrow X$ via $\phi$, we get a smooth map $\pi: \widetilde{X}\rightarrow X$.
The smooth manifold $\widetilde{X}$ and the map $\pi: \widetilde{X}\rightarrow X$ \emph{depends on} the choices of the complex vector bundle structure of $N_{Y/X}$ and the embedding $l:U\rightarrow X$.
We say that $\pi:\widetilde{X}\rightarrow X$ is the \emph{generalized blow-up} of $X$ along $Y$ associated to the  complex vector bundle structure of $N_{Y/X}$ and the embedding $l:U\rightarrow X$.
Moreover, $E=\pi^{-1}(Y)$ is said to be the \emph{exceptional divisor}.
Clearly, symplectic blow-ups \cite[Definition 2.2]{Mc} and locally conformal symplectic blow-ups \cite[Definition 3.4]{YYZ} are generalized blow-ups.

The following lemma may be well known for experts and we don't find the references.
For readers' convenience, we will give a proof.
\begin{lem}\label{NB-vb}
Let  $F$ be a smooth  vector bundle  over a smooth  manifold $X$.
Then the normal bundle $N_{X/F}$ is isomorphic to $F$ over $X$.
\end{lem}
\begin{proof}
Let $U_\alpha$ be a coordinate chart of $X$ with a trivialization $F|_{U_\alpha}\cong U_\alpha \times \mathbb{R}^r$.
Denote by $g_{\alpha\beta}$ the translation functions of the vector bundle $F$ over $U_\alpha\cap U_\beta$.
Suppose that $x_\alpha^1,\ldots, x_\alpha^n$ and $v_\alpha^1,\ldots,v_\alpha^r$ are coordinates of $U_\alpha$ and $\mathbb{R}^r$ respectively.
Then $(v_\beta^1,\ldots, v_\beta^r)=(v_\alpha^1,\ldots, v_\alpha^r)\cdot g_{\alpha\beta}$.
The tangent bundles $\textrm{T}F|_{F|_{U_\alpha}}$ and $\textrm{T}X|_{U_\alpha}$ are linearly generated by $\frac{\partial}{\partial x_\alpha^1}$, \ldots, $\frac{\partial}{\partial x_\alpha^n}$, $\frac{\partial}{\partial v_\alpha^1}$, \ldots, $\frac{\partial}{\partial v_\alpha^r}$ and $\frac{\partial}{\partial x_\alpha^1}$, \ldots, $\frac{\partial}{\partial x_\alpha^n}$ respectively.
Hence $N_{X/F}|_{U_\alpha}$ is linearly generated by $\frac{\partial}{\partial v_\alpha^1}$, \ldots, $\frac{\partial}{\partial v_\alpha^r}$ modulo $\textrm{T}X|_{U_\alpha}$, which gives a natural trivialization $N_{X/F}|_{U_\alpha}\tilde{\rightarrow} U_\alpha\times \mathbb{R}^r$.
Under such trivializations, the transition functions of $N_{X/F}$  is just $g_{\alpha\beta}$ on $U_\alpha\cap U_\beta$. So $N_{Y/F}\cong F$
\end{proof}

\begin{lem}\label{orientable}
Let $F$ be a smooth fiber bundle over an oriented smooth manifold $X$.
Assume that there exists an open covering $\mathfrak{U}$ with trivializations $h_U:F|_U\tilde{\rightarrow} U \times Z$  for all $U\in \mathfrak{U}$ satisfying that:

$(i)$ $Z$ is a complex manifold

$(ii)$ For any $U$, $V\in \mathfrak{U}$ and any $x\in U\cap V$, $pr_2\circ h_{UV}(x,\cdot):Z\rightarrow Z$ is holomorphic, where $h_{UV}=h_V\circ h_U^{-1}:(U\cap V) \times Z\tilde{\rightarrow} (V\cap U) \times Z$ and $pr_2:(V\cap U) \times Z\rightarrow Z$ is the
second projection.\\
Then $F$ is orientable.
In particular,  a complex vector bundle $F$ on an oriented smooth manifold  and its complex projectivization $\mathbb{P}(F)$  are both orientable.
\end{lem}
\begin{proof}
Let $\{V_\mu|\mu\in I\}$ be an open covering  of $Z$ such that $V_\mu$ are holomorphic coordinate charts for all $\mu\in I$.
Without loss of generality, assume that $\mathfrak{U}=\{U_\alpha|\alpha\in J\}$  consists of  coordinate charts of $X$.
Since $X$ is orientable, we can choose coordinates  $x_\alpha^1$, \ldots, $x_\alpha^n$ for $U_\alpha$  such that  $\textrm{det}\left(\frac{\partial x_\beta^i}{\partial x_\alpha^j}\right)_{1 \leq i \leq n \atop 1 \leq j \leq n} > 0$ on $U_\alpha\cap U_\beta$ for any $\alpha,\mbox{ }\beta\in J$.
Let $z_\mu^1=u_\mu^1+\sqrt{-1}v_\mu^1$, \ldots, $z_\mu^r=u_\mu^r+\sqrt{-1}v_\mu^r$ be the holomorphic coordinates of $V_\mu$.
Set $h_{\alpha\beta}=h_{U_\alpha U_\beta}$.
By the assumption $(ii)$,  $\frac{\partial (z_\lambda^k\circ h_{\alpha\beta})}{\partial\bar{z}_\mu^l}=0$, i.e.,
$\frac{\partial (u_\lambda^k\circ h_{\alpha\beta})}{\partial u_\mu^l}=\frac{\partial (v_\lambda^k\circ h_{\alpha\beta})}{\partial v_\mu^l}$ and $\frac{\partial (u_\lambda^k\circ h_{\alpha\beta})}{\partial v_\mu^l}=-\frac{\partial (v_\lambda^k\circ h_{\alpha\beta})}{\partial u_\mu^l}$.
Notice that $x_\beta^i\circ h_{\alpha\beta}$ is only dependent on $x_\alpha^1$, \ldots, $x_\alpha^n$, hence $\frac{\partial (x_\beta^i\circ h_{\alpha\beta})}{\partial u_\mu^j}=0$ and  $\frac{\partial (x_\beta^i\circ h_{\alpha\beta})}{\partial v_\mu^j}=0$.
Consider the coordinate charts $\left(h_{U_\alpha}^{-1}(U_\alpha\times V_\mu),x^1_\alpha,\ldots, x^n_\alpha,u^1_\mu,\ldots,u^r_\mu,v^1_\mu,\ldots,v^r_\mu\right)$ for $\alpha\in J$, $\mu\in I$ of $F$.
The Jacobi of $h_{\alpha\beta}$ is
\begin{displaymath}
\begin{aligned}
&\textrm{det}\left(\begin{array}{lll}
\left(\frac{\partial (x_\beta^i\circ h_{\alpha\beta})}{\partial x_\alpha^j}\right)_{1 \leq i \leq n \atop 1 \leq j \leq n} & \qquad\textrm{O}_{n\times r} & \qquad\textrm{O}_{n\times r} \\
\left(\frac{\partial (u_\lambda^k\circ h_{\alpha\beta})}{\partial x_\alpha^j}\right)_{1 \leq k \leq r \atop 1 \leq j \leq n} & \left(\frac{\partial (u_\lambda^k\circ h_{\alpha\beta})}{\partial u_\mu^l}\right)_{1 \leq k \leq r \atop 1 \leq l \leq r} & \left(\frac{\partial (u_\lambda^k\circ h_{\alpha\beta})}{\partial v_\mu^l}\right)_{1 \leq k \leq r \atop 1 \leq l \leq r} \\
\left(\frac{\partial (v_\lambda^k\circ h_{\alpha\beta})}{\partial x_\alpha^j}\right)_{1 \leq k \leq r \atop 1 \leq j \leq n} & \left(\frac{\partial (v_\lambda^k\circ h_{\alpha\beta})}{\partial u_\mu^l}\right)_{1 \leq k \leq r \atop 1 \leq l \leq r} & \left(\frac{\partial (v_\lambda^k\circ h_{\alpha\beta})}{\partial v_\mu^l}\right)_{1 \leq k \leq r \atop 1 \leq l \leq r}
\end{array}\right)\\
=&\textrm{det}\left(\frac{\partial x_\beta^i}{\partial x_\alpha^j}\right)_{1 \leq i \leq n \atop 1 \leq j \leq n}
\cdot 2^{2r} \cdot
\textrm{det}\left(\begin{array}{lll}
\quad\left(\frac{\partial (u_\lambda^k\circ h_{\alpha\beta})}{\partial z_\mu^l}\right)_{1 \leq k \leq r \atop 1 \leq l \leq r} & \qquad\textrm{O}_{r\times r} \\
-\left(\frac{\partial (u_\lambda^k\circ h_{\alpha\beta})}{\partial v_\mu^l}\right)_{1 \leq k \leq r \atop 1 \leq l \leq r} & \left(\overline{\frac{\partial (u_\lambda^k\circ h_{\alpha\beta})}{\partial z_\mu^l}}\right)_{1 \leq k \leq r \atop 1 \leq l \leq r}
\end{array}\right)\\
=& \textrm{det}\left(\frac{\partial x_\beta^i}{\partial x_\alpha^j}\right)_{1 \leq i \leq n \atop 1 \leq j \leq n}
\cdot
\left|\textrm{det}
\left(\frac{\partial (z_\lambda^k\circ h_{\alpha\beta})}{\partial z_\mu^l}\right)_{1 \leq k \leq r \atop 1 \leq l \leq r}\right|^2 > 0,
\end{aligned}
\end{displaymath}
where we use the fact that
$\textrm{det}
\left(\begin{array}{lll}
A & B \\
-B & A
\end{array}\right)
=\left| \textrm{det}(A-iB) \right|^2$
for any $r\times r$ matrices $A$, $B$.
Hence $F$ is orientable.
\end{proof}

Now, we consider the properties of generalized blow-ups.
\begin{prop}\label{blow-up-properties}
Let $\pi:\widetilde{X}\rightarrow X$ be a generalized blow-up of $X$ along $Y$  associated to a complex  vector bundle structure of $N_{Y/X}$ and an embedding $l:U\rightarrow X$. Denote by $E$ its exceptional divisor.
They satisfy the properties:

$(1)$ For any open set $V\subseteq X$, $\pi:\pi^{-1}(V)\rightarrow V$ is the generalized blow-up of $V$ along $Y\cap V$  associated to the complex  vector bundle structure of $N_{Y\cap V/V}=N_{Y/X}|_{Y\cap V}$ induced by $N_{Y/X}$ and the embedding $l^{\prime}:l^{-1}(V)\rightarrow V$ induced by $l$.
In particular, $\pi|_{\widetilde{X}-E}:\widetilde{X}-E\rightarrow X-Y$ is a diffeomorphism.

$(2)$ The map $\pi$ is surjective and proper.

$(3)$ $E$ is a submanifold  of  $\widetilde{X}$ with codimension $2$ and $\pi|_E:E\rightarrow Y$ is diffeomorphic to the complex projectivization $\mathbb{P}(N_{Y/X})$ over $Y$.

$(4)$ The normal bundle $N_{E/\widetilde{X}}$ is isomorphic to $\mathcal{O}_{E}(-1)$ as smooth vector bundles over $E=\mathbb{P}(N_{Y/X})$ and hence $N_{E/\widetilde{X}}$ has a natural complex vector bundle structure induced by $\mathcal{O}_{E}(-1)$.

$(5)$ Assume that $U^{\prime}\subseteq U$ is an open or a closed neighbourhood of the zero section of $N_{Y/X}$. Let $\pi^{\prime}:\widetilde{X}^{\prime}\rightarrow X$ be the blow-up of $X$ along $Y$  associated to the given complex vector bundle structure of $N_{Y/X}$ and the embedding   $l|_{U^{\prime}}:U^{\prime}\rightarrow X$. Then $\widetilde{X}^{\prime}$ is diffeomorphic to $\widetilde{X}$ over $X$.

$(6)$ If $X$ and $Y$ are orientable, then $\widetilde{X}$ and $E$ are orientable.
\end{prop}
\begin{proof}
Evidently, $(1)$ holds and $\pi$ is surjective by the definition.
Assume that $N$, $\varphi$, $\psi$, $\phi$, $\widetilde{U}$ and $W$ are the ones in the definition of generalized blow-ups.
Let $T\subseteq W$ be a tubular neighborhood of $Y$.
Given a metric on the smooth vector bundle $T$ over $Y$ and denote by $T_r$, $\overline{T}_r$ the open, closed disc bundles with radius $r$ respectively.
Then $X-T_{1/2}\subseteq X-Y$ and $\overline{T}_1\subseteq W$.
For any compact set $K\subseteq X$,  $K\cap (X-T_{1/2})$ and $(l\circ \varphi|_{\widetilde{U}})^{-1}(K\cap \overline{T}_1)$  are compact.
Under the quotient map $(X-Y)\sqcup \widetilde{U}\rightarrow\widetilde{X}$, $\pi^{-1}(K)$ is the image of $(K\cap (X-T_{1/2}))\sqcup(l\circ \varphi|_{\widetilde{U}})^{-1}(K\cap \overline{T}_1)$ and then is compact.
Hence $\pi$ is proper. We proved $(2)$.
By the definition, $\psi$ restricts to a diffeomorphism  $E=\varphi^{-1}(Y)\tilde{\rightarrow} \mathbb{P}(N)$ over $Y$, which implies $(3)$.
By Lemma \ref{NB-vb},
\begin{displaymath}
N_{E/\widetilde{X}}=N_{E/\widetilde{U}}=N_{\mathbb{P}(N)/\mathcal{O}_{\mathbb{P}(N)}(-1)}\cong \mathcal{O}_{\mathbb{P}(N)}(-1),
\end{displaymath}
i.e., $(4)$ holds.
For $(5)$, gluing the identity $id:X-Y\rightarrow X-Y$ and the inclusion $\widetilde{U}^{\prime}=\varphi^{-1}(U^{\prime})\hookrightarrow \widetilde{U}$ gives a diffeomorphism between $\widetilde{X}^{\prime}$ and $\widetilde{X}$ over $X$.
Assume that $X$ and $Y$ are orientable.
By Lemma \ref{orientable}, $E=\mathbb{P}(N)$ and then $\mathcal{O}_{\mathbb{P}(N)}(-1)$ are orientable, so is $\widetilde{U}\subseteq \mathcal{O}_{\mathbb{P}(N)}(-1)$.
Clearly, $X-Y$ is orientable.
Choose suitable orientations for $\mathcal{O}_{\mathbb{P}(N)}(-1)$ and $X-Y$ such that $\phi:\widetilde{U}-\mathbb{P}(N)\tilde{\rightarrow} W-Y$ preserves orientations, which gives an orientation on  $\widetilde{X}$.
We obtain $(6)$.
\end{proof}

\subsubsection{A class of generalized blow-ups}
We recall the blow-up of a holomorphic ideal in a smooth manifold defined in \cite{BCv}.
Denote by $\mathcal{C}^\infty_X$ the sheaf of germs of complex valued smooth functions on a smooth manifold $X$.
\begin{defn}\label{ideal}
$(1)$ (\cite[Definition 3.1]{BCv})
Let $Y$ be a submanifold of a smooth manifold $X$ with codimension  $2r$ for $r\geq 1$.
A \emph{holomorphic ideal} for $Y$ in $X$ is an ideal sheaf $\mathcal{I}_Y\subseteq \mathcal{C}^\infty_X$ satisfying that:

$(i)$ $\mathcal{I}_Y|_{X-Y}= \mathcal{C}^\infty_X|_{X-Y}$.

$(ii)$ For any $y\in Y$, there exists an open neighborhood $U$ of $y$ and $z_1$, $\ldots$, $z_r$ $\in \mathcal{I}_Y(U)$, such that $\mathcal{I}_Y|_U$ is generated by $z_1$, $\ldots$, $z_r$ and $z=(z_1,\ldots, z_r):U\rightarrow\mathbb{C}^r$ is a submersion with $z^{-1}(o)=Y\cap U$.

$(2)$ (\cite[Definition 3.4]{BCv}) A \emph{divisor} on a smooth manifold $X$ is an ideal sheaf $\mathcal{I}\subseteq \mathcal{C}^\infty_X$ which can be locally generated by a single function  and whose zero set is nowhere dense in $X$.
\end{defn}

Suppose that $f:Y\rightarrow X$ is a smooth map of smooth manifolds and $\mathcal{I}$ is an ideal sheaf  of $\mathcal{C}^\infty_X$.
Denote by $f^{-1}\mathcal{I}\cdot \mathcal{C}^\infty_Y\subseteq \mathcal{C}_Y^\infty$ the image of the natural morphism $f^*\mathcal{I}=f^{-1}\mathcal{I}\otimes_{f^{-1}\mathcal{C}_X^\infty}\mathcal{C}_Y^\infty\rightarrow f^{-1}\mathcal{C}_X^\infty\otimes_{f^{-1}\mathcal{C}_X^\infty}\mathcal{C}_Y^\infty\cong \mathcal{C}_Y^\infty$ induced by the inclusion $\mathcal{I}\hookrightarrow \mathcal{C}_X^\infty$.
For the inclusion $j:U\hookrightarrow X$ of an open set $U$, $j^{-1}\mathcal{I}\cdot \mathcal{C}^\infty_U=\mathcal{I}|_U$.
Assume that $g:Z\rightarrow Y$ is a smooth map of smooth manifolds.
By $g^*f^*\mathcal{I}\cong(f\circ g)^*\mathcal{I}$, we easily get
\begin{equation}\label{composition}
g^{-1}(f^{-1}\mathcal{I}\cdot\mathcal{C}_Y^\infty)\cdot\mathcal{C}_Z^\infty\cong (f\circ g)^{-1}\mathcal{I}\cdot \mathcal{C}_Z^\infty.
\end{equation}
\begin{defn}[{\cite[Definition 3.6]{BCv}}]\label{real blow-up}
Let $\mathcal{I}_Y$ be a holomorphic ideal for a submanifold $Y$ in a smooth manifold $X$.
The \emph{blow-up} of $\mathcal{I}_Y$ in $X$ is defined as a smooth map $\pi:\widetilde{X}\rightarrow X$ between smooth manifolds such that $\pi^{-1}\mathcal{I}_Y\cdot\mathcal{C}_{\widetilde{X}}^\infty$ is a divisor on $\widetilde{X}$ and the following universal property holds:
For any smooth map $f:Z\rightarrow X$ such that $f^{-1}\mathcal{I}_Y\cdot\mathcal{C}_{Z}^\infty$ is a divisor, there is a unique smooth map $g:Z\rightarrow \widetilde{X}$  such that  $\pi\circ g=f$.
\end{defn}

\begin{thm}[{\cite[Theorem 3.7]{BCv}}]\label{exist unique real blow-up}
Given a holomorphic ideal $\mathcal{I}_Y$ for a submanifold $Y$ in a smooth manifold $X$, there exists a unique blow-up $\pi:\widetilde{X}\rightarrow X$ of $\mathcal{I}_Y$ in $X$ up to unique isomorphism.
\end{thm}

We have the following properties.
\begin{prop}\label{blow-up fund}
Suppose that $X$ is a smooth manifold and  $\mathcal{I}_Y\subseteq \mathcal{C}^\infty_X$ is a holomorphic ideal for a submanifold $Y$ in $X$.

$(1)$ Let $\pi:\widetilde{X}\rightarrow X$ be the blow-up of $\mathcal{I}_Y$ in $X$.

$(i)$ For any open set $U\subseteq X$, $\pi|_{\pi^{-1}(U)}:\pi^{-1}(U)\rightarrow U$ is the blow-up of $\mathcal{I}_Y|_U$ in $U$.

$(ii)$ Assume that  $Z$ is a smooth manifold and $pr_2:Z\times X\rightarrow X$ is the second projection.
Then $pr_2^{-1}\mathcal{I}_Y\cdot\mathcal{C}^\infty_{Z\times X}$ is a holomorphic ideal for $Z\times Y$ in $Z\times X$ and $id_Z\times \pi:Z\times \widetilde{X}\rightarrow Z\times X$ is the blow-up of $pr_2^{-1}\mathcal{I}_Y\cdot\mathcal{C}^\infty_{Z\times X}$ in $Z\times X$.

$(2)$ Suppose that $\pi:\widetilde{X}\rightarrow X$ is a smooth map and $\mathfrak{U}$ is an open covering of $X$ such that $\pi|_{\pi^{-1}(U)}:\pi^{-1}(U)\rightarrow U$ is the blow-up of $\mathcal{I}_Y|_U$ in $U$ for any $U\in \mathfrak{U}$.
Then $\pi:\widetilde{X}\rightarrow X$ is the blow-up of $\mathcal{I}_Y$ in $X$.
\end{prop}
\begin{proof}
$(1)$ $(i)$ Evidently, $(\pi|_{\pi^{-1}(U)})^{-1}(\mathcal{I}_Y|_U)\cdot\mathcal{C}^\infty_{\pi^{-1}(U)}=(\pi^{-1}\mathcal{I}_Y\cdot\mathcal{C}_{\widetilde{X}}^\infty)|_{\pi^{-1}(U)}$ is a divisor on $\pi^{-1}(U)$.
Suppose that $j_U:U\rightarrow X$  and $\tilde{j}_U:\pi^{-1}(U)\rightarrow \widetilde{X}$ are the inclusions.
Let $f:Z\rightarrow U$ be any smooth map of smooth manifolds such that $f^{-1}(\mathcal{I}_Y|_U)\cdot\mathcal{C}_Z^\infty$ is a divisor on $Z$.
By (\ref{composition}), $(j_U\circ f)^{-1}\mathcal{I}_Y\cdot \mathcal{C}^\infty_Z=f^{-1}(\mathcal{I}_Y|_U)\cdot\mathcal{C}_Z^\infty$ is a divisor on $Z$.
By the universal property of the blow-up $\pi$, there is a smooth map $\tilde{f}_1:Z\rightarrow \widetilde{X}$ such that $\pi\circ \tilde{f}_1=j_U\circ f$.
So $\tilde{f}_1(Z)\subseteq \pi^{-1}(U)$.
Then $\tilde{f}_1$ induces a smooth map $\tilde{f}:Z\rightarrow \pi^{-1}(U)$, i.e., $\tilde{j}_U\circ \tilde{f} =\tilde{f}_1$.
Since $j_U$ is injective, $\pi|_{\pi^{-1}(U)}\circ \tilde{f}=f$.
Let  $g:Z\rightarrow \pi^{-1}(U)$ be another smooth map such that $\pi|_{\pi^{-1}(U)}\circ g=f$.
Then $\pi\circ(\tilde{j}_U\circ g)=j_U\circ f$.
By the universal property of the blow-up $\pi$, $\tilde{j}_U\circ g=\tilde{f}_1$, so $g=\tilde{f}$.

$(ii)$
By (\ref{composition}), $(pr_2^{-1}\mathcal{I}_Y\cdot\mathcal{C}^\infty_{Z\times X})|_{Z\times X-Z\times Y}=pr_2^{-1}(\mathcal{I}_Y|_{X-Y})\cdot\mathcal{C}^\infty_{Z\times X}=\mathcal{C}^\infty_{Z\times X}$.
For any $(z,y)\in Z\times Y$, there exists an open neighborhood $U$ of $y$ and $u_1$, \ldots, $u_r$ $\in \mathcal{I}_Y(U)$ such that $\mathcal{I}_Y|_Y=\langle u_1,\ldots, u_r\rangle$ and $u=(u_1,\ldots,u_r):U\rightarrow \mathbb{C}^r$ is a submersion with $u^{-1}(o)=Y\cap U$.
Set $v_i=pr_2^*u_i\in\mathcal{C}^\infty_{Z\times X}(Z\times U)$.
Then $(pr_2^{-1}\mathcal{I}_Y\cdot\mathcal{C}^\infty_{Z\times X})|_{Z\times U}=\langle v_1,\ldots,v_r\rangle$ and $v=(v_1,\ldots,v_r):Z\times U\rightarrow \mathbb{C}^r$ is a submersion with $u^{-1}(o)=Z\times (Y\cap U)$.
So $pr_2^{-1}\mathcal{I}_Y\cdot\mathcal{C}^\infty_{Z\times X}$ is a holomorphic ideal for $Z\times Y$ in $Z\times X$.
Let $pr^{\prime}_1$ and $pr^{\prime}_2$ be the projections from $Z\times \widetilde{X}$ onto $Z$ and $\widetilde{X}$ respectively.
Since $\pi^{-1}\mathcal{I}_Y\cdot\mathcal{C}^\infty_{\widetilde{X}}$ is a divisor on $\widetilde{X}$, so is  $(id_Z\times \pi)^{-1}(pr_2^{-1}\mathcal{I}_Y\cdot\mathcal{C}_{Z\times X}^\infty)\cdot \mathcal{C}^\infty_{Z\times \widetilde{X}}=pr_2^{\prime-1}(\pi^{-1}\mathcal{I}_Y\cdot\mathcal{C}^\infty_{\widetilde{X}})\cdot\mathcal{C}^\infty_{Z\times \widetilde{X}}$  on $Z\times \widetilde{X}$.
Let $f:M\rightarrow Z\times X$ be any smooth map of smooth manifolds such that $f^{-1}(pr_2^{-1}\mathcal{I}_Y\cdot\mathcal{C}^\infty_{Z\times Y})\cdot\mathcal{C}_M^\infty$ is a divisor on $M$.
Then $(pr_2\circ f)^{-1}\mathcal{I}_Y\cdot \mathcal{C}^\infty_M=f^{-1}(pr_2^{-1}\mathcal{I}_Y\cdot\mathcal{C}^\infty_{Z\times X})\cdot\mathcal{C}^\infty_M$ is a divisor on $M$.
By the universal property of the blow-up $\pi$, there a unique smooth map $\tilde{f}_1:M\rightarrow X$ such that $\pi\circ\tilde{f}_1=pr_2\circ f$.
Define $\tilde{f}:M\rightarrow Z\times \widetilde{X}$ as $m\mapsto (pr_1\circ f(m),\tilde{f}_1(m))$.
Clearly, $(id_Z\times \pi)\circ\tilde{f}=f$.
Suppose that $g:M\rightarrow Z\times \widetilde{X}$ is another smooth map satisfying that $(id_Z\times \pi)\circ g=f$.
Then $\pi\circ (pr^{\prime}_2\circ g)=pr_2\circ (id_Z\times \pi)\circ g=pr_2\circ f=\pi\circ\tilde{f}_1$.
So $pr^{\prime}_2\circ g=\tilde{f}_1$ by the universal property of the blow-up $\pi$.
Moreover, $pr^{\prime}_1\circ g=pr_1\circ (id_Z\times \pi)\circ g=pr_1\circ f$.
Hence $g=\tilde{f}$.

$(2)$ By the assumption,  $(\pi^{-1}\mathcal{I}_Y\cdot\mathcal{C}^\infty_{\widetilde{X}})|_{\pi^{-1}(U)}=(\pi|_{\pi^{-1}(U)})^{-1}(\mathcal{I}_Y|_{U})\cdot\mathcal{C}^\infty_{\pi^{-1}(U)}$ is a divisor on $\pi^{-1}(U)$  for any $U\in \mathfrak{U}$, so is $\pi^{-1}\mathcal{I}_Y\cdot\mathcal{C}^\infty_{\widetilde{X}}$ on $\widetilde{X}$.
Let $f:Z\rightarrow  X$ be any smooth map of smooth manifolds such that $f^{-1}\mathcal{I}_Y\cdot\mathcal{C}^\infty_Z$ is a divisor on $Z$.
By (\ref{composition}), $(f|_{f^{-1}(U)})^{-1}(\mathcal{I}_Y|_U)\cdot\mathcal{C}^\infty_{f^{-1}(U)}=(f^{-1}\mathcal{I}_Y\cdot\mathcal{C}^\infty_{f^{-1}(U)})|_{f^{-1}(U)}$ is a divisor on $f^{-1}(U)$.
By the universal property of the blow-up $\pi|_{\pi^{-1}(U)}$, there a unique smooth map $\tilde{f}_U:f^{-1}(U)\rightarrow \pi^{-1}(U)$ such that $\pi|_{\pi^{-1}(U)}\circ\tilde{f}_U=f|_{f^{-1}(U)}$.
For any $U$, $V$ $\in\mathfrak{U}$,  $\pi|_{\pi^{-1}(U\cap V)}\circ\tilde{f}_U|_{f^{-1}(U)\cap f^{-1}(V)}=f|_{f^{-1}(U\cap V)}$ and $\pi|_{\pi^{-1}(U\cap V)}\circ\tilde{f}_V|_{f^{-1}(U)\cap f^{-1}(V)}=f|_{f^{-1}(U\cap V)}$.
By $(1)$ $(i)$, $\pi|_{\pi^{-1}(U\cap V)}:\pi^{-1}(U\cap V)\rightarrow U\cap V$ is the blow-up of $\mathcal{I}_Y|_{U\cap V}$ in $U\cap V$.
Notice that $(f|_{f^{-1}(U\cap V)})^{-1}(\mathcal{I}_Y|_{U\cap V})\cdot\mathcal{C}^\infty_{f^{-1}(U\cap V)}=(f^{-1}\mathcal{I}_Y\cdot\mathcal{C}^\infty_{Z})|_{f^{-1}(U\cap V)}$ is a divisor on $f^{-1}(U\cap V)$.
So $\tilde{f}_U|_{f^{-1}(U)\cap f^{-1}(V)}=\tilde{f}_V|_{f^{-1}(U)\cap f^{-1}(V)}$ by the universal property of the blow-up $\pi|_{\pi^{-1}(U)}$.
We obtain a smooth map $\tilde{f}:Z\rightarrow \widetilde{X}$ satisfying that $\tilde{f}|_{f^{-1}(U)}=\tilde{f}_U$ for any $U\in \mathfrak{U}$.
Clearly, $\pi\circ \tilde{f}=f$.
Assume that $g:Z\rightarrow \widetilde{X}$ is another smooth map such that $\pi\circ g=f$.
By the universal property of the blow-up $\pi|_{\pi^{-1}(U)}$, $g|_{f^{-1}(U)}=\tilde{f}_U$ for any $U\in \mathfrak{U}$, which implies that $g=\tilde{f}$.
\end{proof}

M. Bailey, G. Cavalcanti and J. van der Leer Dur\'{a}n \cite[p. 2114]{BCv} defined a canonical holomorphic ideal for the zero section in the normal bundle $N_{Y/X}$.
We generalize their definition on general complex vector bundles.
Suppose that $F$ is a complex vector bundle over a smooth manifold $X$.
A \emph{canonical holomorphic ideal} $\mathcal{I}_{0,F}$ for the zero section $X$ in $F$ is constructed as follows:
Suppose that  $U\subseteq X$ is an open set with a trivialization $\phi_U:F|_U\cong U\times \mathbb{C}^r$.
Let  $z_1,\ldots,z_r$ be the canonical holomorphic coordinates of $\mathbb{C}^r$.
For every $i$, $z_i$ can be viewed as a complex valued smooth function on $F|_U$ via the projection $U\times \mathbb{C}^r\rightarrow \mathbb{C}^r$.
Set $\mathcal{I}_U=\langle z_1, \ldots, z_r\rangle\subseteq \mathcal{C}^\infty_{F|_U}$.
On such two charts $U$ and $V$, the generators of $\mathcal{I}_U$ and the ones of $\mathcal{I}_V$ defined as above can represent each other via a translation matrix $\phi_V\circ \phi_U^{-1}$ of $F$ over $U\cap V$.
So $\mathcal{I}_U|_{U\cap V}= \mathcal{I}_V|_{U\cap V}$.
Hence the definition of $\mathcal{I}_U$ is independent of the choice of the trivialization $\phi_U$  and $\{\mathcal{I}_U\}$ define a ideal $\mathcal{I}_{0,F}\subseteq \mathcal{C}^\infty_{F}$ such that $\mathcal{I}_{0,F}|_{F|_U}=\mathcal{I}_U$.
Clearly, $\mathcal{I}_{0,F}$ is a holomorphic ideal for the zero section $X$.

\begin{prop}\label{ideal-vb}
Let $F$ be a complex vector bundle over a smooth manifold $X$.

$(1)$ For any open set $U\subseteq X$, $\mathcal{I}_{0,F}|_{F|_U}=\mathcal{I}_{0,F|_U}$.

$(2)$ If $F$ is a complex line bundle, $\mathcal{I}_{0,F}$ is a divisor on $F$.

$(3)$ Suppose that $Y$  is a smooth manifold and $pr_2:Y\times X\rightarrow X$ is the second projection.
Denote by $pr_2^{\prime}$ the second projection of  $pr_2^*F=Y\times F\rightarrow F$.
Then
\begin{equation}\label{pullback}
pr_2^{\prime-1}\mathcal{I}_{0,F}\cdot \mathcal{C}^\infty_{pr_2^*F}=\mathcal{I}_{0,pr_2^*F}.
\end{equation}

$(4)$ The blow-up of $\mathcal{I}_{0,F}$ in $F$ is the projection $\varphi:\mathcal{O}_{\mathbb{P}(F)}(-1)\rightarrow F$.
Moreover,  $\varphi^{-1}\mathcal{I}_{0,F}\cdot\mathcal{C}^\infty_{\mathcal{O}_{\mathbb{P}(F)}(-1)}=\mathcal{I}_{0,\mathcal{O}_{\mathbb{P}(F)}(-1)}$.
\end{prop}
\begin{proof}
Clearly, $(1)$ and $(2)$ hold by the definition of $\mathcal{I}_{0,F}$.
Let $U\subseteq X$ be any open set with a trivialization $\phi_U:F|_U\tilde{\rightarrow} U\times \mathbb{C}^r$.
There is the trivialization $id_Y\times\phi_U:(pr_2^*F)|_{Y\times U}\tilde{\rightarrow}  Y\times U\times \mathbb{C}^r$.
Denote by $pr_2^{\prime\prime}:U\times \mathbb{C}^r\rightarrow \mathbb{C}^r$ the second projection and by $pr_3^{\prime\prime\prime}:Y\times U\times \mathbb{C}^r\rightarrow \mathbb{C}^r$ the third projection.
Let $z_1$, \ldots, $z_r$ be the canonical holomorphic coordinates of $\mathbb{C}^r$.
Then $\mathcal{I}_{0,F}|_{F|_U}=\langle z_1, \ldots, z_r\rangle$ and $(\mathcal{I}_{0,pr_2^*F})|_{(pr_2^*F)|_{Y\times U}}=\langle z_1, \ldots, z_r\rangle$,  where $z_i$ is viewed as the complex valued smooth function on  $F|_U$ via $pr_2^{\prime\prime}$ and on $pr_2^*F|_{Y\times U}$ via $pr_3^{\prime\prime\prime}$ respectively for $1\leq i\leq r$.
Evidently, (\ref{pullback}) holds on $(pr_2^*F)|_{Y\times U}$, so does (\ref{pullback}) on $pr_2^*F$.
We proved $(3)$.
Obviously, $\langle z_1,\ldots,z_r\rangle\subseteq \mathcal{C}^\infty_{\mathbb{C}^r}$ is a holomorphic ideal for the original point $o$ in $\mathbb{C}^r$.
By \cite[Proposition 2.3.1]{v}, the blow-up of $\langle z_1,\ldots,z_r\rangle$ in $\mathbb{C}^r$ is just the complex blow-up $\varphi_o:\mathcal{O}_{\mathbb{C}P^{r-1}}(-1)\rightarrow \mathbb{C}^r$ of $\mathbb{C}^r$ along  $o$, which is naturally induced by the projection $\mathbb{C}P^{r-1}\times \mathbb{C}^r\rightarrow \mathbb{C}^r$.
As we know, $\mathcal{O}_{\mathbb{C}P^{r-1}}(-1)$ consists of $([w_1,\ldots,w_r],(z_1,\ldots,z_r))\in\mathbb{C}P^{r-1}\times\mathbb{C}^r$ such that $w_iz_j=w_jz_i$ for $1\leq i, j\leq r$.
Set $U_i=\{[w_1,\ldots,w_r]\in\mathbb{C}P^{r-1}|w_i\neq 0\}$ for $1\leq i\leq r$.
Then $([w_1,\ldots,w_r],(z_1,\ldots,z_r))\mapsto (u_1,\ldots,\widehat{u_i},\ldots,u_r,t)=(\frac{w_1}{w_i},\ldots,\widehat{\frac{w_i}{w_i}},\ldots,\frac{w_r}{w_i},z_i)$ gives a trivialization $\mathcal{O}_{\mathbb{C}P^{r-1}}(-1)|_{U_i}\tilde{\rightarrow}U_i\times \mathbb{C}$, where $(u_1,\ldots,\widehat{u_i},\ldots,u_r)$ is the coordinates of $U_i\cong \mathbb{C}^{r-1}$.
Then
\begin{displaymath}
\left(\varphi_o^{-1}\langle z_1,\ldots,z_r\rangle\cdot\mathcal{C}^\infty_{\mathcal{O}_{\mathbb{C}P^{r-1}}(-1)}\right)|_{\mathcal{O}_{\mathbb{C}P^{r-1}}(-1)|_{U_i}}=\langle u_1t,\ldots, \widehat{u_it},\ldots,u_rt,t\rangle =\langle t\rangle.
\end{displaymath}
By the definition,
\begin{equation}\label{taut}
\varphi_o^{-1}\langle z_1,\ldots,z_r\rangle\cdot\mathcal{C}^\infty_{\mathcal{O}_{\mathbb{C}P^{r-1}}(-1)}=\mathcal{I}_{0,\mathcal{O}_{\mathbb{C}P^{r-1}}(-1)}.
\end{equation}
By $(2)$,
\begin{equation}\label{trivial}
pr^{\prime\prime-1}_2\langle z_1,\ldots,z_r\rangle\cdot\mathcal{C}_{U\times \mathbb{C}^r}^\infty=\mathcal{I}_{0,U\times \mathbb{C}^r}.
\end{equation}
Hence $id_U\times \varphi_o:\mathcal{O}_{U\times\mathbb{C}P^{r-1}}(-1)=U\times \mathcal{O}_{\mathbb{C}P^{r-1}}(-1)\rightarrow U\times \mathbb{C}^r$ is the blow-up of $\mathcal{I}_{0,U\times \mathbb{C}^r}$ in $U\times \mathbb{C}^r$ by  Proposition \ref{blow-up fund} $(1)$ $(ii)$.
Denote by $pr_2^{\prime}:U\times \mathcal{O}_{\mathbb{C}P^{r-1}}(-1)\rightarrow \mathcal{O}_{\mathbb{C}P^{r-1}}(-1)$ the second projection.
Then
\begin{displaymath}
\begin{aligned}
&(id_U\times \varphi_o)^{-1}\mathcal{I}_{0,U\times \mathbb{C}^r}\cdot \mathcal{C}^\infty_{\mathcal{O}_{U\times\mathbb{C}P^{r-1}}(-1)}\\
=&(id_U\times \varphi_o)^{-1}\left(pr^{\prime\prime-1}_2\langle z_1,\ldots,z_r\rangle\cdot \mathcal{C}^\infty_{U\times \mathbb{C}^r}\right)\cdot \mathcal{C}^\infty_{\mathcal{O}_{U\times\mathbb{C}P^{r-1}}(-1)} \qquad\quad\mbox{ }(\mbox{by  (\ref{trivial})})\\
=&pr^{\prime-1}_2\left(\varphi^{-1}_o\langle z_1,\ldots,z_r\rangle\cdot \mathcal{C}^\infty_{\mathcal{O}_{\mathbb{C}P^{r-1}}(-1)}\right)\cdot \mathcal{C}^\infty_{\mathcal{O}_{U\times\mathbb{C}P^{r-1}}(-1)} \qquad\qquad\quad(\mbox{by  (\ref{composition})})\\
=&pr^{\prime-1}_2\mathcal{I}_{0,\mathcal{O}_{\mathbb{C}P^{r-1}}(-1)}\cdot \mathcal{C}^\infty_{\mathcal{O}_{U\times\mathbb{C}P^{r-1}}(-1)} \qquad\qquad\qquad\qquad\qquad\qquad\quad(\mbox{by  (\ref{taut})})\\
=&\mathcal{I}_{0,\mathcal{O}_{U\times\mathbb{C}P^{r-1}}(-1)} \qquad\qquad\qquad\qquad\qquad\qquad\qquad\qquad\qquad\qquad\mbox{ }\mbox{ }\mbox{ }(\mbox{by (\ref{pullback})}).
\end{aligned}
\end{displaymath}
Via the diffeomorphism $\phi_U$,  $\varphi|_{\mathcal{O}_{\mathbb{P}(F|_U)}(-1)}:\mathcal{O}_{\mathbb{P}(F|_U)}(-1)\rightarrow F|_U$ is the blow-up of $\mathcal{I}_{0,F|_U}=\mathcal{I}_{0,F}|_{F|_U}$ in $F|_U$ with $(\varphi|_{\mathcal{O}_{\mathbb{P}(F|_U)}(-1)})^{-1}\mathcal{I}_{0,F|_U}\cdot\mathcal{C}^\infty_{\mathcal{O}_{\mathbb{P}(F|_U)}(-1)}=\mathcal{I}_{0,\mathcal{O}_{\mathbb{P}(F|_U)}(-1)}$.
By Proposition \ref{blow-up fund} $(2)$, we get $(4)$.
\end{proof}

If $\mathcal{I}_Y$ is a holomorphic ideal for $Y$, then the complexification  of the conormal bundle of $Y$ in $X$ has the decomposition $N_{Y/X}^*\otimes_{\mathbb{R}} \mathbb{C}=N_{Y/X}^{*1,0}\oplus N_{Y/X}^{*0,1}$, where $N_{Y/X,y}^{*1,0}=\{(df)_y|\mbox{ }f\in \mathcal{I}_{Y,y}\}$ and $N_{Y/X,y}^{*0,1}=\{(d\bar{f})_y|\mbox{ }f\in \mathcal{I}_{Y,y}\}$ for any $y\in Y$.
Then $N_{Y/X}^{*1,0}$ gives a complex vector bundle structure on  $N_{Y/X}^*$ via the isomorphism $N_{Y/X}^*\cong N_{Y/X}^{*1,0}$ of smooth vector bundles.
Hence the normal bundle $N_{Y/X}$ has a complex vector bundle structure,
which is said to be the \emph{complex vector bundle structure of $N_{Y/X}$ induced by $\mathcal{I}_Y$}.

For complex vector bundles, Lemma \ref{NB-vb} is strengthen as follows.
\begin{lem}\label{NB-cvb}
Let $F$ be a complex vector bundle over a smooth manifold $X$.
Equip $N_{X/F}$ with the complex vector bundle structure induced by $\mathcal{I}_{0,F}$.
Then there exists an isomorphism $N_{X/F}\tilde{\rightarrow} F$ of complex vector bundles over $X$.
\end{lem}
\begin{proof}
Let $\phi_U:F|_U\tilde{\rightarrow} U\times \mathbb{C}^r$ be  a trivialization of $F$ over an open set $U\subseteq X$ given by  the sections $e_1$, \ldots, $e_r$ of $F|_U$, i.e., $\phi_U(\sum\limits_{i=1}^r z_i\cdot e_i(x))=(x,z_1, \ldots, z_r)$.
For the complex vector bundle structure induced by $\mathcal{I}_{0,F}$, $N^*_{X/F}|_U=\{df|f\in\langle z_1,\ldots, z_r\rangle\}=\textrm{span}_{\mathbb{C}}\{dz_1,\ldots,dz_r\}$, where $z_1$, \ldots, $z_r$  are viewed as complex valued smooth functions on $F|_U$.
Define $\Psi_U:N_{X/F}|_U\rightarrow F|_U$ as $\sum\limits_{i=1}^r a_i\cdot \frac{\partial}{\partial z_i}$ modulo $\textrm{T}X|_U$ $\mapsto \sum\limits_{i=1}^r a_i\cdot e_i(x)$, which is an isomorphism of complex vector bundles over $U$.
Assume that $f_1$, \ldots, $f_r$ give a trivialization $\Psi_V$ of $F$ over  an open set $V\subseteq X$ and  $e_i=\sum\limits_{j=1}^rg_{ij}\cdot f_j$ with $g_{ij}\in \mathcal{C}^\infty(U\cap V)$ for $1\leq i,j\leq r$ on $U\cap V$.
Then
\begin{displaymath}
(x,w_1,\ldots,w_r)=\phi_V\circ \phi_U^{-1}(x,z_1,\ldots,z_r)=(x,\sum_{i=1}^rz_ig_{i1}(x),\ldots,\sum_{i=1}^rz_ig_{ir}(x)).
\end{displaymath}
So $\sum\limits_{i=1}^r a_i\cdot \frac{\partial}{\partial z_i}=\sum\limits_{i,j=1}^r a_ig_{ij}\cdot \frac{\partial}{\partial w_j}$ modulo $\textrm{T}X|_{U\cap V}$, by which we easily check that $\{\Psi_U\}$  give a global isomorphism $N_{X/F}\tilde{\rightarrow} F$ of complex vector bundles over $X$.
\end{proof}

\begin{ex}
$(1)$  Let  $Y$ be a complex submanifold of a complex manifold $X$.
Denote by $\mathcal{I}_Y\subseteq \mathcal{C}_X^\infty$ the ideal generated by the holomorphic functions which vanish on $Y$.
Clearly, $\mathcal{I}_Y$ is a holomorphic ideal for $Y$.
The blow-up of $\mathcal{I}_Y$ in $X$ is just the classical complex blow-up of $X$ along $Y$, see \cite[Proposition 2.3.1]{v}.

$(2)$ Let $Y$ be a generalized Poisson submanifold of a generalized complex manifold $X$. There exists an canonical holomorphic ideal $\mathcal{I}_Y$ for $Y$ (\cite[Proposition 3.12]{BCv}) and we can get the  blow-up $\widetilde{X}$ of $\mathcal{I}_Y$ in $X$.
Under suitable conditions, $\widetilde{X}$ has a generalized complex structure (\cite[Theorem 3.16]{BCv}).

$(3)$ Suppose that $Y$ is a submanifold of $X$ such that normal bundle $N_{Y/X}$ has a complex vector bundle structure and we choose a tubular embedding $l:N_{Y/X}\rightarrow X$.
Let $\mathcal{I}_{Y,W}\subseteq \mathcal{C}^\infty_W$ be the ideal on $W=l(N_{Y/X})$ satisfying that  $\mathcal{I}_{0,N_{Y/X}}=l^{\prime-1}\mathcal{I}_{Y,W}$, where $l^{\prime}:N_{Y/X}\rightarrow W$ is the diffeomorphism induced by $l$.
Since $\mathcal{I}_{Y,W}|_{W-Y}=\mathcal{C}^\infty_{W-Y}$, we obtain a sheaf $\mathcal{I}_Y\subseteq \mathcal{C}^\infty_X$ by gluing $\mathcal{I}_{Y,W}$ and $\mathcal{C}^\infty_{X-Y}$.
Clearly, $\mathcal{I}_Y$ is a holomorphic ideal for $Y$ satisfying  $l^{-1}\mathcal{I}_Y=\mathcal{I}_{0,N_{Y/X}}$.
Then the complex vector bundle structure of $N_{Y/X}$ induced by $\mathcal{I}_Y$ coincides with the one of $N_{Y/N_{Y/X}}$ induced by $\mathcal{I}_{0,N_{Y/X}}$ via the diffeomorphism $l^{\prime}$, hence is just the given one by Lemma \ref{NB-cvb}.
Moreover, if $Y$ is compact, such holomorphic ideal is unique up to (non-canonical) diffeomorphism by \cite[Corollary 3.3]{BCv}.
In such way, the blow-up of a compact generalized Poisson transversal $Y$ in  a generalized complex manifold $X$ is defined well, which carries a generalized complex structure, see \cite[Theorem 3.34]{BCv}.
\end{ex}

\begin{prop}\label{real-generalized}
The blow-up of a holomorphic ideal for $Y$ in $X$ is a generalized blow-up of $X$ along $Y$.
\end{prop}
\begin{proof}
Let $\mathcal{I}_Y$ be a  holomorphic ideal for $Y$ and let $\pi:\widetilde{X}\rightarrow X$ be the blow-up of $\mathcal{I}_Y$ in $X$.
Then $N=N_{Y/X}$ has the  complex  vector bundle structure induced by $\mathcal{I}_Y$.
By \cite[Proposition 3.2]{BCv}, there is a tubular embedding $l:N\rightarrow X$ such that $l^{-1}\mathcal{I}_Y=\mathcal{I}_{0,N}$.
Set $W=l(N)$.
We use the notations in the diagram (\ref{blow-up-vb}).
Let $\pi^{\prime}:\widetilde{X}^{\prime}=(X-Y)\cup_{\phi}\mathcal{O}_{\mathbb{P}(N)}(-1)\rightarrow X$ be the blow-up of $X$ along $Y$ associated to the complex vector bundle structure of $N$ induced by $\mathcal{I}_Y$ and the embedding $l$.
Denote by $[x]$ the class in $\widetilde{X}^{\prime}$ of $x\in (X-Y)\sqcup \mathcal{O}_{\mathbb{P}(N)}(-1)$.
Then $\pi^{\prime}([x])=x$ for $x\in X-Y$ and $\pi^{\prime}([x])=l\circ\varphi(x)$ for $x\in \mathcal{O}_{\mathbb{P}(N)}(-1)$.
Let $l^{\prime}:N\rightarrow W$ be the diffeomorphism induced by $l$.
Since $l^{\prime-1}(\mathcal{I}_Y|_W)=\mathcal{I}_{0,N}$, $l^{\prime}\circ\varphi:\mathcal{O}_{\mathbb{P}(N)}(-1)\rightarrow W$ is  the blow-up of $\mathcal{I}_Y|_W$ in $W$ with $(l^{\prime}\circ\varphi)^{-1}(\mathcal{I}_{Y}|_W)\cdot\mathcal{C}^\infty_{\mathcal{O}_{\mathbb{P}(N)}(-1)}=\mathcal{I}_{0,\mathcal{O}_{\mathbb{P}(N)}(-1)}$ by Proposition \ref{ideal-vb} $(4)$.
By Proposition \ref{blow-up fund} $(1)$ $(i)$, $\pi|_{\pi^{-1}(W)}:\pi^{-1}(W)\rightarrow W$ is the blow-up of $\mathcal{I}_Y|_W$ in $W$.
By Theorem \ref{exist unique real blow-up}, there exists a diffeomorphism $a:\pi^{-1}(W)\rightarrow \mathcal{O}_{\mathbb{P}(N)}(-1)$ such that $(l^{\prime}\circ\varphi)\circ  a=\pi|_{\pi^{-1}(W)}$.
Define $f:\widetilde{X}\rightarrow \widetilde{X}^{\prime}$ as $f(\tilde{x})=[\pi(\tilde{x})]$ for $\tilde{x}\in \pi^{-1}(X-Y)$ and $f(\tilde{x})=[a(\tilde{x})]$ for $\tilde{x}\in \pi^{-1}(W)$.
Obviously, $f$ is defined well and $\pi^{\prime}\circ f=\pi$.
The restrictions  $f:f^{-1}\left(\pi^{\prime-1}(X-Y)\right)=\pi^{-1}(X-Y)\rightarrow \pi^{\prime-1}(X-Y)$ and $f:f^{-1}\left(\pi^{\prime-1}(W)\right)=\pi^{-1}(W)\rightarrow \pi^{\prime-1}(W)$ are induced by diffeomorphisms $\pi^{-1}(X-Y)\tilde{\rightarrow} X-Y$ and $a$ respectively, hence they are both diffeomorphisms.
Then $f$ is a diffeomorphism over $X$.
So $\pi:\widetilde{X}\rightarrow X$ is a generalized blow-up.
\end{proof}

\subsubsection{A proof of Theorem \ref{1.3}}
By Proposition \ref{blow-up-properties} $(6)$, $\widetilde{X}$ and $E$ are oriented,  so the pushforward $i_{E*}$ is defined well.
Set $U=X-Y$ and $\widetilde{U}=\widetilde{X}-E$.
By Proposition \ref{blow-up-properties} $(1)$, $\pi|_{\widetilde{U}}:\widetilde{U}\rightarrow U$ is a diffeomorphism.
By Proposition \ref{comm-long}, there is a commutative diagram of long exact sequences
\begin{displaymath}
\small{\xymatrix{
 \cdots H_{\Phi|_U}^k(U,\mathcal{L})\ar[d]^{(\pi|_{\widetilde{U}})^*}_{\cong} \ar[r]& H_{\Phi}^k(X,\mathcal{L}) \ar[d]^{\pi^*} \ar[r]^{i_Y^*}& H_{\Phi|_Y}^k(Y,\mathcal{L}|_Y) \ar[d]^{(\pi\mid_E)^*}\ar[r]& H_{\Phi|_U}^{k+1}(U,\mathcal{L})\ar[d]^{(\pi|_{\widetilde{U}})^*}_{\cong}\cdots\\
 \cdots H_{(\pi^{-1}\Phi)|_{\widetilde{U}}}^k(\widetilde{U},\pi^{-1}\mathcal{L})       \ar[r]& H_{\pi^{-1}\Phi}^k(\widetilde{X},\pi^{-1}\mathcal{L})     \ar[r]^{i_E^*\quad} & H_{(\pi^{-1}\Phi)|_E}^k(E,(\pi^{-1}\mathcal{L})|_E)     \ar[r]& H_{(\pi^{-1}\Phi)|_{\widetilde{U}}}^{k+1}(\widetilde{U},\pi^{-1}\mathcal{L})  \cdots. }}
\end{displaymath}
By (\ref{pro-formula2-dR}), $\pi_*\pi^*=id$, so $\pi^*$ is injective.
By the snake lemma, $i_E^*$ induces an isomorphism
$\textrm{coker}\pi^*\tilde{\rightarrow}\textrm{coker}(\pi|_E)^*$.
We get a commutative diagram of short exact sequences
\begin{equation}\label{commutative2}
\xymatrix{
 0\ar[r]&H_{\Phi}^k(X,\mathcal{L})\ar[d]^{i_Y^*} \ar[r]^{\pi^*\quad}& H_{\pi^{-1}\Phi}^k(\widetilde{X},\pi^{-1}\mathcal{L})\ar[d]^{i_E^*} \ar[r]& \textrm{coker}\pi^* \ar[d]^{\cong}\ar[r]& 0\\
 0\ar[r]&H_{\Phi|_Y}^k(Y,i_Y^{-1}\mathcal{L})       \ar[r]^{(\pi|_E)^*\qquad}& H_{(\pi^{-1}\Phi)|_E}^k(E,i_E^{-1}\pi^{-1}\mathcal{L})   \ar[r]^{} &  \textrm{coker} (\pi|_E)^*    \ar[r]& 0. }
\end{equation}
By Proposition \ref{blow-up-properties} $(4)$, $c_1(N_{E/\widetilde{X}})=h$, hence $i_E^*i_{E*}(\bullet)=h\cup\bullet$ on $H_{(\pi^{-1}\Phi)|_E}^\bullet(E,i_E^{-1}\pi^{-1}\mathcal{L})$ by Proposition \ref{key}.
With the similar proofs of Lemma \ref{special} and Proposition \ref{relation}, we easily show that (\ref{b-u-m1b}) and (\ref{b-u-m3b}) are inverse isomorphisms.
We prove Theorem \ref{1.3}.

\begin{prop}\label{mod2}
Let $\pi:\widetilde{X}\rightarrow X$ be a proper, surjective, preserving orientations, smooth map of  oriented smooth manifolds with degree one and let $\mathcal{L}$ be a local system of $\mathbf{k}$-modules of finite rank on $X$.
Assume that the set $E$ of critical points of $\pi$ and $Y=\pi(E)$ are \emph{(}not necessarily orientable\emph{)} smooth manifolds.
Then there exists an short exact sequence
\begin{displaymath}
\begin{aligned}
\small{\xymatrix{
0\ar[r] &H_{\Phi}^k(X,\mathcal{L})\ar[r]^{(\pi^*,i_Y^*)\quad\qquad\qquad} &H_{\pi^{-1}\Phi}^k(\widetilde{X},\pi^{-1}\mathcal{L})\oplus H_{\Phi|_Y}^k(Y,\mathcal{L}|_E)\ar[r]^{\qquad i_E^*-\pi|_E^*} & H_{(\pi^{-1}\Phi)|_E}^k(E,i_E^{-1}\pi^{-1}\mathcal{L})\ar[r]&0
}}
\end{aligned}
\end{displaymath}
for any $k$, where $i_Y:Y\rightarrow X$ and  $i_E:E\rightarrow X$ are inclusions.
\end{prop}
\begin{proof}
With the same proof,  (\ref{commutative2}) still holds in such case, which  easily  imply the conclusion by diagram chasing.
\end{proof}

As an application of Theorem \ref{1.3}, we have the \emph{excess intersection formula}.
\begin{cor}\label{excess intersection}
Under the assumptions in Theorem \ref{1.3}, let  $Q$ be the quotient bundle $(\pi|_E)^*N_{Y/X}/O_E(-1)$ over $E$.
Then
\begin{equation}
\xymatrix{
  H_{\Phi|_Y}^{\bullet}(Y,\mathcal{L}|_Y)\ar[d]_{c_{r-1}(Q)\cup(\pi|_E)^*(\bullet)}  \ar[r]^{\quad i_{Y*}\quad}  & H_{\Phi}^{\bullet+2r}(X,\mathcal{L}) \ar[d]^{\pi^*} \\
  H_{(\pi|_E)^{-1}(\Phi|_Y)}^{\bullet+2r-2}(E,(\pi|_E)^{-1}(\mathcal{L}|_Y)) \ar[r]^{\qquad\qquad i_{E*}\quad}  & H_{\pi^{-1}\Phi}^{\bullet+2r}(Y,\pi^{-1}\mathcal{L}).}
\end{equation}
is a commutative diagram.
\end{cor}
\begin{proof}
The total Chern classes satisfy  $c(O_E(-1))\cup c(Q)=(\pi|_E)^*c(N_{Y/X})$, which implies that
\begin{equation}\label{Chern class 1}
h\cup c_{r-1}(Q)=(\pi|_E)^*c_r(N_{Y/X}),
\end{equation}
\begin{equation}\label{Chern class 2}
c_{r-1}(Q)=\sum\limits_{i=0}^{r-1}(-1)^ih^i\cup(\pi|_E)^*c_{r-1-i}(N_{Y/X}).
\end{equation}
Notice that $(\pi|_{E})_*h^i=0$  for $0\leq i\leq r-2$ and  $(\pi|_{E})_*h^{r-1}=(-1)^{r-1}$, see Sect. 4.3.
By (\ref{pro-formula2-dR}),
\begin{equation}\label{Chern class 3}
(\pi|_E)_*(c_{r-1}(Q))=1.
\end{equation}
Fix an integer $k$.
For $\sigma\in H_{\Phi|_Y}^{\bullet}(Y,\mathcal{L}|_Y)$, set $\alpha=i_{E*}(c_{r-1}(Q)\cup (\pi|_E)^*\sigma)$.
By Proposition \ref{key} and (\ref{Chern class 1}),
\begin{equation}\label{self-1}
i_E^*\alpha=h\cup c_{r-1}(Q)\cup (\pi|_E)^*\sigma=(\pi|_E)^*(c_r(N_{Y/X})\cup\sigma).
\end{equation}
By Theorem \ref{1.3}, $\alpha=\pi^*\beta+\sum\limits_{i=1}^{r-1}i_{E*} (h^{i-1}\cup (\pi|_E)^*\gamma_i)$
for unique $\beta\in H_{\Phi}^{k+2r}(X,\mathcal{L})$ and $\gamma_i\in H_{\Phi|_Y}^{k+2r-2i}(Y,\mathcal{L}|_Y)$.
By Proposition \ref{key},
\begin{equation}\label{self-2}
i_{E}^*\alpha=(\pi|_E)^*i_Y^*\beta+\sum_{i=1}^{r-1}h^i\cup (\pi|_E)^*\gamma_i.
\end{equation}
Comparing (\ref{self-1}) and (\ref{self-2}), $\gamma_i=0$ for $1\leq i\leq r-1$ by Proposition \ref{3}.
Then  $\alpha=\pi^*\beta$.
So
\begin{displaymath}
\begin{aligned}
\beta=&\pi_*\alpha\qquad\qquad\qquad\qquad\qquad\qquad(\mbox{by  (\ref{pro-formula2-dR})})\\
=&i_{Y*}(\pi|_E)_*(c_{r-1}(Q)\cup (\pi|_E)^*\sigma) \mbox{ }\mbox{ }\mbox{ }(\mbox{by the definition of $\alpha$})\\
=&i_{Y*}\sigma  \qquad\qquad\qquad\qquad\qquad\quad\mbox{ }\mbox{ }(\mbox{by  (\ref{pro-formula2-dR}) and (\ref{Chern class 3})}).
\end{aligned}
\end{displaymath}
Hence, $\alpha=\pi^*i_{Y*}\sigma$.
We complete the proof.
\end{proof}


\end{document}